\documentclass[11pt,reqno]{amsart}

\usepackage[a4paper,margin=3cm]{geometry}
\usepackage[dvipsnames]{xcolor}
\usepackage[english]{babel}
\usepackage[babel]{microtype}

\usepackage{amsmath}
\usepackage{amssymb}
\usepackage{amsthm}
\usepackage[bb=boondox]{mathalpha}

\theoremstyle{plain}
\newtheorem{theorem}{Theorem}[section]
\newtheorem{proposition}[theorem]{Proposition}
\newtheorem*{proposition*}{Proposition}
\newtheorem*{theorem*}{Theorem}
\newtheorem{lemma}[theorem]{Lemma}
\newtheorem*{lemma*}{Lemma}
\newtheorem{corollary}[theorem]{Corollary}

\theoremstyle{definition}

\newtheorem{remark}{Remark}
\newtheorem*{remark*}{Remark}

\usepackage{mathrsfs}
\usepackage{stmaryrd}
\usepackage{tikz}
\usetikzlibrary{arrows.meta,positioning,calc,decorations.pathreplacing,%
                patterns,fit,backgrounds,shapes.geometric,shapes.misc}
\usepackage{pgfplots}
\pgfplotsset{compat=1.18}
\colorlet{cladecol}{MidnightBlue}         
\colorlet{hicol}{BrickRed}                
\colorlet{ghostcol}{black!22}             
\colorlet{dustcol}{black!45}              
\colorlet{masscol}{ForestGreen!70!black}  
\colorlet{oldrevisioncol}{black!45}       
\colorlet{newrevisioncol}{BrickRed!85!black} 
\usepackage{enumerate}
\usepackage{mathtools}
\usepackage[full]{textcomp}
\mathtoolsset{showonlyrefs}
\numberwithin{equation}{section}

\usepackage[T1]{fontenc}
\usepackage{lmodern}
\usepackage[sb]{libertinus}
\usepackage[scaled=0.95]{cabin}
\usepackage[utf8]{inputenc}
\DeclareMathSizes{10.95}{10.4}{7.6}{5.7}    
\DeclareMathSizes{10}{9.5}{6.65}{4.75}      
\DeclareMathSizes{9}{8.55}{6.65}{4.75}      
\DeclareMathSizes{12}{11.4}{7.6}{5.7}       
\DeclareMathSizes{14.4}{13.68}{9.5}{6.65}   

\usepackage[symbol]{footmisc}
\usepackage{comment}
\usepackage{appendix}
\usepackage{graphicx}
\usepackage{float}

\usepackage[
  breaklinks=true,
  bookmarksdepth=2,
  pagebackref=true,
]{hyperref}
\hypersetup{
  colorlinks=true,
  pdfpagemode=UseNone,
  citecolor=ForestGreen,
  linkcolor=MidnightBlue,
  urlcolor=RoyalBlue,
  pdfstartview=FitW
}

\def\N{\mathbb{N}}

\def\R{\mathbb{R}}
\def\E{\mathbf{E}}
\renewcommand{\P}{\mathbf{P}}

\newcommand{\ind}[1]{\mathbf{1}_{\{#1\}}}
\newcommand{\res}[1]{\!\restriction_{#1}}

 \DeclareMathOperator*{\supp}{supp}
\newcommand*{\dif}{\ensuremath{\mathop{}\!\mathrm{d}}}
\renewcommand{\bar}[1]{\mkern 1.5mu\overline{\mkern-1.5mu#1\mkern-1.5mu}\mkern 1.5mu}
\renewcommand{\hat}[1]{\widehat{#1}}

\newcommand{\Poi}{\operatorname{Poi}}
\newcommand{\PPP}{\operatorname{PPP}}

\newcommand{\DPPP}{\operatorname{DPPP}}
\newcommand{\CTCS}{\operatorname{CTCS}}
\newcommand{\DTCS}{\operatorname{DTCS}}
\newcommand{\Mloc}{\mathcal M_{\mathrm p}}

\definecolor{revisionR2}{rgb}{0.612,0.000,0.000}

\title[A freezing transition in fragmentation trees]
{Extremal separation times and clade-count dynamics in fragmentation trees: a freezing transition}
\author{Heng Ma}
\address[Heng Ma]
{Faculty of Data and Decision Sciences, Technion - Israel's Institute of
Technology, Haifa, 32000, Israel.}
\email{hengmamath(at)gmail(dot)com}
\urladdr{\url{https://hengmamath.github.io}}
\date{\today}

\begin{document}

\begin{abstract}
Motivated by the continuous-time critical beta-splitting tree, we investigate chronological 
trees recording the successive block splits and their times in homogeneous
fragmentation processes restricted to $n$ labels. The separation time of a subset is the first time its labels cease to lie
in a common block. 
For each fixed integer $q\geq2$,  we study   extremal separation times of $q$-element subsets via their associated point process.
 
Under mild conditions,  as $n\to\infty$, the last such  time has a randomly shifted Gumbel limit after deterministic
centering. Above a  threshold
$\theta_*$,  the  coefficients of  both the leading   $\log n$ term and the $\log\log n$ correction in the centering become independent of $q$, as does the limiting law up to deterministic translation.
The extremal process converges  to a randomly shifted
Poisson point process, without and with decorations,   for $q \le \theta_*$ and $q>\theta_*$ respectively. 
Finally, in the same centered time window, the number of size-$q$ blocks (clades) converges  to a pure-death process for $q\leq\theta_*$, and to a process whose trajectories are non-monotone with positive probability for $q>\theta_*$.
\end{abstract}

   \maketitle

\section{Introduction and main results}
\label{sec:ctcs-question}
 
\subsection{Continuous-time critical beta-splitting trees: a first result} Introduced by Aldous~\cite{AldousCladograms}, the critical beta-splitting tree is a toy model for phylogenetic trees designed to capture the uneven splits observed in real-world data.
Its continuous-time version, denoted $\CTCS(n)$, is constructed recursively by splitting clades over time as follows.  

At time $t=0$, there is a single clade $[n]:=\{1,\dots,n\}$  of all
$n$ labels. A clade of size $m\geq2$ waits  an exponential time of rate
$
  h_{m-1}:=\sum_{j=1}^{m-1}\frac{1}{j} 
  $
and then splits into a left and a right sub-clade, of sizes $i$ and $m-i$
respectively, with probability
\begin{equation}\label{eq:intro-ctcs-rule}
  q(m,i)=\frac1{2h_{m-1}}\Bigl(\frac1i+\frac1{m-i}\Bigr),
  \qquad 1\leq i\leq m-1.
\end{equation}  
The left part is chosen uniformly among the $\binom{m}{i}$ subsets of size $i$, and the two children evolve independently by the same rule.  
Since singletons never split, the process ends after exactly $n-1$ splits, when all clades are singletons. 

The family $\mathcal{C}_n$ of all clades that appear during the evolution form a natural tree structure
induced by the inclusion order: the root is $[n]$, its leaves are the
singletons, and the parent of a non-root clade $\mathsf{C}$, denoted by $\mathrm{par}(\mathsf{C})$, is its minimal strict
superset in $\mathcal{C}_n$. Moreover, the continuous time setting gives 
each clade $\mathsf C $  chronological marks.  Let
$b(\mathsf C)$ be its birth time, with $b([n])=0$.  If
$|\mathsf C|\geq2$, write $s(\mathsf C)$ for its splitting time.    At time $s(\mathsf C)$ it disappears and its two children are born and
become present.     
Represent each non-singleton clade $\mathsf{C}$ by the interval $[b(\mathsf{C}), s(\mathsf{C}))$ and each singleton $\{i\}$ by $\{ b(\{i\})\}$,  connect it to their parent at level
$b(\mathsf C)=s(\operatorname{par}(\mathsf C))$. 
This yields the chronological tree $\CTCS(n)$, adapting the chronological tree
definition in \cite[Section 2.2]{LambertContour}. 
Figure~\ref{fig:ctcs-tree} (a) illustrates this representation for $n=9$.

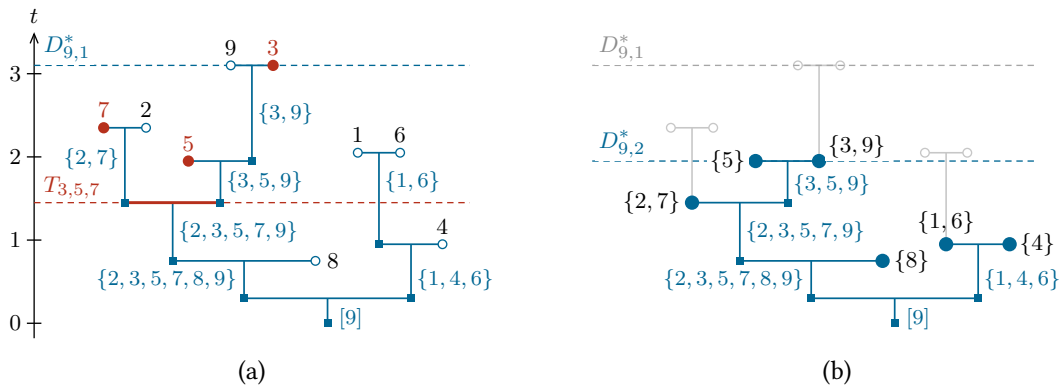
\begin{figure}[b]
\centering
\begin{tikzpicture}[
  x=0.8cm, y=1.1cm,
  font=\footnotesize,
  br/.style={line width=.65pt,color=cladecol},          
  hbr/.style={line width=1.1pt,color=hicol},            
  gbr/.style={line width=.6pt,color=ghostcol},          
  cl/.style={rectangle,fill=cladecol,inner sep=1.35pt}, 
  lf/.style={circle,draw=cladecol,fill=white,line width=.5pt,inner sep=1.15pt},
  hlf/.style={circle,fill=hicol,inner sep=1.5pt},       
  term/.style={circle,fill=cladecol,inner sep=1.9pt},   
  lev/.style={dashed,dash pattern=on 2.4pt off 1.8pt,line width=.5pt},
  lab/.style={inner sep=1.2pt},
  clab/.style={inner sep=1.1pt,font=\scriptsize,text=cladecol,fill=white},
]

\begin{scope}
  \draw[line width=.5pt,-{Straight Barb[length=3.6pt,width=3pt]}]
        (-1.15,-0.18) -- (-1.15,3.5);
  \foreach \t in {0,1,2,3}{
    \draw[line width=.5pt] (-1.27,\t) -- (-1.03,\t);
    \node[lab,anchor=east] at (-1.30,\t) {$\t$};}
  \node[lab,anchor=south] at (-1.15,3.56) {$t$};

  \draw[lev,color=hicol]  (-1.15,1.45) -- (6.05,1.45);
  \node[lab,anchor=south west,text=hicol,fill=white] at (-1.05,1.45) {$T_{3,5,7}$};
  \draw[lev,color=cladecol] (-1.15,3.10) -- (6.05,3.10);
  \node[lab,anchor=south west,text=cladecol,fill=white] at (-1.05,3.10) {$D^{*}_{9,1}$};

  \draw[br] (3.70,0) -- (3.70,0.30);                    
  \draw[br] (2.32,0.30) -- (5.08,0.30);
  \draw[br] (2.32,0.30) -- (2.32,0.75);                 
  \draw[br] (1.14,0.75) -- (3.50,0.75);
  \draw[br] (5.08,0.30) -- (5.08,0.95);                 
  \draw[br] (4.55,0.95) -- (5.60,0.95);
  \draw[br] (1.14,0.75) -- (1.14,1.45);                 
  \draw[hbr] (0.35,1.45) -- (1.93,1.45);                
  \draw[br] (4.55,0.95) -- (4.55,2.05);                 
  \draw[br] (4.20,2.05) -- (4.90,2.05);
  \draw[br] (0.35,1.45) -- (0.35,2.35);                 
  \draw[br] (0.00,2.35) -- (0.70,2.35);
  \draw[br] (1.93,1.45) -- (1.93,1.95);                 
  \draw[br] (1.40,1.95) -- (2.45,1.95);
  \draw[br] (2.45,1.95) -- (2.45,3.10);                 
  \draw[br] (2.10,3.10) -- (2.80,3.10);

  \foreach \p in {(3.70,0),(2.32,0.30),(5.08,0.30),(1.14,0.75),(4.55,0.95),%
                  (0.35,1.45),(1.93,1.45),(2.45,1.95)} {\node[cl] at \p {};}

  \node[hlf] at (0.00,2.35) {}; \node[lab,anchor=south,yshift=2.5pt,color=hicol] at (0.00,2.35) {$7$};
  \node[lf]  at (0.70,2.35) {}; \node[lab,anchor=south,yshift=2.5pt] at (0.70,2.35) {$2$};
  \node[hlf] at (1.40,1.95) {}; \node[lab,anchor=south,yshift=2.5pt,color=hicol] at (1.40,1.95) {$5$};
  \node[lf]  at (2.10,3.10) {}; \node[lab,anchor=south,yshift=2.5pt] at (2.10,3.10) {$9$};
  \node[hlf] at (2.80,3.10) {}; \node[lab,anchor=south,yshift=2.5pt,color=hicol] at (2.80,3.10) {$3$};
  \node[lf]  at (3.50,0.75) {}; \node[lab,anchor=west,xshift=3pt] at (3.50,0.75) {$8$};
  \node[lf]  at (4.20,2.05) {}; \node[lab,anchor=south,yshift=2.5pt] at (4.20,2.05) {$1$};
  \node[lf]  at (4.90,2.05) {}; \node[lab,anchor=south,yshift=2.5pt] at (4.90,2.05) {$6$};
  \node[lf]  at (5.60,0.95) {}; \node[lab,anchor=south,yshift=2.5pt] at (5.60,0.95) {$4$};

  \node[clab,anchor=west,xshift=3pt]     at (3.70,0.06) {$[9]$};                
  \node[clab,anchor=east]                at (2.26,0.52) {$\{2,3,5,7,8,9\}$};    
  \node[clab,anchor=west]                at (5.12,0.52) {$\{1,4,6\}$};          
  \node[clab,anchor=west]                at (1.18,1.10) {$\{2,3,5,7,9\}$};      
  \node[clab,anchor=west]                at (4.59,1.70) {$\{1,6\}$};            
  \node[clab,anchor=east]                at (0.31,1.98) {$\{2,7\}$};            
  \node[clab,anchor=west]                at (1.97,1.70) {$\{3,5,9\}$};          
  \node[clab,anchor=west]                at (2.49,2.52) {$\{3,9\}$};            

  \node[anchor=north,font=\small] at (2.45,-0.32) {(a)};
\end{scope}

\begin{scope}[xshift=7.5cm]
  \draw[gbr] (4.55,0.95) -- (4.55,2.05);
  \draw[gbr] (4.20,2.05) -- (4.90,2.05);
  \draw[gbr] (0.35,1.45) -- (0.35,2.35);
  \draw[gbr] (0.00,2.35) -- (0.70,2.35);
  \draw[gbr] (2.45,1.95) -- (2.45,3.10);
  \draw[gbr] (2.10,3.10) -- (2.80,3.10);
  \foreach \p in {(0.00,2.35),(0.70,2.35),(2.10,3.10),(2.80,3.10),(4.20,2.05),(4.90,2.05)}
     {\node[circle,draw=ghostcol,fill=white,line width=.5pt,inner sep=1.15pt] at \p {};}

  \draw[lev,color=cladecol] (-1.30,1.95) -- (6.05,1.95);
  \node[lab,anchor=south west,text=cladecol,fill=white] at (-1.25,1.95) {$D^{*}_{9,2}$};
  \draw[lev,color=black!35] (-1.30,3.10) -- (6.05,3.10);
  \node[lab,anchor=south west,text=black!45,fill=white] at (-1.25,3.10) {$D^{*}_{9,1}$};

  \draw[br] (3.70,0) -- (3.70,0.30);
  \draw[br] (2.32,0.30) -- (5.08,0.30);
  \draw[br] (2.32,0.30) -- (2.32,0.75);
  \draw[br] (1.14,0.75) -- (3.50,0.75);
  \draw[br] (5.08,0.30) -- (5.08,0.95);
  \draw[br] (4.55,0.95) -- (5.60,0.95);
  \draw[br] (1.14,0.75) -- (1.14,1.45);
  \draw[br] (0.35,1.45) -- (1.93,1.45);
  \draw[br] (1.93,1.45) -- (1.93,1.95);
  \draw[br] (1.40,1.95) -- (2.45,1.95);

  \foreach \p in {(3.70,0),(2.32,0.30),(5.08,0.30),(1.14,0.75),(1.93,1.45)}
     {\node[cl] at \p {};}

  \node[clab,anchor=west,xshift=3pt]     at (3.70,0.06) {$[9]$};
  \node[clab,anchor=east]                at (2.26,0.52) {$\{2,3,5,7,8,9\}$};
  \node[clab,anchor=west]                at (5.12,0.52) {$\{1,4,6\}$};
  \node[clab,anchor=west]                at (1.18,1.10) {$\{2,3,5,7,9\}$};
  \node[clab,anchor=west]                at (1.97,1.66) {$\{3,5,9\}$};

  \node[term] at (3.50,0.75) {}; \node[lab,anchor=west,xshift=3pt] at (3.50,0.75) {$\{8\}$};
  \node[term] at (5.60,0.95) {}; \node[lab,anchor=west,xshift=2.5pt]  at (5.60,0.95) {$\{4\}$};
  \node[term] at (4.55,0.95) {}; \node[lab,anchor=south,yshift=2.5pt] at (4.55,0.95) {$\{1,6\}$};
  \node[term] at (0.35,1.45) {}; \node[lab,anchor=east,xshift=-3pt] at (0.35,1.45) {$\{2,7\}$};
  \node[term] at (1.40,1.95) {}; \node[lab,anchor=east,xshift=-2.5pt,fill=white] at (1.40,1.95) {$\{5\}$};
  \node[term] at (2.45,1.95) {}; \node[lab,anchor=south west,xshift=2pt,fill=white] at (2.45,1.95) {$\{3,9\}$};

  \node[anchor=north,font=\small] at (2.75,-0.32) {(b)};
\end{scope}
\end{tikzpicture}
\caption{ 
A realization of $\CTCS(9)$. Vertical segments represent clade labels.
\textbf{(a)} The clade $\{2,3,5,7,9\}$ is the MRCA clade of
$A=\{3,5,7\}$ (red). Nine triples share this MRCA clade, yielding
$D_{9,2}^{(2)}=\cdots=D_{9,2}^{(10)}$. 
\textbf{(b)} Grey marks are deleted from the $2$-stopped tree, whose height
$D^*_{9,2}$ is attained when $\{3,5,9\}$ splits.}
\label{fig:ctcs-tree}
\end{figure}

According to the chronological tree representation, the \emph{height} of  $\CTCS(n)$ is given by 
\[  D_n^{*} = \max_{1 \le i \le n} b(\{i\}) = \max_{ \mathsf{C} \in  \mathcal{C}_n :\: |\mathsf{C}|>1} s( \mathsf{C} ) .  \]
More generally, for $1\leq r<n$, the \emph{$r$-shattering time}
$D^*_{n,r}$ is the first time at which every existing clade has size
at most $r$. Equivalently, it is the last splitting time of a clade
of size greater than $r$:
\[
  D^*_{n,r}
  = \max_{\mathsf C\in\mathcal C_n:\,|\mathsf C|>r}
      s(\mathsf C).
\]
In the chronological tree representation, this is the height of
the \emph{$r$-stopped tree}, obtained by truncating each root-to-leaf
path at the birth of its first clade of size at most $r$;
see Figure~\ref{fig:ctcs-tree}(b).
In particular, $D^*_{n,1}=D_n^*$.

In \cite[Open Problem~3]{AldousJansonII},  Aldous and Janson  asked for the first-order
asymptotics of the maximum:
 find constant $c$ such that $D^{*}_n / \log n \to c$ in probability.
The answer is in fact essentially implicit in the existing literature.
Aldous and Janson~\cite{AldousJansonIII} later proved that $\CTCS(n)$
is a homogeneous fragmentation process, while earlier result of Joseph~\cite{Joseph}
studied   heights of generation-indexed  trees associated with such
fragmentation processes. A continuous-time adaptation of Joseph's
argument then yields 
\[ {D_n^*}/{\log n} \xrightarrow[n\to\infty]{\mathrm{a.s.}}  2  \] 
thereby resolving \cite[Open Problem~3]{AldousJansonII}.  Joseph's 
result also showed that, for every fixed $r \ge 2$, 
$ {D^{*}_{n,r}}/{\log n} \to {\theta_*}/{\kappa_*} $,
where $\theta_*/\kappa_*\in(0,2)$ is independent of $r$. 
While this manuscript was in preparation, Chorbadzhiyska, Minchev, and
Savov~\cite[Theorem~2.12]{ChorbadzhiyskaMinchevSavov} established a
mixed Gumbel limit for  $D_n^*$: They
showed that there is a positive random variable
$\mathsf W_{-1}(\infty)$ with unit mean such that, for every
$x\in\mathbb R$,
\[
  \P(D_n^*-2\log n\leq x)
  \xrightarrow{n\to\infty}
  \E[\exp\{-\tfrac12\mathsf W_{-1}(\infty)e^{-x}\}].
\]
 
The present paper works within Bertoin's continuous-time  homogeneous fragmentation processes~\cite{Bertoin}, while adopting the
fragmentation-tree perspective underlying the generation-indexed models of Joseph~\cite{Joseph}. Our
motivating example $\CTCS(n)$ fits naturally into this setting. 
We give a refined analysis of $D_n^*$ and its generalization $D_{n,r}^*$, identifying their limiting distributions after suitable centering. Furthermore, interpreting $D_{n,r}^*$ as the maximum, over all $A \subset [n]$ with $|A| = r+1$, of the splitting time $s(\mathsf{C}_A)$ of $\mathsf{C}_A$, the most recent common ancestor (MRCA) clade of the leaves indexed by $A$, we introduce its decreasing rearrangement, with multiplicities,
\[
D_{n,r}^{*} = D_{n,r}^{(1)} \geq D_{n,r}^{(2)} \geq \cdots \geq D_{n,r}^{(j)} \geq \cdots  
\]
of the splitting times $\{s(\mathsf{C}_A) : A \subset [n],\ |A| = r+1\}$.
We also establish the joint convergence  of the gaps between consecutive terms of this sequence.

\begin{theorem*}[$\CTCS(n)$]
The following  phase transition holds.
\begin{enumerate}[(i)]
\item Let $(\mathsf{E}_j)_{j\geq1}$ be independent random
variables with $\mathsf{E}_j\sim\operatorname{Exp}(j)$. There exist a
standard Gumbel random variable $G$ and a positive random variable $W_2$
with $\E[W_2]=1$, such that $W_2$ is independent of
$(G,(\mathsf{E}_j)_{j\geq1})$ and, for every fixed $k\geq1$,  
\begin{equation}
 \Bigl( 
   D_{n,1}^*-2\log n,
   \bigl(D_{n,1}^{(j)}-D_{n,1}^{(j+1)}\bigr)_{j=1}^k
  \Bigr)
 \xrightarrow[n\to\infty]{\mathrm{law}}
\Bigl( 
   G+\log\bigl(\tfrac12 W_2\bigr),
   (\mathsf{E}_{j})_{j=1}^k
  \Bigr).
\end{equation} 

\item There exist absolute constants $\theta_*,\kappa_*>0$ and a
positive random variable $Z$ satisfying the following.  For
any fixed $r\geq2$, set $q=r+1$.  There exist a constant
$C_q^\star>0$, a standard Gumbel random variable $G_r$, and a random
sequence $(\mathsf{E}_{r,j})_{j\geq1}$ (not necessarily  
mutually independent), such that
$(G_r,(\mathsf{E}_{r,j})_{j\geq1})$ is independent of $Z$ and, for every
fixed $k\geq1$,  
\begin{equation}
 \Bigl( 
   \kappa_*D_{n,r}^*-\theta_*\log n+\frac32\log\log n,
   \bigl(D_{n,r}^{(j)}-D_{n,r}^{(j+1)}\bigr)_{j=1}^k
  \Bigr)
 \xrightarrow[n\to\infty]{\mathrm{law}}
 \Bigl(  
   G_r+\log(C_q^\star Z),
   (\mathsf{E}_{r,j})_{j=1}^k
  \Bigr).
\end{equation} 
Moreover, in contrast to $(\mathsf{E}_{j})_{j\geq1}$, each limiting gap
$\mathsf{E}_{r,j}$ has an atom at zero:
\begin{equation}
\P(\mathsf{E}_{r,j}=0)>0
\qquad\text{for every }j\geq1.
\end{equation}
\end{enumerate}
\end{theorem*}

\begin{remark}[Concurrent work]
While this manuscript was in preparation, Chorbadzhiyska, Minchev, and
Savov~\cite[Theorem~2.12]{ChorbadzhiyskaMinchevSavov} independently proved a
law of large numbers and a mixed Gumbel limit for the height of the
continuous-time beta-splitting tree, for every $\beta>-2$.  At $\beta=-1$,
where that tree is $\CTCS(n)$, their limit law reads
\[
  \P\bigl(D_n^*-2\log n\leq x\bigr)
  \xrightarrow{n\to\infty}
  \E\bigl[\exp\{-\tfrac12 W_2e^{-x}\}\bigr],
\]
the random variable appearing there being, in our notation, the
additive-martingale limit $W_2$.  This is the height marginal in
part~\textup{(i)} above, i.e.\ the case $q=2$ of
Corollary~\ref{thm:three-regimes}.  Both proofs of this common marginal
proceed via Kingman's paintbox representation and a Chen--Stein Poisson
approximation.  The frozen thresholds $q\geq3$ (that is, $r\geq2$) treated
here require a different route, through the extremal fragments and their
cluster structure.

The two works were obtained independently and are otherwise complementary in
scope. The authors in \cite{ChorbadzhiyskaMinchevSavov}  work within the beta-splitting
family and describe several functionals of a uniformly sampled leaf.  
Here the
dislocation measure is general, and the focus is on separation-time extremal
processes (Theorem~\ref{thm:intro-separation-point-process},
Corollary~\ref{thm:three-regimes}) and clade-count dynamics
(Theorem~\ref{thm:intro-exact-q-clade-process}).  Beyond the common height
marginal, our results show that the freezing transition in the order of
$D^*_{n,r}$ is accompanied by two simultaneous phase transitions: the limiting
separation-time extremal process passes from Poisson to decorated Poisson, and
the limiting clade-count process from a pure-death process to one with both
upward and downward jumps.
\end{remark}

\begin{remark}[DTCS height]\label{rem:dtcs-height}
The genealogical tree underlying $\CTCS(n)$ is the discrete-time critical
beta-splitting tree $\DTCS(n)$. Let $L_n^*$ denote the height of  $\DTCS(n)$.  
In a companion paper,  \cite {MaDTCSHeight}
determines its first-order asymptotic height:
\[
  \frac{L_n^*}{(\log n)^2}
  \xrightarrow[n\to\infty]{\P}
  C_{\mathrm{ht}}
  :=\min_{\theta>1}\frac{\theta}{2\{\psi(\theta)+\gamma\}}
  \approx 0.976,
\]
where $\psi=\Gamma'/\Gamma$ is the digamma function and $\gamma$ is
Euler's constant.
\end{remark}

\subsection{Homogeneous fragmentation trees: main results}\label{sec:main-results}
We now apply the fragmentation-tree viewpoint of
 Joseph~\cite{Joseph} to Bertoin's
continuous-time homogeneous fragmentation processes~\cite{Bertoin}, and
state our results in this setting.  This   naturally encompasses
$\CTCS(n)$ by Aldous and Janson's representation \cite{AldousJansonIII}.  Working in continuous time also removes the lattice
oscillations caused by an integer-valued graph distance, thereby allowing
genuine weak limits along the full sequence rather than only subsequential
limits. 

Denote the space of mass partitions by 
\[
  \mathcal S^\downarrow
  := \Bigl\{   
    \mathbf s=(s_i)_{i\geq1}:
    s_1\geq s_2\geq\cdots\geq0 \ ,
    \  \sum_{i\geq1}s_i\leq1\Bigr\}
\]
 Its subspace of conservative mass partitions
 is 
$  \mathcal S_1^\downarrow
  := \{
    \mathbf s\in\mathcal S^\downarrow:
    \sum_{i\geq1}s_i=1
 \}.$ 
Fix an erosion coefficient $c\geq 0$ and a dislocation measure $\nu$, which is a  
measure on $\mathcal S^\downarrow$ satisfying
\begin{equation}\label{eq:dislocation-integrability}
  \nu\bigl(\{(1,0,\ldots)\}\bigr)=0,
  \qquad
  \int_{\mathcal S^\downarrow}(1-s_1)\,\nu(\dif\mathbf s)<\infty. \tag{$A^{\mathrm{int}}_{\nu}$} 
\end{equation} 

Let $\Pi=(\Pi(t))_{t\geq0}$ be the homogeneous fragmentation process with
characteristics $(c,\nu)$; see Section \ref{sec:fragmentation} and 
\cite[Chapter 3]{Bertoin}.  
The finite restrictions $(\Pi(t)|_{[n]})_{t\geq0}$, $n\geq1$, induce a
sequence of chronological trees $(\mathcal T_n^{c,\nu})_{n\geq1}$, called
\emph{fragmentation trees} here.

We record the dynamics of $\mathcal{T}^{c,\nu}_n$ directly:  
At time $t=0$ there is a single clade $[n]$.    A clade $\mathsf C$ of
size $m\geq2$ carries an independent exponential clock of rate
$\kappa(m)=cm+\kappa_\nu(m)$.  When it rings, 
\begin{itemize}
\item with probability $cm/\kappa(m)$, an \emph{erosion event} occurs:
choose one label uniformly from $\mathsf C$ and split $\mathsf C$ into
this singleton and the sub-clade formed by the remaining labels;
\item with probability $\kappa_\nu(m)/\kappa(m)$, a \emph{dislocation event}
occurs:  First sample a mass partition
$\mathbf s=(s_i)_{i\geq1}$ according to the probability distribution
\begin{equation}\label{eq:finite-visible-dislocation-law}
 \widehat{\nu}_m(\dif\mathbf s)
 :=
 \frac{1-\sum_{i\geq1}s_i^m}{\kappa_\nu(m)}
 \,\nu(\dif\mathbf s).
\end{equation}    
Given $\mathbf s$, partition $\mathsf{C}$ according to Kingman's paintbox construction with frequencies $\mathbf{s}$ (see Figure \ref{fig:paintbox}),   
conditioned on obtaining at least two blocks. Each resulting block forms a sub-clade of $\mathsf{C}$.
\end{itemize}
In either case, the resulting sub-clades are the children of $\mathsf C$,
and every non-singleton child evolves independently according to the same
rule.

\begin{figure}[tb]
\centering
\begin{tikzpicture}[
  font=\footnotesize,
  lab/.style={inner sep=1.2pt},
  box/.style={fill=masscol!12,draw=masscol,line width=.5pt},
  dust/.style={pattern=north east lines,pattern color=dustcol,draw=dustcol,
               line width=.4pt,dash pattern=on 1.6pt off 1.4pt},
  ldot/.style={circle,fill=cladecol,inner sep=1.1pt},
  drop/.style={line width=.45pt,color=cladecol,-{Straight Barb[length=3pt,width=2.6pt]}},
  child/.style={rectangle,rounded corners=1.5pt,draw=cladecol,line width=.5pt,
                fill=cladecol!6,inner sep=2pt},
]

\begin{scope}[x=1cm,y=-1.0cm]
  \node[lab,anchor=east] at (-0.15,-0.85) {$\mathsf C$};
  \foreach \x/\n in {0.55/1,1.60/4,2.50/2,3.09/6,3.68/8,4.50/5,5.23/3,5.87/7}{
     \node[ldot] at (\x,-0.85) {};
     \node[lab,anchor=south,yshift=2pt] at (\x,-0.85) {$\n$};
     \draw[drop] (\x,-0.70) -- (\x,-0.12);}

  \filldraw[box] (0,0) rectangle (2.18,0.60);
  \filldraw[box] (2.18,0) rectangle (4.00,0.60);
  \filldraw[box] (4.00,0) rectangle (5.00,0.60);
  \filldraw[dust] (5.00,0) rectangle (6.10,0.60);
  \node[lab,text=masscol] at (1.09,0.30) {$s_1$};
  \node[lab,text=masscol] at (3.09,0.30) {$s_2$};
  \node[lab,text=masscol] at (4.50,0.30) {$s_3$};
  \node[lab,anchor=west,text=dustcol,xshift=3pt] at (6.10,0.30) {dust};
  \node[lab,anchor=north,yshift=-1pt] at (0,0.60) {$0$};
  \node[lab,anchor=north,yshift=-1pt] at (6.10,0.60) {$1$};

  \foreach \x/\s in {1.09/{\{1,4\}},3.09/{\{2,6,8\}},4.50/{\{5\}},5.23/{\{3\}},5.87/{\{7\}}}{
     \draw[drop] (\x,0.70) -- (\x,1.16);
     \node[child] at (\x,1.40) {$\s$};}
  \node[anchor=west,text width=6.4cm,align=left,text=black!62,inner sep=0pt] at (7.40,0.30)
       {Each label in $\mathsf{C}$ goes independently to box~$i$ with probability $s_i$, and to dust with
        probability $1-\sum_{i\geq1}s_i$. Labels in the same box form a block, while each dust label forms a singleton.};
\end{scope} 
\end{tikzpicture}
\caption{A partition generated by an $\mathbf{s}$-paintbox.}
\label{fig:paintbox}
\end{figure}
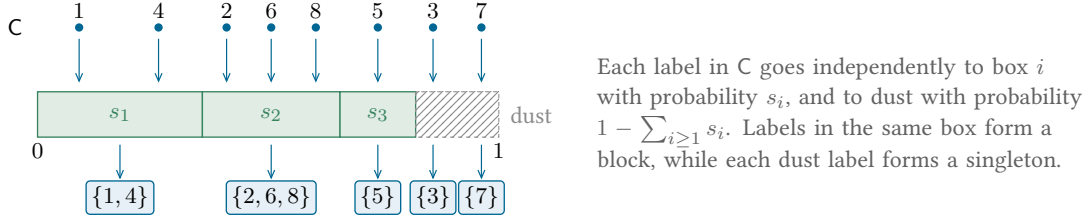

As in the construction of $\CTCS(n)$, the clades \footnote{
In partition terminology, clades correspond to \emph{blocks}.
We generally use \emph{clade} when discussing trees and the finite
restrictions $\Pi(t)|_{[n]}$, and \emph{block} when discussing
partitions of $\mathbb{N}$.
}
 appearing in
$(\Pi(t)|_{[n]})_{t\geq0}$ form a  inclusion tree and carry
chronological marks: a non-singleton clade $\mathsf C$ is represented by
$[b(\mathsf C),s(\mathsf C))$, whereas a singleton is represented by
$\{b(\mathsf C)\}$.  
 Every child $\mathsf C'$ is born when its parent
$\mathsf C$ splits, so $b(\mathsf C')=s(\mathsf C)$.  Unlike in
$\CTCS(n)$, a dislocation may create more than two children.  We denote
the resulting chronological  tree by $\mathcal T_n^{c,\nu}$.

The family $(\mathcal T_n^{c,\nu})_{n\geq1}$ is sampling consistent.
To obtain $\mathcal T_n^{c,\nu}$ from $\mathcal T_{n+1}^{c,\nu}$, remove
the label $n+1$ from every clade, discard empty clades, and identify
consecutive clades with identical resulting label sets. Replace the
corresponding consecutive half-open interval marks by their union.
If the restricted lineage
has become a singleton, retain only the point mark at its first restricted
birth time, rather than a positive-length interval mark.  This gives $\mathcal T_n^{c,\nu}$ because
$ \bigl(\Pi(t)|_{[n+1]}\bigr)|_{[n]}=\Pi(t)|_{[n]}$.

\medskip
\noindent\textbf{Separation times.} 
For a finite  set $A\subset\mathbb N$ with $|A|\geq2$, we define the
\emph{separation time} of $A$ by
\[
  T_A:=\inf\bigl\{t\geq0:
    A\text{ is not contained in any block of the partition }\Pi(t)
  \bigr\}.
\]
In the chronological tree $\mathcal T_n^{c,\nu}$ with $A\subset[n]$, $T_A=s(\mathsf C_A)$
is the splitting time of the MRCA clade $\mathsf C_A$ of $A$ in $\mathcal T_n^{c,\nu}$. 
By consistency under restriction, this definition of $T_A$ does not depend on $n \ge \max A$. 
The time $T_A$ is also called the \emph{coalescence level} of $A$
in \cite{LambertContour}.

For $\mathcal T_n^{c,\nu}$, we use the same $r$-stopped tree construction of $\mathcal T_n^{c,\nu}$   as in Section \ref{sec:ctcs-question}.  Thus, for $1\leq r<n$, the $r$-shattering time is given by
\[
  D^*_{n,r}=\max_{A\in\binom{[n]}{r+1}}T_A .  
\]  
Here by a slight abuse of notation, we write $\binom{I}{q}$ for the collection of all $q$-element subsets of $I$.
The \emph{ranked $(r+1)$-point separation times}  
 $(D_{n,r}^{(j)}: 1\leq j\leq\binom{n}{r+1})$ is the
decreasing rearrangement, counted with multiplicity, of
$\bigl(T_A:A\in\binom{[n]}{r+1}\bigr)$.

 \smallskip
\noindent\textbf{Generating function.} We introduce the generating functions 
\begin{equation}\label{eq:intro-kappa}
  \kappa_\nu(\theta):=\int_{\mathcal S^\downarrow}
  \Bigl(1-\sum_{i\geq1}s_i^\theta\Bigr)\nu(\dif\mathbf s) \ , 
  \quad 
  \kappa(\theta):=c\theta+\kappa_\nu(\theta),
  \qquad \theta  \ge 1.
\end{equation}
Using \eqref{eq:dislocation-integrability} and  $1-\sum_{i\geq1}s_i^\theta\leq \theta(1-s_1)$, we see $\kappa_\nu$ is well
defined. Indeed $\kappa_\nu$ is smooth on
$(1,\infty)$, by using the fact that for $m \ge 1$ and $p >1$, $ \sup_{x \in (0,1)} \frac{ x^p \log^{m}(1/x)}{ x \wedge (1-x)} <\infty $ and dominated convergence. 
  Moreover $  (-1)^{k+1} \kappa^{(k)}_\nu(\theta)=\int_{\mathcal S^\downarrow} \sum_{i\geq1}s_i^\theta\log^{k}(1/s_i) \nu(\dif\mathbf s) \in (0,\infty)$  for every integer $k\geq1$ and every $\theta>1$. \footnote{Throughout we assume $0 \log \frac{1}{0} = \lim_{s \downarrow 0} s \log \frac{1}{s} =0$}
 We therefore   define
\[ \Delta(\theta) := \theta \kappa'(\theta) - \kappa(\theta) =\theta \kappa_{\nu}'(\theta) - \kappa_{\nu}(\theta)  \quad \text{ for } \theta >1 . \] 

\smallskip
\noindent
\textbf{Standing assumptions on the dislocation measure $\nu$.} We assume that  
 \[ \text{ the function } \ \Delta  \ \text{ has a unique root } \  \theta_* \in (1,\infty).  \] 
 This is equivalent to the following  \emph{entropy-dominance  condition}:
\begin{equation}  
    \lim_{\theta \downarrow 1} \Delta(\theta) =   \int_{\mathcal{S}^\downarrow} \Bigl[ \sum_{i \ge 1}s_i \log (1/s_i) - \big(1-\sum_{i \ge 1} s_i\big) \Bigr] \nu(\dif \mathbf{s}) >0 . \tag{$A^{\mathrm{ed}}_{\nu}$}    \label{eq:entropy-dominance}
\end{equation}
To see this, note first that  $\Delta$ is strictly decreasing (see Figure \ref{fig:freezing} (a)), since  $\Delta'(\theta)=\theta \kappa_{\nu}''(\theta)<0$.  
For any   $\mathbf{s}\in\mathcal S^\downarrow
\setminus \{(1,0,\ldots)\}$, we have
$
 \theta\sum_{i\geq1}s_i^\theta\log(1/s_i)
   - (1-\sum_{i\geq1}s_i^\theta ) \to -1$  as $\theta\to\infty$.
Applying monotone convergence   gives
$
 \lim_{\theta\to\infty}\Delta(\theta)
 = 
  -\nu(\mathcal S^\downarrow) \in [-\infty,0)$. Thus $\lim_{\theta \downarrow 1} \Delta(\theta)>0$ is equivalent to the existence of unique root of $\Delta$  on $(1,\infty)$.

Whenever $\theta_*$ exists,    set 
\begin{equation}\label{eq:def-a*}
    \kappa_*:=\kappa(\theta_*) \ ,
    \quad \sigma_*^2:=-\kappa''(\theta_*) > 0 \ ,
    \quad   a_*=   \Bigl(  \frac{2}{\pi\, \theta_*^3} \frac{\kappa_* }{\sigma_*^2 } \Bigr)^{1/2}   .
\end{equation}
 
We emphasize that condition \eqref{eq:entropy-dominance}  is quite mild:
{ it is automatically satisfied by every nonzero conservative
dislocation measure.  Indeed, the dust term
$1-\sum_{i\geq1}s_i$ vanishes on $\mathcal S_1^\downarrow$, whereas the
entropy term $\sum_{i\geq1}s_i\log(1/s_i)$ is strictly positive away from
the excluded trivial partition.}
Moreover, \eqref{eq:entropy-dominance} implies the genuine-branching:
$\nu(\{\mathbf s:s_2>0\})>0$, because when $s_2=0$, its integrand equals
$s_1\log(1/s_1)-(1-s_1)\leq0$.
 
 We also impose the \emph{no-sudden-extinction condition}
\begin{equation}
\label{eq:no-sudden-extinction} 
  \nu(\{\mathbf0\})=0. \tag{$\mathrm A_\nu^{\mathrm{ne}}$}
\end{equation}

\noindent\textbf{Our results.} 
We introduce the regime-dependent spatial scale and centering
\begin{equation}\label{eq:intro-point-process-normalization} 
 \gamma_q:=
 \begin{cases}
  \kappa(q),&q\leq\theta_*,\\
  \kappa_*:=\kappa(\theta_*),&q>\theta_*,
 \end{cases}  
  \qquad
 m_{n,q}:=
 \begin{cases}
  q\log n,&q<\theta_*,\\
  q\log n-\dfrac12\log\log n,&q=\theta_*,\\
  \theta_*\log n-\dfrac32\log\log n,&q>\theta_*.
 \end{cases} 
\end{equation}
Observe that the coefficients in $m_{n,q}$ and $\gamma_q$ \emph{\textbf{freeze}} once
$q$ crosses the critical value $\theta_*$, i.e., they do not depend on $q$ anymore. 
Figure~\ref{fig:freezing} illustrates this for $\CTCS(n)$.  The corresponding first-order freezing phenomenon was established
by Joseph~\cite{Joseph}.

For a locally compact polish space $E$, let $\Mloc(E)$, the space of locally finite point measures
equipped with the vague topology (see Section \ref{eq:point-process} for details). A point process on $E$ is a random element of $\Mloc(E)$. 
The \emph{separation-time extremal process} is a point process on $(-\infty,\infty]$ defined as 
\begin{equation}\label{eq:intro-separation-point-process}
 \Xi_{n,q}
 :=\sum_{A\in\binom{[n]}q}\delta_{\gamma_qT_A-m_{n,q}} \quad \text{ for each integer }   q \ge 2. 
\end{equation}
In what follows, $\PPP(\mu)$ denotes a Poisson point process with intensity
measure $\mu$, and a decorated Poisson point process $\DPPP(\mu,\mathcal D)$ is obtained by
placing an independent copy of the point process $\mathcal D$ at each atom of $\PPP(\mu)$;
see \eqref{eq-def-DPPP} for the precise definition.

 For a random element $\xi$ taking values in a Polish space
and a sub-$\sigma$-field $\mathcal G$, we write
$\mathcal L(\xi\mid\mathcal G)$ or $\mathcal{L}_{\xi \mid \mathcal{G}}$ for (a version of) the regular
conditional distribution of $\xi$ given $\mathcal G$. 
 
\begin{theorem}[Separation-time extremal process]
\label{thm:intro-separation-point-process}
Let $c\geq0$ be an erosion coefficient and let $\nu$ be a dislocation
measure on $\mathcal S^\downarrow$ satisfying
\eqref{eq:dislocation-integrability}, \eqref{eq:entropy-dominance}, and
\eqref{eq:no-sudden-extinction}.  
Let $(W_{\theta})_{1<\theta<\theta_*}$ and $ Z$ be almost surely positive the martingale limits,  defined in
Lemma~\ref{lem:brw-inputs} (\ref{prop:martingale}, \ref{input:derivative}). 
Fix an integer
$q\geq2$. When
$q>\theta_*$, assume in addition the nonlattice condition
\begin{equation}
  \text{ for every } d>0 \ , \quad 
  \int_{\mathcal S^\downarrow} 
  \sum_{i:s_i>0}s_i
  \ind { \log (1/s_i)\notin d\, \mathbb Z }
  \,\nu(\dif \mathbf s)>0. \tag{$\mathrm{A}^{\mathrm{nl}}_{\nu}$} \label{non-lattice-cond}
\end{equation} 

 There is a  point process
$\Xi_q$ on $(-\infty,\infty]$ with $\Xi_q(\{\infty\})=0$,   such that
\begin{equation}\label{eq:intro-point-process-limit}
 \Xi_{n,q}
 \xrightarrow[n\to\infty]{\mathrm{law}}
 \Xi_q
 \quad\text{in }\Mloc(({-}\infty,+\infty]).
\end{equation}
The law of the limits undergoes the following phase transition. 
\begin{enumerate}[(i)]
\item\label{case:additive} If $q<\theta_*$,   $\Xi_q $ can be coupled with $ W_q$ such that 
\[ \mathcal{L} (  \Xi_q \mid W_q)   = 
 \PPP \bigl( \frac{1}{q!} W_q  e^{-x}\,\dif x \bigr)  . \quad  \]

\item\label{case:critical} If $q=\theta_*$,   then  $\Xi_q $ can be coupled with $Z$ such that  
\[ \mathcal{L} (  \Xi_q \mid Z )  =  \PPP \bigl(  \frac{a_*}{q!} Z  e^{-x}\,\dif x  \bigr) . 
\]
\item\label{case:frozen} If $q>\theta_*$, then there exists a point process $\mathcal D_q^{\star}$ on
$(-\infty,0]$ whose rightmost atom is at zero almost surely,  and a constant
$C_q^\star\in(0,\infty)$, and a coupling of $\Xi_q $ and $Z$ such that 
\[  \mathcal{L} (  \Xi_q \mid Z)   
= \DPPP \bigl(C_q^\star Z e^{-x}\,\dif x, \mathcal D_q^{\star}\bigr) .  \]
 Moreover the law of   $\mathcal{D}^{\star}_q$  has an implicit description given by \eqref{eq:def-Xi-q-a}. 
\end{enumerate} 
\end{theorem}

The next sub-section, as well as  Figure~\ref{fig:decoration}, explains the mechanism behind this transition. Remark~\ref{rem:lattice-case} discusses
assumption~\eqref{non-lattice-cond}.

\begin{corollary}[Shattering times and separation-time gaps]
\label{thm:three-regimes}
Under the assumptions of
Theorem~\ref{thm:intro-separation-point-process}, fix $r\geq1$ and set
$q:=r+1$.  
Let
$ \mathsf A_q^{(1)}\geq\mathsf A_q^{(2)}\geq\cdots $
be the atoms of $\Xi_q$, listed with multiplicities, and set
$\mathsf{\Delta}_{q,j}:=\mathsf A_q^{(j)}-\mathsf A_q^{(j+1)}$. Then, for every
fixed $k\geq1$,
\[
 \Bigl(  
   \gamma_qD_{n,r}^{*}-m_{n,q},
   \bigl(\gamma_q(D_{n,r}^{(j)}-D_{n,r}^{(j+1)})\bigr)_{j=1}^{k}
 \Bigr) 
 \xrightarrow[n\to\infty]{\mathrm{law}}
 \Bigl( 
   \mathsf A_q^{(1)},( \mathsf{\Delta}_{q,j} )_{j=1}^{k} \Bigr).
\] 
Moreover, the limits satisfy the following properties:
\begin{itemize}
  \item Let $G$ be a standard Gumbel random variable, independent of
  $W_q$ in the first case and of $Z$ in the other two.  Then  
\[
 \mathsf A_q^{(1)}
 \overset{\mathrm{law}}{=}
 \begin{cases}
  G+\log(W_q/q!),&q<\theta_*,\\[2mm]
  G+\log(a_*Z/q!),&q=\theta_*,\\[2mm]
  G+\log(C_q^\star Z),&q>\theta_*.
 \end{cases}
\] 
\item 
If $q\leq\theta_*$, 
$(\mathsf{\Delta}_{q,j})_{j\geq1}$ are mutually independent, with
$\mathsf{\Delta}_{q,j}\sim\operatorname{Exp}(j)$. In particular,
\[ \P( \mathsf{\Delta}_{q,j} = 0 )= 0 \quad \text{ for every } \,  j \ge 1 . \]  
While if $q>\theta_*$, they are
generally dependent and satisfy
\[
  \P(\mathsf{\Delta}_{q,j}=0)>0  \quad \text{ for every } \,  j \ge 1 .
\]
\end{itemize} 
\end{corollary}

\begin{figure}[!htbp]
\centering
\begin{tikzpicture}[font=\footnotesize]
\pgfplotsset{
  every axis/.append style={
    width=6.9cm, height=4.5cm,
    line width=.5pt, tick align=outside,
    tick style={line width=.4pt,black!70},
    label style={font=\footnotesize}, tick label style={font=\footnotesize},
    clip mode=individual,
  },
  curve/.style={line width=.9pt,color=masscol,mark=none},
  guide/.style={line width=.4pt,color=black!55,dash pattern=on 1.6pt off 1.4pt,mark=none},
  chord/.style={line width=.6pt,color=black!45,dash pattern=on 2.6pt off 1.8pt,mark=none},
  frozen/.style={line width=1pt,color=cladecol,mark=none},
  prefreeze/.style={line width=1pt,color=hicol,mark=none},
}

\begin{axis}[
  name=A,
  axis x line=middle, axis y line=none,
  xmin=0.92, xmax=4.18, ymin=-1.12, ymax=1.92,
  xtick={1,2,3,4},
  xlabel={$\theta$}, x label style={at={(axis description cs:1,0.37)},anchor=west},
  ]
  \addplot[curve] coordinates {(1.0000,1.6449) (1.0500,1.5296) (1.1000,1.4232) (1.1500,1.3245) (1.2000,1.2327) (1.2500,1.1469) (1.3000,1.0665) (1.3500,0.9909) (1.4000,0.9197) (1.4500,0.8523) (1.5000,0.7885) (1.5500,0.7279) (1.6000,0.6702) (1.6500,0.6152) (1.7000,0.5627) (1.7500,0.5125) (1.8000,0.4643) (1.8500,0.4181) (1.9000,0.3737) (1.9500,0.3310) (2.0000,0.2899) (2.0500,0.2502) (2.1000,0.2118) (2.1500,0.1748) (2.2000,0.1389) (2.2500,0.1042) (2.3000,0.0706) (2.3500,0.0379) (2.4000,0.0062) (2.4500,-0.0245) (2.5000,-0.0545) (2.5500,-0.0836) (2.6000,-0.1120) (2.6500,-0.1396) (2.7000,-0.1665) (2.7500,-0.1928) (2.8000,-0.2184) (2.8500,-0.2435) (2.9000,-0.2679) (2.9500,-0.2918) (3.0000,-0.3152) (3.0500,-0.3381) (3.1000,-0.3604) (3.1500,-0.3824) (3.2000,-0.4038) (3.2500,-0.4249) (3.3000,-0.4455) (3.3500,-0.4657) (3.4000,-0.4855) (3.4500,-0.5050) (3.5000,-0.5241) (3.5500,-0.5429) (3.6000,-0.5613) (3.6500,-0.5794) (3.7000,-0.5972) (3.7500,-0.6147) (3.8000,-0.6319) (3.8500,-0.6489) (3.9000,-0.6655) (3.9500,-0.6819) (4.0000,-0.6980)};
  \addplot[only marks,mark=*,mark size=1.4pt,color=cladecol] coordinates {(2.4100,0)};
  \addplot[only marks,mark=o,mark size=1.4pt,line width=.5pt,color=masscol]
     coordinates {(2,0.2899) (3,-0.3152)};
  \node[anchor=north west,text=masscol,inner sep=1.5pt] at (axis cs:1.30,1.88)
       {$\Delta(\theta)=\theta\kappa'(\theta)-\kappa(\theta)$};
  \draw[line width=.35pt,color=black!55] (axis cs:2.31,0.84) -- (axis cs:2.04,0.36);
  \node[anchor=south west,inner sep=1pt] at (axis cs:2.28,0.84) {$\Delta(2)>0$};
  \draw[line width=.35pt,color=black!55] (axis cs:2.42,-0.58) -- (axis cs:2.93,-0.34);
  \node[anchor=north west,inner sep=1pt] at (axis cs:1.36,-0.50) {$\Delta(3)<0$};
  \node[anchor=north,text=cladecol,inner sep=1.5pt] at (axis cs:2.4100,-0.04) {$\theta_*$};
\end{axis}

\begin{axis}[
  name=B, at={($(A.right of south east)+(0.75cm,0)$)}, anchor=left of south west,
  axis x line=bottom, axis y line=none,
  xmin=1.42, xmax=5.4, ymin=1.856, ymax=2.53,
  xtick={2,3,4,5},
  xlabel={$q$}, x label style={at={(axis description cs:1,0)},anchor=west},
  ]
  \addplot[fill=cladecol!7,draw=none] coordinates
     {(2.4100,1.856) (5.4,1.856) (5.4,2.53) (2.4100,2.53)} \closedcycle;
  \addplot[chord] coordinates {(1.5000,2.4442) (1.5500,2.3505) (1.6000,2.2751) (1.6500,2.2138) (1.7000,2.1635) (1.7500,2.1220) (1.8000,2.0877) (1.8500,2.0592) (1.9000,2.0356) (1.9500,2.0161) (2.0000,2.0000) (2.0500,1.9869) (2.1000,1.9764) (2.1500,1.9680) (2.2000,1.9616) (2.2500,1.9569) (2.3000,1.9537) (2.3500,1.9518) (2.4000,1.9510) (2.4500,1.9513) (2.5000,1.9526) (2.5500,1.9546) (2.6000,1.9574) (2.6500,1.9609) (2.7000,1.9651) (2.7500,1.9697) (2.8000,1.9749) (2.8500,1.9806) (2.9000,1.9867) (2.9500,1.9932) (3.0000,2.0000) (3.0500,2.0072) (3.1000,2.0146) (3.1500,2.0224) (3.2000,2.0304) (3.2500,2.0386) (3.3000,2.0471) (3.3500,2.0558) (3.4000,2.0646) (3.4500,2.0737) (3.5000,2.0829) (3.5500,2.0922) (3.6000,2.1017) (3.6500,2.1114) (3.7000,2.1211) (3.7500,2.1310) (3.8000,2.1410) (3.8500,2.1510) (3.9000,2.1612) (3.9500,2.1715) (4.0000,2.1818) (4.0500,2.1922) (4.1000,2.2027) (4.1500,2.2133) (4.2000,2.2239) (4.2500,2.2346) (4.3000,2.2454) (4.3500,2.2561) (4.4000,2.2670) (4.4500,2.2779) (4.5000,2.2888) (4.5500,2.2998) (4.6000,2.3108) (4.6500,2.3218) (4.7000,2.3329) (4.7500,2.3440) (4.8000,2.3552) (4.8500,2.3663) (4.9000,2.3775) (4.9500,2.3888) (5.0000,2.4000)};
  \addplot[prefreeze] coordinates {(1.5000,2.4442) (1.5500,2.3505) (1.6000,2.2751) (1.6500,2.2138) (1.7000,2.1635) (1.7500,2.1220) (1.8000,2.0877) (1.8500,2.0592) (1.9000,2.0356) (1.9500,2.0161) (2.0000,2.0000) (2.0500,1.9869) (2.1000,1.9764) (2.1500,1.9680) (2.2000,1.9616) (2.2500,1.9569) (2.3000,1.9537) (2.3500,1.9518) (2.4000,1.9510) (2.4100,1.9510)};
  \addplot[frozen] coordinates {(2.4100,1.9510) (5.4,1.9510)};
  \addplot[guide] coordinates {(2.4100,1.856) (2.4100,2.53)};
  \addplot[only marks,mark=o,mark size=1.4pt,line width=.5pt,color=black!45]
     coordinates {(3,2.0) (4,2.1818) (5,2.4)};
  \draw[-{Straight Barb[length=3pt,width=2.6pt]},line width=.45pt,color=black!45]
        (axis cs:3,1.99) -- (axis cs:3,1.9645);
  \draw[-{Straight Barb[length=3pt,width=2.6pt]},line width=.45pt,color=black!45]
        (axis cs:4,2.172) -- (axis cs:4,1.9645);
  \draw[-{Straight Barb[length=3pt,width=2.6pt]},line width=.45pt,color=black!45]
        (axis cs:5,2.390) -- (axis cs:5,1.9645);
  \addplot[only marks,mark=*,mark size=1.5pt,color=hicol] coordinates {(2,2.0)};
  \addplot[only marks,mark=*,mark size=1.5pt,color=cladecol]
     coordinates {(3,1.9510) (4,1.9510) (5,1.9510)};
  \node[anchor=west,inner sep=1.5pt] at (axis cs:1.70,2.31) {$c_1(q)$};
  \node[anchor=south east,text=black!55,inner sep=1.5pt] at (axis cs:3.95,2.28) {$q/\kappa(q)$};
  \node[anchor=south,text=cladecol,inner sep=2.5pt] at (axis cs:4.62,1.9510) {$\theta_*/\kappa_*$};
  \node[anchor=north,text=cladecol,inner sep=1pt] at (axis cs:2.4100,1.854) {$\theta_*$};
  \node[anchor=south west,text=hicol,inner sep=1.5pt] at (axis cs:1.62,1.876) {pre-freezing};
  \node[anchor=south east,text=cladecol,inner sep=1.5pt] at (axis cs:5.36,1.876) {frozen};
\end{axis}

\begin{axis}[
  name=C, at={($(A.south west)+(0,-1.75cm)$)}, anchor=north west,
  axis lines=left,
  xmin=1.42, xmax=5.4, ymin=-1.46, ymax=0.42,
  xtick={2,3,4,5}, ytick={-1.2143,-0.4048,0},
  yticklabels={$-\tfrac{3}{2\kappa_*}$,$-\tfrac{1}{2\kappa_*}$,$0$},
  xlabel={$q$}, x label style={at={(axis description cs:1,0)},anchor=west},
  ]
  \addplot[fill=cladecol!7,draw=none] coordinates
     {(2.4100,-1.46) (5.4,-1.46) (5.4,0.42) (2.4100,0.42)} \closedcycle;
  \addplot[frozen] coordinates {(2.4100,-1.2143) (5.4,-1.2143)};
  \addplot[guide] coordinates {(2.4100,-1.46) (2.4100,0.02)};
  \addplot[only marks,mark=o,mark size=1.5pt,line width=.6pt,color=cladecol]
     coordinates {(2.4100,-1.2143)};
  \addplot[only marks,mark=*,mark size=1.5pt,color=hicol] coordinates {(2,0)};
  \addplot[only marks,mark=*,mark size=1.5pt,color=black] coordinates {(2.4100,-0.4048)};
  \addplot[only marks,mark=*,mark size=1.5pt,color=cladecol]
     coordinates {(3,-1.2143) (4,-1.2143) (5,-1.2143)};
  \node[anchor=west,inner sep=1.5pt] at (axis cs:1.54,0.365) {$c_2(q)$};
  \node[anchor=south west,text=hicol,inner sep=1.5pt] at (axis cs:1.46,0.03) {no correction};
  \node[anchor=south west,text=cladecol,inner sep=1.5pt] at (axis cs:2.95,-1.19)
       {Bramson correction, frozen};
  \node[anchor=west,inner sep=2.5pt,align=left] at (axis cs:2.47,-0.4048)
       {$q=\theta_*$\\(does not occur)};
  \node[anchor=north east,text=cladecol,inner sep=1pt] at (axis cs:2.36,-1.27) {$\theta_*$};
\end{axis}

\begin{axis}[
  name=D, at={($(B.south west)+(0,-1.75cm)$)}, anchor=north west,
  axis x line=bottom, axis y line=none,
  xmin=-3.2, xmax=8.7, ymin=0, ymax=0.72,
  xtick={\empty},
  xlabel={$x$}, x label style={at={(axis description cs:1,0)},anchor=west},
  ]
  \addplot[domain=-3.2:8.7,samples=130,line width=.8pt,color=hicol,
           dash pattern=on 2.4pt off 1.6pt]
     {exp(-(x-5.5))*exp(-exp(-(x-5.5)))};
  \addplot[frozen] coordinates {(-3.200,0.0001) (-3.070,0.0002) (-2.940,0.0003) (-2.810,0.0006) (-2.680,0.0010) (-2.550,0.0018) (-2.420,0.0028) (-2.290,0.0044) (-2.160,0.0067) (-2.030,0.0097) (-1.900,0.0138) (-1.770,0.0189) (-1.640,0.0254) (-1.510,0.0332) (-1.380,0.0424) (-1.250,0.0530) (-1.120,0.0650) (-0.990,0.0781) (-0.860,0.0922) (-0.730,0.1070) (-0.600,0.1223) (-0.470,0.1377) (-0.340,0.1530) (-0.210,0.1677) (-0.080,0.1816) (0.050,0.1945) (0.180,0.2060) (0.310,0.2161) (0.440,0.2245) (0.570,0.2312) (0.700,0.2361) (0.830,0.2392) (0.960,0.2407) (1.090,0.2404) (1.220,0.2387) (1.350,0.2355) (1.480,0.2310) (1.610,0.2254) (1.740,0.2189) (1.870,0.2115) (2.000,0.2035) (2.130,0.1949) (2.260,0.1860) (2.390,0.1768) (2.520,0.1675) (2.650,0.1582) (2.780,0.1489) (2.910,0.1398) (3.040,0.1309) (3.170,0.1222) (3.300,0.1138) (3.430,0.1058) (3.560,0.0981) (3.690,0.0907) (3.820,0.0838) (3.950,0.0772) (4.080,0.0711) (4.210,0.0653) (4.340,0.0599) (4.470,0.0549) (4.600,0.0502) (4.730,0.0458) (4.860,0.0418) (4.990,0.0381) (5.120,0.0347) (5.250,0.0315) (5.380,0.0286) (5.510,0.0260) (5.640,0.0236) (5.770,0.0213) (5.900,0.0193) (6.030,0.0175) (6.160,0.0158) (6.290,0.0142) (6.420,0.0128) (6.550,0.0116) (6.680,0.0104) (6.810,0.0094) (6.940,0.0085) (7.070,0.0076) (7.200,0.0068) (7.330,0.0061) (7.460,0.0055) (7.590,0.0050) (7.720,0.0044) (7.850,0.0040) (7.980,0.0036) (8.110,0.0032) (8.240,0.0029) (8.370,0.0026) (8.500,0.0023) (8.630,0.0021)};
  \addplot[frozen] coordinates {(-3.200,0.0000) (-3.070,0.0000) (-2.940,0.0000) (-2.810,0.0000) (-2.680,0.0000) (-2.550,0.0000) (-2.420,0.0000) (-2.290,0.0000) (-2.160,0.0000) (-2.030,0.0000) (-1.900,0.0000) (-1.770,0.0001) (-1.640,0.0001) (-1.510,0.0002) (-1.380,0.0004) (-1.250,0.0008) (-1.120,0.0013) (-0.990,0.0022) (-0.860,0.0035) (-0.730,0.0054) (-0.600,0.0080) (-0.470,0.0115) (-0.340,0.0160) (-0.210,0.0218) (-0.080,0.0288) (0.050,0.0373) (0.180,0.0472) (0.310,0.0584) (0.440,0.0709) (0.570,0.0845) (0.700,0.0990) (0.830,0.1141) (0.960,0.1294) (1.090,0.1448) (1.220,0.1599) (1.350,0.1743) (1.480,0.1877) (1.610,0.2000) (1.740,0.2109) (1.870,0.2202) (2.000,0.2278) (2.130,0.2337) (2.260,0.2378) (2.390,0.2401) (2.520,0.2407) (2.650,0.2398) (2.780,0.2374) (2.910,0.2336) (3.040,0.2286) (3.170,0.2225) (3.300,0.2155) (3.430,0.2079) (3.560,0.1996) (3.690,0.1908) (3.820,0.1818) (3.950,0.1726) (4.080,0.1632) (4.210,0.1539) (4.340,0.1447) (4.470,0.1357) (4.600,0.1268) (4.730,0.1183) (4.860,0.1101) (4.990,0.1022) (5.120,0.0946) (5.250,0.0875) (5.380,0.0807) (5.510,0.0744) (5.640,0.0684) (5.770,0.0628) (5.900,0.0575) (6.030,0.0527) (6.160,0.0481) (6.290,0.0439) (6.420,0.0401) (6.550,0.0365) (6.680,0.0332) (6.810,0.0302) (6.940,0.0274) (7.070,0.0248) (7.200,0.0225) (7.330,0.0204) (7.460,0.0184) (7.590,0.0167) (7.720,0.0150) (7.850,0.0136) (7.980,0.0122) (8.110,0.0110) (8.240,0.0099) (8.370,0.0090) (8.500,0.0081) (8.630,0.0072)};
  \addplot[frozen] coordinates {(-3.200,0.0000) (-3.070,0.0000) (-2.940,0.0000) (-2.810,0.0000) (-2.680,0.0000) (-2.550,0.0000) (-2.420,0.0000) (-2.290,0.0000) (-2.160,0.0000) (-2.030,0.0000) (-1.900,0.0000) (-1.770,0.0000) (-1.640,0.0000) (-1.510,0.0000) (-1.380,0.0000) (-1.250,0.0000) (-1.120,0.0000) (-0.990,0.0000) (-0.860,0.0000) (-0.730,0.0000) (-0.600,0.0000) (-0.470,0.0000) (-0.340,0.0000) (-0.210,0.0001) (-0.080,0.0002) (0.050,0.0003) (0.180,0.0006) (0.310,0.0010) (0.440,0.0017) (0.570,0.0027) (0.700,0.0043) (0.830,0.0065) (0.960,0.0095) (1.090,0.0134) (1.220,0.0185) (1.350,0.0249) (1.480,0.0326) (1.610,0.0417) (1.740,0.0522) (1.870,0.0640) (2.000,0.0770) (2.130,0.0911) (2.260,0.1059) (2.390,0.1211) (2.520,0.1366) (2.650,0.1518) (2.780,0.1666) (2.910,0.1806) (3.040,0.1935) (3.170,0.2052) (3.300,0.2154) (3.430,0.2239) (3.560,0.2307) (3.690,0.2358) (3.820,0.2391) (3.950,0.2406) (4.080,0.2405) (4.210,0.2388) (4.340,0.2358) (4.470,0.2314) (4.600,0.2259) (4.730,0.2194) (4.860,0.2121) (4.990,0.2041) (5.120,0.1956) (5.250,0.1867) (5.380,0.1776) (5.510,0.1683) (5.640,0.1589) (5.770,0.1497) (5.900,0.1405) (6.030,0.1316) (6.160,0.1229) (6.290,0.1144) (6.420,0.1064) (6.550,0.0986) (6.680,0.0913) (6.810,0.0843) (6.940,0.0777) (7.070,0.0715) (7.200,0.0657) (7.330,0.0603) (7.460,0.0552) (7.590,0.0505) (7.720,0.0462) (7.850,0.0421) (7.980,0.0384) (8.110,0.0349) (8.240,0.0318) (8.370,0.0288) (8.500,0.0262) (8.630,0.0237)};
  \node[anchor=south,text=hicol,inner sep=1pt]    at (axis cs:5.5,0.45) {$q=2$};
    \node[anchor=south,text=cladecol,inner sep=1pt]
       at (axis cs:1.0,0.245) {$5$};
  \node[anchor=south,text=cladecol,inner sep=1pt]
       at (axis cs:2.5,0.245) {$4$};
  \node[anchor=south,text=cladecol,inner sep=1pt]
       at (axis cs:4.0,0.245) {$3$};
  \draw[{Straight Barb[length=3pt,width=2.6pt]}-{Straight Barb[length=3pt,width=2.6pt]},
        line width=.45pt,color=cladecol]
       (axis cs:2.5,0.335) -- (axis cs:4.0,0.335);
  \node[anchor=south,text=cladecol,inner sep=1.5pt]
       at (axis cs:3.25,0.345)
       {$\log(C^\star_3/C^\star_4)$};
  \node[anchor=north west,text=cladecol,inner sep=1.5pt,align=left] at (axis cs:-3.15,0.65)
       {$q>\theta_*$ };
  \node[anchor=north east,text=hicol,inner sep=1.5pt,align=right] at (axis cs:8.65,0.65)
       {$q<\theta_*$ };
\end{axis}

\node[anchor=north,font=\small] at ($(A.south)+(0,-0.62cm)$) {(a) the threshold $\theta_*$};
\node[anchor=north,font=\small] at ($(B.south)+(0,-0.62cm)$) {(b) leading constant};
\node[anchor=north,font=\small] at ($(C.south)+(0,-0.62cm)$) {(c) second order};
\node[anchor=north,font=\small] at ($(D.south)+(0,-0.62cm)$) {(d) the limiting distribution};
\end{tikzpicture}
\caption{Phase diagram for the shattering times of $\CTCS$.
\textbf{(a)} The threshold $\theta_*\approx2.41$.
\textbf{(b)}, \textbf{(c)} The leading and the subleading constant of the order of $D^{*}_{n,q-1}$.
\textbf{(d)} Schematic densities of its limiting law, whose shape also freezes for
$q\geq\theta_*$: these laws then differ only by a translation.}
\label{fig:freezing}
\end{figure}
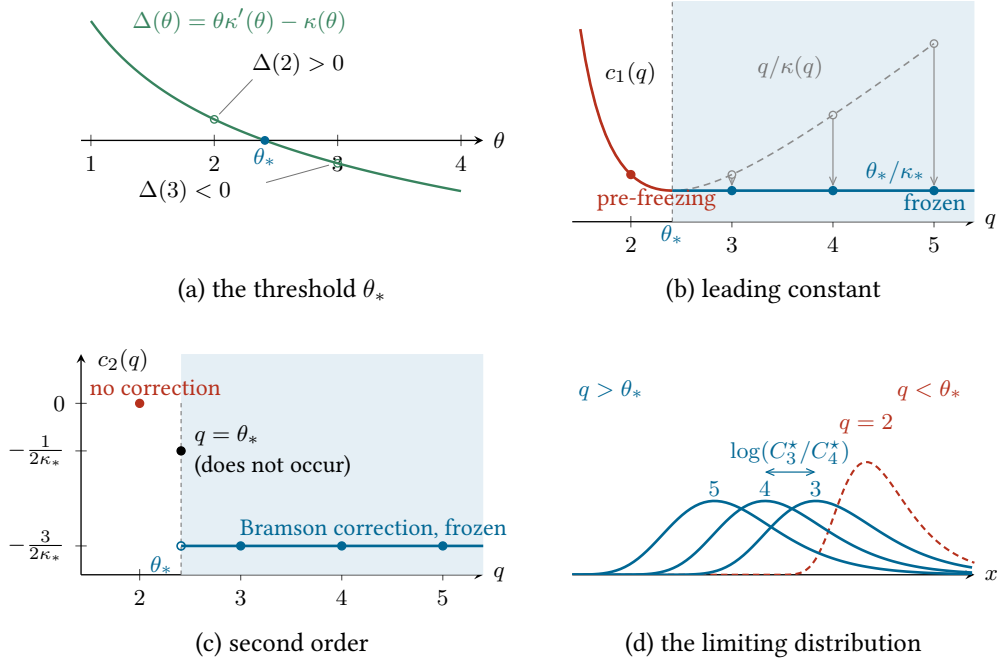

The separation-time extremal process determines the number of unseparated
$q$-subsets, but in general it does not determine the number of size-$q$
clades present at the observation time.  Their count therefore requires a
separate process limit.  For $q\geq2$ and $x\in\mathbb{R}$, let  
\begin{equation}\label{eq:unified-sampling-window}
  t_{n,q}(x) := \frac{m_{n,q}+x}{\gamma_q} \vee 0
  \quad\text{and}\quad
  N_{n,q}(x) :=
  |\{\mathsf C\in\Pi(t_{n,q}(x))|_{[n]}:|\mathsf C|=q\}|.
\end{equation}

\begin{theorem}[Size-$q$ clade count]
\label{thm:intro-exact-q-clade-process}
Under the assumptions of
Theorem~\ref{thm:intro-separation-point-process}, fix an integer $q\geq2$.
The following convergence holds in
$D(\R,\mathbb Z_{\geq0})$, equipped with the local Skorokhod
$J_1$ topology.
\begin{enumerate}[(i)]
\item If $q\leq\theta_*$, then
\[
 (N_{n,q}(x))_{x\in\R}
 \xrightarrow[n\to\infty]{\mathrm{law}}
 \bigl(\Xi_q((x,+\infty])\bigr)_{x\in\R}.
\]
Conditional on $W_q$ when $q<\theta_*$, and on $Z$ when $q=\theta_*$, both defined in Theorem \ref{thm:intro-separation-point-process},
the limit $(\Xi_q((x,+\infty]))$ is a pure-death process in which each individual dies at rate one.

\item If $q>\theta_*$, then
\[
 (N_{n,q}(x))_{x\in\R}
 \xrightarrow[n\to\infty]{\mathrm{law}}
 (N_q(x))_{x\in\R},
\]
where $N_q$ is constructed in
Proposition~\ref{thm:frozen-exact-q-blocks}.   In contrast to case~\textup{(i)}, for every $a<b$, with positive probability the limiting process $x \mapsto N_{q}(x)$ is not monotone on $(a,b)$.
\end{enumerate}
\end{theorem}
Figure~\ref{fig:clade-count} contrasts the two kinds of sample path.

\begin{figure}[!htbp]
\centering
\begin{tikzpicture}[
  x=0.85cm,y=0.4cm,font=\footnotesize,
  lab/.style={inner sep=1.2pt},
   ax/.style={line width=.5pt,{Straight Barb[length=3.4pt,width=2.8pt]}-{Straight Barb[length=3.4pt,width=2.8pt]}},
  path1/.style={line width=.9pt,color=cladecol},
  jump/.style={
    line width=.5pt,color=cladecol,
    dash pattern=on 1.4pt off 1.4pt
  },
  guide/.style={
    line width=.4pt,color=cladecol,
    dash pattern=on 1.6pt off 1.4pt
  },
  lead/.style={
    line width=.4pt,color=black!55,
    -{Straight Barb[length=2.6pt,width=2.2pt]}
  },
]

\begin{scope}
  \draw[ax] (-3.45,0) -- (3.45,0);
  \node[lab,anchor=north east,inner sep=2.5pt]
       at (3.45,0) {$x$};
  \draw[line width=.4pt] (0,-0.28) -- (0,0.28);
  \node[lab,anchor=north,inner sep=2.5pt]
       at (0,-0.24) {$0$};

  \foreach \xa/\xb/\y in {
    -3.00/-2.40/6, -2.40/-1.90/5, -1.90/-1.35/4,
    -1.35/-0.70/3, -0.70/0.25/2, 0.25/1.70/1,
    1.70/3.30/0
  }
    {\draw[path1] (\xa,\y) -- (\xb,\y);}

  \foreach \x/\ya/\yb in {
    -2.40/6/5, -1.90/5/4, -1.35/4/3,
    -0.70/3/2, 0.25/2/1, 1.70/1/0
  }
    {\draw[jump] (\x,\ya) -- (\x,\yb);}

  \draw[guide] (-3.00,6) -- (-3.00,6.95);
  \node[lab,anchor=west,text=cladecol,inner sep=2pt]
       at (-2.95,6.8) {$\uparrow\infty$};

  \node[lab,anchor=south west,text=cladecol,inner sep=2.5pt]
       at (0.35,1.05) {$\Xi_q((x,+\infty])$};

  \node[lab,anchor=north,text=black!62,
        align=center,text width=5.9cm]
       at (0.05,-1.35)
       {$q\leq\theta_*$: a pure-death limiting path};
\end{scope}

\begin{scope}[xshift=7.4cm]
  \draw[ax] (-3.45,0) -- (3.45,0);
  \node[lab,anchor=north east,inner sep=2.5pt]
       at (3.45,0) {$x$};
  \draw[line width=.4pt] (0,-0.28) -- (0,0.28);
  \node[lab,anchor=north,inner sep=2.5pt]
       at (0,-0.24) {$0$};

  \foreach \xa/\xb/\y in {
    -3.00/-2.55/5, -2.55/-2.00/4, -2.00/-1.45/5,
    -1.45/-0.75/4, -0.75/-0.40/3, -0.40/-0.05/2,
    -0.05/0.60/3, 0.60/1.15/2, 1.15/1.65/1,
    1.65/2.25/2, 2.25/2.80/1, 2.80/3.30/0
  }
    {\draw[path1] (\xa,\y) -- (\xb,\y);}

  \foreach \x/\ya/\yb in {
    -2.55/5/4, -2.00/4/5, -1.45/5/4,
    -0.75/4/3, -0.40/3/2, -0.05/2/3,
    0.60/3/2, 1.15/2/1, 1.65/1/2,
    2.25/2/1, 2.80/1/0
  }
    {\draw[jump] (\x,\ya) -- (\x,\yb);}

  \node[lab,anchor=south west,text=cladecol,inner sep=2.5pt]
       at (1.72,2.05) {$N_q(x)$};

  \draw[lead] (-1.05,6.35) -- (-1.98,5.15);
  \node[lab,anchor=west,text=black!62,
        align=left,text width=3.35cm]
       at (-1.0,6.3)
       {an extremal fragment splits and makes a new size-$q$ clade};

  \node[lab,anchor=north,text=black!62,
        align=center,text width=5.9cm]
       at (0.05,-1.35)
       {$q>\theta_*$: a possible non-monotone limiting path};
\end{scope}
\end{tikzpicture}
\caption{Schematic limiting paths of the size-$q$ clade-count process.}
\label{fig:clade-count}
\end{figure}
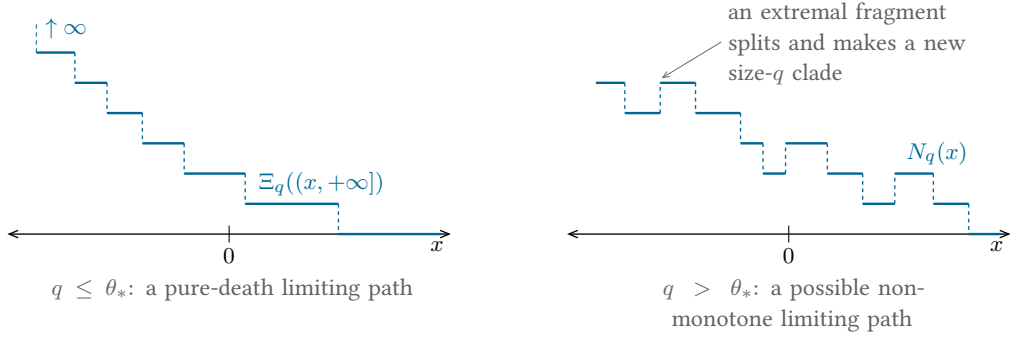

 \begin{remark} 
  \label{emk:theta-star-less-than-3}
Whenever $\theta_*$ exists, it necessarily satisfies $\theta_*\in (1,3)$. The cases $\theta_* \in (1,2)$, $\theta_*=2$ may occur, see \eqref{eq:ex-pa-kappa-delta} in sub-section \ref{sec:ctcs-verification}.
Indeed, let $f(u)=3u^2-2u^3$.
For $x,y\geq0$ with $x+y\leq1$,
$f(x+y)-f(x)-f(y)=6xy(1-x-y)\geq0$.
Since $f$ is increasing on $[0,1]$, it follows that
$\sum_i f(s_i)\leq f(\sum_i s_i)\leq1$ for every
$\mathbf s\in\mathcal S^\downarrow$.
Using $-\log u<(1-u)/u$ for $0<u<1$, we obtain
\[
 -3\sum_i s_i^3\log s_i-\Bigl(1-\sum_i s_i^3\Bigr)
 <\sum_i(3s_i^2-2s_i^3)-1\leq0
\]
for every $\mathbf s\notin\{\mathbf0,(1,0,\ldots)\}$.
At $\mathbf s=\mathbf0$, the left-hand side equals $-1$.
Integrating against $\nu$ gives $\Delta(3)<0$, so the strict
decrease of $\Delta$ yields $\theta_*<3$. 
 \end{remark}

\subsection{Why phase transitions occur? Heuristics and proof strategies} 
  
We begin by explaining where the freezing recorded in
\eqref{eq:intro-point-process-normalization} comes from, and why one and the
same threshold governs the two limit theorems above.  The three lines of
\eqref{eq:intro-point-process-normalization} are produced by a single
conditional first-moment computation, and the freezing they display is the
freezing transition of the partition function of a branching random walk.  We
first carry out that computation; we then read off from it which block sizes
survive at the extremal scale; and we finally explain how that one structural
difference forces both the transitions in 
Theorem~\ref{thm:intro-separation-point-process} and \ref{thm:intro-exact-q-clade-process}.  The discussion is heuristic;
the statements it points to are proved in Sections~\ref{sec:poisson}
and~\ref{sec:frozen}.

\medskip
\noindent\textbf{The conditional first moment guides the choice of $\gamma_q$ and $m_{n,q}$.}
Kingman's paintbox theorem for exchangeable partitions \cite{Kingman} recovers 
$\Pi(t)$ from the mass fragmentation process $X(t)=(X_i(t))_{i\ge1}$, defined in
\eqref{eq:def-Xi} below. 
 Conditionally on $X(t)$,  partition $\mathbb{N}$ via an $X(t)$-paintbox gives $\Pi(t)$. Thus for $q \ge 2$,   
given $X(t)$, the  conditional probability of 
a given $q$-set (a $q$-element subset of $[n]$) is still
unseparated at time $t$ is $S_q(t):= \sum_{i\geq1}X_i(t)^q$. By  linearity we get 
\begin{equation}\label{eq:heuristic-first-moment}
 \E\Bigl[\,\sum_{A\in\binom{[n]}q}\ind{T_A>t} \,\Bigm|\,  X(t) \,\Bigr]
 =\binom nq S_q(t)
 \sim\frac{n^q}{q!}S_q(t).
\end{equation}
Therefore to observe the last (few) unseparated $q$-sets, we seek a time $t$ at which
\begin{equation}
    n^qS_q(t) \asymp 1  . \label{eq-seek-time-t}
\end{equation}

The fragmentation property implies that $(\sum_{i} \delta_{\log(1/X_i(t))}: t>0)$, looked at discrete time skeleton $t\in \{ nh, n \ge 1\}$, forms a branching random walk (BRW). 
Now $S_q(t)=\sum_ie^{-q\log(1/X_i(t))}$ is the partition function of this
BRW, at inverse temperature
$q$. The freezing phase transitions of BRW partition functions have been well-studied in 
\cite{HuShi,AidekonShi,MadauleTip}. The 
typical decay of $S_q(t)$  changes when
$\Delta(q)=q\kappa'(q)-\kappa(q)$ changes sign at $\theta_*$.
Using only the average $\E[S_q(t)]=e^{-\kappa(q)t}$ from
\eqref{eq:exponent-S} would predict $t=q\log n/\kappa(q)$ for every $q$ in \eqref{eq-seek-time-t}.
At and above the threshold, however, this expectation is  no longer honest. From the known result in BRW, we derive:
\begin{enumerate}[(i)]
\item If $q<\theta_*$, then $\Delta(q)>0$ and
$S_q(t)=e^{-\kappa(q)t}W_q(t)$, with $W_q(t)\to W_q>0$ and $X_1(t)^q/ S_q(t) \to 0 $.  
Solving $n^qe^{-\kappa(q)t}=1$ gives
\[
 \kappa(q)t=q\log n.
\]

\item If $q=\theta_*$, the Seneta--Heyde scaling   gives $S_q(t)$ of order
$t^{-1/2}e^{-\kappa_*t}$ and $X_1(t)^q/ S_q(t) \to 0 $.  Solving $n^qt^{-1/2}e^{-\kappa_*t}=1$ we get 
\[
 \kappa_*t 
 =q\log n-\tfrac12\log\log n+O(1).
\]

\item If $q>\theta_*$, then $\Delta(q)<0$ and the sum $S_q(t)$ is dominated by
the largest fragment mass. That is, $S_q(t) \asymp X_1(t)^q$ and $X_{1}(t)$ is of order    $e^{- \kappa_*t/\theta_*}t^{-3 /(2\theta_*)}$. Thus \eqref{eq-seek-time-t} becomes
\begin{equation}\label{eq:heuristic-frozen-calibration}
 \begin{aligned}
 nX_1(t)\asymp1,\qquad
 \kappa_*t 
  =\theta_*\log n-\tfrac32\log\log n+O(1).
 \end{aligned}
\end{equation} 
\end{enumerate} 
This recover
$(\gamma_q,m_{n,q})$ in \eqref{eq:intro-point-process-normalization} and suggest \[ \gamma_q D^{*}_{n,q-1} = m_{n,q} + O_{\P}(1) .\]  
Writing
$\gamma_qt-m_{n,q}=x$  yields the sampling window
$ t_{n,q}(x)=\frac{m_{n,q}+x}{\gamma_q} $.

  \medskip
\noindent\textbf{Clades containing $q$-sets: uniform or mixed sizes.}
The key distinction between the two regimes $q\leq\theta_*$ and $q>\theta_*$
is whether, in the time window $t_{n,q}(0)+O(1)$, unseparated $q$-sets
lie in different clades of size exactly $q$ or can coexist within larger clades.
This is governed by the chance of additional labels landing in a fragment
already containing a given $q$-set.
The relative positions of $t_{n,q}(0)$ and $t_{n,q+1}(0)$ reflect
the same distinction.

If $q\leq\theta_*$, then by (i) (ii), we  have $nX_1(t_{n,q}(x))\to0$ for each fixed $x$,
so the largest fragments lie below the sampling resolution $1/n$.
Conditional on   $X(t_{n,q}(x))$, a fragment carrying a given $q$-set
contains an additional  label in $[n]$ with probability at most
$nX_1(t_{n,q}(x))\to0$.
Moreover, the comparison of the shattering scales gives 
$t_{n,q}(0)-t_{n,q+1}(0)  \to\infty$.
(The gap is of order $\log n$ when $q<\theta_*$, as
$\theta\mapsto\theta/\kappa(\theta)$ is strictly decreasing on
$(1,\theta_*]$, and of order $\log\log n$ when
$q=\theta_*$.)
That is, at time $t_{n,q}(x)=t_{n,q}(0)+O(1)$, all $(q+1)$-sets
have already separated with high probability, leaving no clade
of size $\geq q+1$. In summation,
 clades containing   $q$-sets at time $t_{n,q}(0)+O(1)$  are therefore
\emph{uniform in size}: with high probability, each has size exactly $q$,
so distinct surviving $q$-sets lie in distinct fragments. 
This is Figure~\ref{fig:decoration}\textbf{(a)}.

If $q>\theta_*$, then \eqref{eq:heuristic-frozen-calibration} places
the largest fragment masses at order $1/n$ at time $t_{n,q}(x)$,
for each fixed $x$.
Conditional on  $X(t_{n,q}(x))$, 
for fixed $i$, the number of labels of $[n]$ in the
$i$-th largest fragment has approximately a Poisson  distribution with mean $nX_i(t_{n,q}(x)) \asymp 1$.
Clades containing   $q$-sets can therefore have
\emph{mixed sizes}: for every fixed $r\geq q$, the probability
that a clade of size $r$ presents at time $t_{n,q}(x)$ stays bounded
away from zero.
Correspondingly, neither $\gamma_q$ nor $m_{n,q}$ depends on $q$
in this regime, and the shattering times
$D^*_{n,q-1},\ldots,D^*_{n,r-1}$ all lie within a common time window
 $ \frac{\theta_*}{\kappa_*} \log n - \frac{3}{2 \kappa_*} \log\log n + O(1)$. 
This is Figure~\ref{fig:decoration}\textbf{(b)}.
  
\begin{figure}[tb]
\centering
\begin{tikzpicture}[
  x=1cm,y=0.55cm,font=\footnotesize,
  lab/.style={inner sep=1.2pt},
  frag/.style={line width=1.5pt,color=masscol!85},
  bigfrag/.style={line width=2.4pt,color=masscol!85},
  hifrag/.style={line width=2.4pt,color=hicol},
  ldot/.style={circle,fill=black!45,inner sep=.75pt},
  hdot/.style={circle,fill=hicol,inner sep=.85pt},
  atom/.style={line width=.8pt,color=cladecol},
  hatom/.style={line width=.9pt,color=hicol},
  ax/.style={line width=.5pt,-{Straight Barb[length=3.4pt,width=2.8pt]}},
  axx/.style={line width=.5pt,{Straight Barb[length=3.4pt,width=2.8pt]}-%
                             {Straight Barb[length=3.4pt,width=2.8pt]}},
  br/.style={decorate,decoration={brace,amplitude=2.6pt},line width=.4pt},
  brm/.style={decorate,decoration={brace,mirror,amplitude=2.6pt},line width=.4pt},
  note/.style={lab,anchor=north,text=black!62,align=flush center},
  mult/.style={lab,anchor=south,inner sep=1.6pt,font=\scriptsize},
  limit/.style={lab,anchor=north,text=cladecol,align=flush center,text width=7.6cm},
]
\relpenalty=10000 \binoppenalty=10000 

\begin{scope}
  \node[font=\small] at (-1.05,0.95) {(a)};
  \draw[ax] (-0.28,0) -- (-0.28,1.75);
  \node[lab,anchor=south] at (-0.28,1.75) {$nX_i$};
  \draw[dotted,line width=.45pt,color=black!45] (-0.28,1.40) -- (5.05,1.40);
  \node[lab,anchor=east,inner sep=2pt] at (-0.32,1.40) {$1$};
  \draw[line width=.45pt,color=black!55] (-0.37,0.62) -- (-0.19,0.72);
  \draw[line width=.45pt,color=black!55] (-0.37,0.72) -- (-0.19,0.82);
  \draw[line width=.5pt,color=black!55] (-0.28,0) -- (5.05,0);
  \foreach \k/\h in {0/0.50,1/0.475,2/0.45,3/0.425,4/0.40,5/0.375,6/0.35,7/0.33,
                     8/0.31,9/0.29,10/0.27,11/0.25,12/0.23,13/0.21,14/0.19,
                     15/0.17,16/0.15,17/0.13,18/0.11,19/0.09}
     {\draw[frag] (0.15+0.25*\k,0) -- (0.15+0.25*\k,\h);}
  \foreach \k/\h in {0/0.50,3/0.425,8/0.31}{
     \draw[hifrag] (0.15+0.25*\k,0) -- (0.15+0.25*\k,\h);
     \node[hdot] at (0.15+0.25*\k-0.075,\h+0.18) {};
     \node[hdot] at (0.15+0.25*\k+0.075,\h+0.18) {};}
  \foreach \k/\h in {1/0.475,5/0.375,11/0.25,15/0.17,18/0.11}
     {\node[ldot] at (0.15+0.25*\k,\h+0.18) {};}
  \node[note,text width=6.9cm] at (2.40,-0.75)
       {$q\leq\theta_*$, drawn for $q=2$: \\
       every fragment has mass $o(1/n)$, so a fragment
        carrying a $q$-set almost never carries more};

  \draw[ax,color=black!60] (5.30,1.05) -- (5.85,1.05);
  \node[lab,anchor=north,align=center,text=black!60,inner sep=2pt,font=\scriptsize]
       at (5.575,2.5) {future splits\\$A\mapsto T_A$};

  \draw[axx] (6.05,0) -- (13.34,0);
  \node[lab,anchor=south west,inner sep=2.5pt] at (6.32,0.02) {$\gamma_qT_A-m_{n,q}$};
  \foreach \x in {9.90,9.62,9.40,9.22,9.08,8.97}{\draw[atom] (\x,0) -- (\x,0.75);}
  \node[lab,anchor=east,text=cladecol,inner sep=2pt] at (8.91,0.38) {$\cdots$};
  \foreach \x in {12.60,11.30,10.40}{\draw[hatom] (\x,0) -- (\x,0.75);}
  \draw[dashed,line width=.4pt,color=black!45] (10.15,-0.10) -- (10.15,1.75);
  \node[lab,anchor=north,text=black!60] at (10.15,-0.12) {$x$};
  \draw[br,color=hicol] (11.30,0.98) -- (12.60,0.98);
  \node[lab,anchor=south,text=hicol,inner sep=2.5pt] at (11.95,0.98)
       {$\mathsf\Delta_{q,1}$};
  \draw[br,color=hicol] (10.40,0.98) -- (11.30,0.98);
  \node[lab,anchor=south,text=hicol,inner sep=2.5pt] at (10.85,0.98)
       {$\mathsf\Delta_{q,2}$};
  \node[note,text width=5.6cm] at (9.75,-0.75)
       {each highlighted pair splits into $1+1$, giving one atom}; 
\end{scope}

\begin{scope}[yshift=-4.25cm]
  \node[font=\small] at (-1.05,0.95) {(b)};
  \draw[ax] (-0.28,0) -- (-0.28,2.20);
  \node[lab,anchor=south] at (-0.28,2.20) {$nX_i$};
  \draw[dotted,line width=.45pt,color=black!45] (-0.28,1) -- (5.05,1);
  \node[lab,anchor=east,inner sep=2pt] at (-0.32,1) {$1$};
  \draw[line width=.5pt,color=black!55] (-0.28,0) -- (5.05,0);
  \draw[hifrag] (0.35,0) -- (0.35,1.55);
  \draw[hifrag] (0.90,0) -- (0.90,1.15);
  \foreach \dx in {-0.195,-0.065,0.065,0.195}{\node[hdot] at (0.35+\dx,1.77) {};}
  \foreach \dx in {-0.13,0,0.13}{\node[hdot] at (0.90+\dx,1.37) {};}
  \node[lab,anchor=west,text=hicol,inner sep=2pt] at (0.58,2.05) {4 labels};
  \node[lab,anchor=west,text=hicol,inner sep=2pt] at (1.08,1.37) {3 labels};
  \foreach \x/\h in {1.45/0.95,2.00/0.72,2.55/0.58,3.10/0.45,3.65/0.30,4.20/0.20,4.75/0.12}
     {\draw[bigfrag] (\x,0) -- (\x,\h);}
  \node[ldot] at (2.00,0.93) {};
  \node[ldot] at (2.48,0.79) {}; \node[ldot] at (2.62,0.79) {};
  \node[ldot] at (3.65,0.51) {};
  \node[note,text width=6.9cm] at (2.40,-0.75)
      {$q>\theta_*$, drawn for $q=3$:\\
 finitely many fragments with mass $\Theta(1/n)$,
 each carrying approximately $\Poi(\Theta(1))$ labels};

  \draw[ax,color=black!60] (5.30,1.05) -- (5.85,1.05);
  \node[lab,anchor=north,align=center,text=black!60,inner sep=2pt,font=\scriptsize]
       at (5.575,2.5) {future splits\\$A\mapsto T_A$};

  \draw[axx] (6.05,0) -- (13.34,0);
  \node[lab,anchor=south west,inner sep=2.5pt] at (6.32,0.02) {$\gamma_qT_A-m_{n,q}$};
  \node[lab,anchor=east,text=cladecol,inner sep=2pt] at (8.84,0.38) {$\cdots$};
  \foreach \x in {8.90,8.96,9.03}{\draw[atom] (\x,0) -- (\x,0.75);}
  \foreach \x in {9.30,9.50}{\draw[atom] (\x,0) -- (\x,0.75);}
  \node[mult,text=cladecol] at (9.30,0.73) {$\times3$};
  \draw[brm,color=cladecol!70] (9.24,-0.14) -- (9.56,-0.14);
  \draw[atom] (9.94,0) -- (9.94,0.75);
  \draw[brm,color=cladecol!70] (9.86,-0.14) -- (10.02,-0.14);
  \draw[atom] (10.44,0) -- (10.44,0.75);
  \node[mult,text=cladecol] at (10.44,0.73) {$\times4$};
  \draw[brm,color=cladecol!70] (10.36,-0.14) -- (10.52,-0.14);
  \draw[dashed,line width=.4pt,color=black!45] (10.75,-0.10) -- (10.75,1.75);
  \node[lab,anchor=north,text=black!60] at (10.75,-0.12) {$x$};
  \draw[hatom] (11.10,0) -- (11.10,0.75);
  \draw[brm,color=hicol] (11.02,-0.14) -- (11.18,-0.14);
  \draw[hatom] (11.85,0) -- (11.85,0.75);
  \node[mult,text=hicol] at (11.85,0.73) {$\times3$};
  \draw[hatom] (12.60,0) -- (12.60,0.75);
  \draw[br,color=hicol] (11.85,1.42) -- (12.60,1.42);
  \node[lab,anchor=south,text=hicol,inner sep=2.5pt] at (12.22,1.42)
       {$\mathsf\Delta_{q,1}$};
  \draw[brm,color=hicol] (11.78,-0.14) -- (12.67,-0.14);
\node[note,text width=3.9cm,text=hicol,anchor=north east]
     at (13.34,-0.32)
     {$4\to1+3$: three triples separate at once;
      the remaining triple splits later};
\node[note,text width=3.3cm,text=hicol] at (7.70,-0.75)
     {the 3-label fragment splits into $1+2$, giving one atom};
 
\end{scope}
\end{tikzpicture}
\caption[Separation-time extremal process]{Separation-time extremal process.
Bars  \ \textcolor{masscol}{\rule[-.4pt]{1.5pt}{6.5pt}} \  on the left show selected fragment masses $nX_i$ at time $t_{n,q}(x)$;
dots  \textcolor{black!45}{$\bullet$}  mark labels of $[n]$.
Ticks  \textcolor{cladecol}{\rule[-.4pt]{.8pt}{7pt}}  on the right are atoms of $\Xi_{n,q}$. 
\textcolor{cladecol}{Blue atoms} record $q$-set separations before $t_{n,q}(x)$;
these past separation times cannot be inferred from the left panel.
\textcolor{hicol}{Red atoms} arise from future splits of the highlighted fragments.
The mark $\times k$ means multiplicity $k$.}
\label{fig:decoration}
\end{figure}
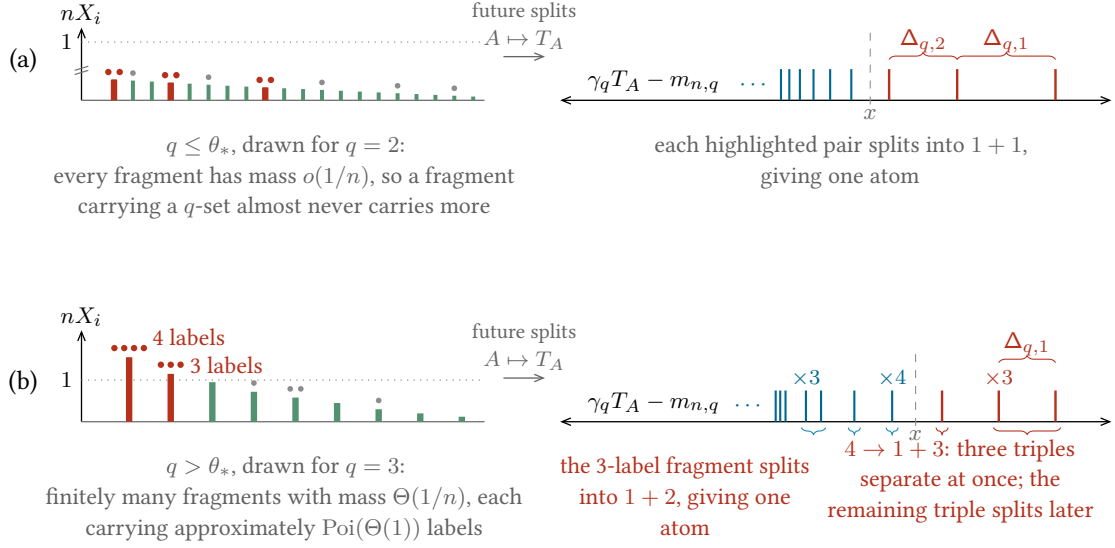
\medskip
\noindent\textbf{ Clade sizes explain the two phase transitions ($q\leq\theta_*$).}   
Since with high probability clades containing  $q$-sets are uniform in sizes at time $t_{n,q}(x)$,    $\Xi_{n,q}$  describes both separation times and clade counts  and in particular    
$
 N_{n,q}(x)=\Xi_{n,q}((x,+\infty]) $ with high probability.


We   explain the Poisson limit of $\Xi_{n,q}$, and starts from the asymptotics of its mean. Rewriting \eqref{eq:heuristic-first-moment}  at time $t_{n ,q}(x)$ and further using  the BRW
estimates  (Lemma~\ref{lem:point-process-inputs}) gives, for some positive random variable $\mathcal{W}_q$, 
\begin{equation}\label{eq-conv-mean}
   \E[\Xi_{n,q}((x,+\infty])\mid X(t_{n,q}(x)) ]
   =\binom nq S_q(t_{n,q}(x))
   \xrightarrow[n\to\infty]{\P}
   \frac{\mathcal W_q}{q!}e^{-x}.
\end{equation}  
Furthermore, $\Xi_{n,q}((x,+\infty])$ counts the unseparated $q$-sets
at time $t_{n,q}(x)$,  and the dependence among its summands becomes
negligible, and thus should be  approximated by Poisson distribution. Indeed,
at time $t$, the events that fixed disjoint $q$-sets remain
unseparated are conditionally independent given $X(t)$.
Two distinct overlapping $q$-sets can both remain unseparated only
if their union lies in the same clade, whose size must exceed $q$. But
with high probability, no such clade is present at time $t_{n,q}(x)$. 
Together with \eqref{eq-conv-mean}, this hints
that $\Xi_{n,q}((x,+\infty])$  converges to  a 
Poisson r.v. with random mean $(\mathcal W_q/q!)e^{-x}$. It  also suggest the 
weak convergence of $\Xi_{n,q}$ to $\Xi_q$, where, conditional on
$\mathcal W_q$, $\Xi_q$ is a PPP with intensity
$(\mathcal W_q/q!)e^{-x}\dif x$. 
The separation-gap and clade-count limits can  be read from $\Xi_q$.  Below we give  a direct explanation. 

Fix $a\in\R$. With high probability, all clades in
$\Pi(t_{n,q}(a))|_{[n]}$ have size at most $q$.
Conditional on this partition, the size-$q$ clades evolve independently,
each with an exponential waiting time of rate $\kappa(q)$ until
its labels separate.
In the rescaled coordinate $x=\kappa(q)t-m_{n,q}$, these waiting
times have distribution $\operatorname{Exp}(1)$.
When $j$ size-$q$ clades remain, the waiting time to the next
separation is the minimum of $j$ independent
$\operatorname{Exp}(1)$ variables, hence has distribution
$\operatorname{Exp}(j)$.
That separation reduces the count to $j-1$, producing only smaller
clades; by memorylessness, the remaining clocks are again independent
with rate one.
This explains the pure-death limit for $N_{n,q}(x)$.
Between the $(j+1)$st and the $j$th last separation times,
exactly $j$ size-$q$ clades remain.
Thus the limiting separation time gaps are these holding times:
they are independent, with
$\mathsf\Delta_{q,j}\sim\operatorname{Exp}(j)$.

\medskip
\noindent\textbf{Clade sizes explain the two phase transitions
($q>\theta_*$).}  Fix $a \in \mathbb{R}$. 
Conditional on the fragment masses  $X(t_{n,q}(a))$, the $i$-th fragment then contains approximately
$\Poi(nX_i(t_{n,q}(a)))$ labels in $[n]$.
By the fragmentation property,  
their subsequent evolutions are independent fragmentation processes
started from the corresponding finite sets of labels. Consequently, the restriction of $ \Xi_{n,q} $ on $(a,\infty]$ should be approximately 
\begin{equation}
   \sum_{i\geq1}\vartheta_a
 \mathcal R^{(i)}_{q,nX_i(t_{n,q}(a))},\label{eq-Xi-n-q-inro}
\end{equation}
where $\mathcal R_{q,\lambda}$ records the separation times,
scaled by $\kappa_*$, of a fragmentation started from
$\Poi(\lambda)$ labels,  and
$\vartheta_a$ shifts those times by $a$. Indeed, $\vartheta_a$ comes because a separation occurring $s$ time units after $t_{n,q}(a)$
has rescaled time $a+\kappa_*s$. 
The BRW extremal-process limit \cite{MadauleTip} gives weak convergence  
fragment masses extremal process $\sum_{i} \delta_{n X_i(t_{n,q}(a))}$ to a limiting point process $ \sum_{i} \delta_{\lambda_i(a)}$, which is a pushforward of a (randomly shifted) decorated poisson point process under the exponential map. 
This suggests  
\begin{equation}
   \Xi_{n,q}\res{(a,+\infty]}
   \Rightarrow
   \sum_{i\geq1}\vartheta_a
   \mathcal R^{(i)}_{q,\lambda_i(a)}
   =:\Xi_{q,a}. \label{eq-Xi-q-a-intro}
\end{equation} 
The limits for different $a$ are consistent under restriction,
defining $\Xi_q$ on the whole line.
The limiting  process $ \sum_{i} \delta_{\lambda_i(a)}$  already has a Poisson cluster structure, and each $\mathcal{R}^{(i)}_{q,\lambda_i(a)}$ records the separation  times comes from the same copy of a finite fragmentation process. This suggests that $\Xi_q$ is a (randomly shifted) DPPP too.  
We establish this representation by removing the random shift,
verifying exp-$1$-stability, and applying the Poisson representation
in Lemma~\ref{lem:maillard-exp-stable} below.

The consequences for separation time gaps and clade counts can be seen in one
split. 
A split of a size-$(q+1)$ clade into a singleton and a size-$q$
clade separates $q$ distinct $q$-sets simultaneously, producing
repeated separation times and increasing the size-$q$ clade count
by one. 
This illustrates  how zero separation gaps arise and how
larger clades create upward jumps in the size-$q$ clade count.
See Figure~\ref{fig:decoration}\textbf{(b)} for the coinciding separation
times and Figure~\ref{fig:clade-count} for the resulting upward jumps of the
count. 
Moreover, a size-$r$ clade contributes $\binom rq$ to the number
of unseparated $q$-sets, but contributes to the size-$q$ clade count
only if $r=q$.
Thus in general $N_q(x)=\Xi_q((x,+\infty])$ no longer holds. The construction of $N_q$ parallels that of $\Xi_q$:
we record each fragment's size-$q$ clade count process instead of
its separation times in \eqref{eq-Xi-n-q-inro}, and pass to the limit in their superposition, 
see \eqref{eq:def-K-q}.

\begin{remark}[Non-lattice condition]
\label{rem:lattice-case}
For $q\leq\theta_*$, individual fragments become negligible at the
sampling scale, and their contribution is governed by
the additive martingale, with Seneta--Heyde scaling at criticality.
These inputs do not require a nonlattice assumption.
For $q>\theta_*$, individual extremal fragments remain visible,
so lattice-phase oscillations in their sizes may persist in the
centered-height distributions. We thus suspect that, in the
lattice case, convergence in the frozen regime may hold only along
suitable subsequences, with limits depending on the lattice phase.
\end{remark}

\subsection{\texorpdfstring{CTCS $(n)$}{CTCS (n)} revisited and further examples}
\label{sec:ctcs-verification}
All examples in this section has erosion  $c=0$. 
We write $B$ for the beta function, $ \Gamma $ for the Gamma function.
\medskip 

\noindent
\textbf{Critical Beta-splitting tree.}
\label{sec:ex-ctcs} By \cite[Section 4.5.]{AldousJansonIII}, the homogeneous fragmentation corresponding to $\CTCS(n)$ has erosion coefficient $c=0$ and 
 dislocation measure $\nu$ given by  the pushforward of the measure
\begin{equation}\label{eq:ex-ctcs-rho}
 \rho_{-1}(\dif x)=\frac{\dif x}{2x(1-x)},\qquad 0<x<1,
\end{equation}
under the map $x\mapsto (\max\{x,1-x\}, \min\{x,1-x\},0,\ldots) $. 
 Its generating function 
 \[  \kappa(\theta)
 =\frac12\int_0^1\frac{1-x^\theta-(1-x)^\theta}{x(1-x)}\,\dif x
 =\psi(\theta)-\psi(1),\] 
 where $\psi(\theta)=\frac{\dif}{\dif \theta} (\ln \Gamma
 (\theta) ) = \frac{\Gamma'(\theta)}{\Gamma(\theta)}$ is the digamma function. Thus the threshold  $ \theta_*\approx2.41$ is the unique solution of 
\begin{equation}\label{eq:ex-ctcs-numerics}
 \theta\psi'(\theta)=\psi(\theta)-\psi(1) \ ,  \quad  \theta > 0.  
\end{equation} 

The CTCS dynamics can also be verified directly from
\eqref{eq:finite-visible-dislocation-law}. Using $\Gamma(x+1)=x\Gamma(x)$ taking logarithms and differentiating gives $\psi(x+1)=\psi(x)+ \frac{1}{x}$. Thus $\kappa(m)=h_{m-1}$ for $m \ge 2$. 
After sampling the ranked mass partition $(s_1,s_2,0,\ldots)$ from
$\widehat\nu_m$, use an independent fair coin to choose which of
$s_1$ and $s_2$ is the left-child mass $x$. Then
\eqref{eq:finite-visible-dislocation-law} gives $x$ the probability distribution
$
 \widehat\rho_m(\dif x)
 =\frac{1-x^m-(1-x)^m}{h_{m-1}}\,
   \frac{\dif x}{2x(1-x)}$. 
Conditionally on $x$, the left sub-clade size is $\operatorname{Bin}(m,x)$ conditioned
to lie in $\{1,\ldots,m-1\}$. Thus, for $1\leq i\leq m-1$, the probability that the left sub-clade has size $i$ is 
 \[ q(m,i)
  =\int_0^1
   \frac{\binom mi x^i(1-x)^{m-i}}{1-x^m-(1-x)^m}\,
   \widehat\rho_m(\dif x) 
  =\frac{\binom mi B(i,m-i)}{2h_{m-1}}
  =\frac1{2h_{m-1}} \Bigl( \frac1i+\frac1{m-i}  \Bigr). \]
The tilt and conditioning factors cancel, recovering exactly the
CTCS split law \eqref{eq:intro-ctcs-rule}.

\medskip 
\noindent
\textbf{ Universal height for conservative binary splitting.}  
Let $\nu$ be a nonzero dislocation measure satisfying
\eqref{eq:dislocation-integrability}, supported on mass partitions
$(s_1,s_2,0,\ldots) \in \mathcal{S}^{\downarrow}_{1}$,
where $s_1+s_2=1$ and $s_1,s_2>0$. Normalize $\nu$ by replacing it
with $\nu/\kappa_\nu(2)$, so that $\kappa(2)=1$.
Then we have 
\[\theta_*\in(2,3). \]
In particular, the  height of the fragmentation tree satisfies  \[ D^*_{n} = 2\log n + O_{\P}(1). \]   No nonlattice assumption is needed here. 
To see this, applying  inequalities $-\log u>1-u$  gives
$  -2\sum_{i=1}^2s_i^2\log s_i
  >2(s_1^2s_2+s_1s_2^2)
  =2s_1s_2=1-s_1^2-s_2^2 $.
Integrating against the nonzero measure $\nu$ yields
$2\kappa'(2)>\kappa(2)$. 
Thus $\Delta(2)>0>\Delta(3)$, and the strict decrease of $\Delta$
implies $\theta_*\in(2,3)$.

\medskip  
\noindent
\textbf{The beta-splitting family.}
\label{sec:ex-binary}
Generally, for $\beta>-2$, let $\nu_\beta$ be the pushforward under
the same map $x\mapsto (\max\{x,1-x\}, \min\{x,1-x\},0,\ldots) $ of
\begin{equation}\label{eq:ex-binary-rho-beta}
 \rho_\beta(\dif x)
 =\frac12x^\beta(1-x)^\beta\,\dif x ,
 \quad 0<x<1.
\end{equation}
With this convention, $\kappa_\beta(2)=B(\beta+2,\beta+2)$, and
\begin{align}\label{eq:ex-binary-kappa-beta}
 \kappa_\beta(\theta)
 &=\frac12\int_0^1
 [1-x^\theta-(1-x)^\theta]x^\beta(1-x)^\beta\,\dif x \\
 &=\frac12[B(\beta+1,\beta+1)-2B(\theta+\beta+1,\beta+1)]
 \quad\text{if }\beta>-1.
\end{align}  
For every $\beta>-2$, the root satisfies $2<\theta_*(\beta)<3$,
by the preceding argument.
As in the CTCS calculation, the tilt in
\eqref{eq:finite-visible-dislocation-law} cancels with the paintbox
conditioning, giving Aldous' beta-splitting law
\cite{AldousCladograms}:    for $1\leq i\leq m-1$, the probability that the left sub-clade has size $i$ is
\begin{equation}\label{eq:ex-binary-split-prob}
\begin{aligned}
 q_\beta(m,i)
 &=\binom mi\frac{1}{2\kappa_\beta(m)}
   \int_0^1x^{i+\beta}(1-x)^{m-i+\beta}\,\dif x\\
 &=\binom mi \frac{B(i+\beta+1,m-i+\beta+1)}
 {2\kappa_\beta(m)},
 \qquad 1\leq i\leq m-1.
\end{aligned}
\end{equation} 
McCullagh, Pitman and Winkel \cite[Theorem~2]{McCullaghPitmanWinkel}
show that this family, together with the equal-split case $\beta=\infty$,
comprises all consistent Markovian binary fragmentation trees of Gibbs type. 
 
\medskip  
\noindent
\textbf{Poisson--Dirichlet split vectors.}
\label{sec:ex-pa}
Take   $0\leq\alpha<1$, $\gamma>-\alpha$.  Set $  \nu= \operatorname{PD}(\alpha,\gamma)$ be the Poisson--Dirichlet distribution with parameter $(\alpha,\gamma)$. These splitting rules belong to the multifurcating Gibbs family
characterized by 
\cite[Theorem~8]{McCullaghPitmanWinkel}. We have 
\[  \kappa_{\alpha,\gamma}(\theta) = 1- \frac{\Gamma(1+\gamma)\Gamma(\theta-\alpha)}
 {\Gamma(\theta+\gamma)\Gamma(1-\alpha)}. \]
Recall $\psi=\Gamma'/\Gamma$.  The threshold $\theta_*(\alpha,\gamma)$ is the unique solution in $(1,\infty)$ of
\begin{equation}\label{eq:ex-pd-root}
 1+\theta
 [\psi(\theta+\gamma)-\psi(\theta-\alpha)] =\frac
 {\Gamma(\theta+\gamma)\Gamma(1-\alpha)}{\Gamma(1+\gamma)\Gamma(\theta-\alpha)}.
\end{equation}
To see this, recall that a random mass partition
$\mathbf{S}=(S_i)\in\mathcal S_1^\downarrow$ with distribution
$\operatorname{PD}(\alpha,\gamma)$ can be constructed as the decreasing
rearrangement of the $\operatorname{GEM}(\alpha,\gamma)$ stick-breaking
sequence $(P_j)_{j\geq1}$, where
$P_j=V_j\prod_{\ell<j}(1-V_\ell)$ and the $V_j$ are independent with
$V_j\sim\operatorname{Beta}(1-\alpha,\gamma+j\alpha)$ (see 
\cite[Equation (3.8)]{PitmanCSP}). 
To calculate $\kappa(\theta)$, choose $I$ conditionally on $\mathbf S$ with
$\P(I=i\mid\mathbf S)=S_i$.   
The GEM representation orders the masses by size-biased sampling of $\mathbf{S}$ (\cite[Theorem~3.2 and Definition~3.3]{PitmanCSP}), so
$S_I$ has the law of the first GEM stick,
$\operatorname{Beta}(1-\alpha,\gamma+\alpha)$. 
Since $\E[S_I^{\theta-1}\mid\mathbf S]=\sum_iS_i^\theta$,
the beta integral gives
$
 \kappa(\theta) =1-\E\sum_iS_i^\theta
 =1-\E[ S_I^{\theta-1}]
  =1-\frac{B(\theta-\alpha,\gamma+\alpha)}
        {B(1-\alpha,\gamma+\alpha)} $, which yields the previous expression.

The subfamily $\gamma=1-\alpha$ gives interesting split vectors related with
linear preferential-attachment trees (\cite[Theorem~1.5]{JansonSplitTrees}), for which
\begin{equation}\label{eq:ex-pa-kappa-delta}
 \kappa_{\alpha}(\theta)=\frac{\theta-1}{\theta-\alpha},\qquad
 \theta_*(\alpha)=1+\sqrt{1-\alpha} \in (1,2]
\end{equation}
In particular, $\operatorname{PD}(0,1)$ gives $\theta_*=2$, so the
corresponding tree height lies in the critical regime, which satisfies $D_n^*=4\log n-\log\log n+O_{\P}(1)$.
For $\operatorname{PD}(1/2,1/2)$, we have $\theta_*=1+1/\sqrt2<2$,
so the tree height lies in the frozen regime and satisfies
$
 D_n^*=\frac{3+2\sqrt2}{2}\log n
 -\frac{3(2+\sqrt2)}{4}\log\log n+O_{\P}(1)$.

Consider a clade with label set $[m]$, and fix a partition
$\{B_1,\ldots,B_k\}$ of $[m]$, with $k\geq2$ and $m_j=|B_j|$.
Conditional on a split of $[m]$, the probability that its child clades are
exactly $B_1,\ldots,B_k$ is given by
\begin{equation}\label{eq:ex-pa-eppf-cond}
 \hat{q}_{\alpha,\gamma}(m_1,\ldots,m_k)
 =\frac{ \prod_{i=1}^{k-1}(\gamma+i\alpha)
       \prod_{j=1}^k(1-\alpha)_{m_j-1\uparrow}}
      {\kappa_{\alpha,\gamma}(m)(\gamma+1)_{m-1\uparrow}}.
\end{equation}
Here $(x)_{r\uparrow}=x(x+1)\cdots(x+r-1)$ and
$(x)_{0\uparrow}=1$.
The formula follows by conditioning the Ewens--Pitman exchangeable
partition probability function (EPPF)
\cite[Theorem~3.2]{PitmanCSP} on at least two blocks.

\medskip
\noindent
\textbf{Symmetric Dirichlet split vectors.}
\label{sec:ex-dirichlet}
Take an integer $b\geq2$ and $\alpha>0$, and set
\begin{equation}\label{eq:ex-dirichlet-vector}
 \nu=\mathrm{Law}(\mathbf S^\downarrow),\qquad
 \mathbf S=(S_1,\ldots,S_b)\sim
 \operatorname{Dirichlet}(\alpha,\ldots,\alpha).
\end{equation}
The ranked vector has distribution $\operatorname{PD}(-\alpha,b\alpha)$
and gives another member of the Gibbs family
\cite[Theorem~8]{McCullaghPitmanWinkel}.
The generating function is
\begin{equation}\label{eq:ex-dirichlet-Ap}
 \kappa_{b,\alpha}(\theta)
 =1-\frac{b\,\Gamma(b\alpha)\Gamma(\theta+\alpha)}
 {\Gamma(\alpha)\Gamma(\theta+b\alpha)},
\end{equation}
and $\theta_*(b,\alpha)$ is the unique solution in $(1,\infty)$ of
$
 1+\theta\{\psi(\theta+b\alpha)-\psi(\theta+\alpha)\}
 =\frac{\Gamma(\alpha)\Gamma(\theta+b\alpha)}
 {b\,\Gamma(b\alpha)\Gamma(\theta+\alpha)}$.  
Indeed, since  $S_1\sim\operatorname{Beta}(\alpha,(b-1)\alpha)$, exchangeability
and the beta integral give
$\E\sum_{i=1}^bS_i^\theta
=b\,B(\theta+\alpha,(b-1)\alpha)/B(\alpha,(b-1)\alpha)$.

For $\alpha=1$, $\mathbf S$  is related to the split vectors in $b$-ary search trees \cite[p.~416]{DevroyeSplitTrees}.
Then $\kappa_{b,1}(\theta)=1-b!\,\Gamma(\theta+1)/\Gamma(\theta+b)$;
for example, $\theta_*(5,1)\approx1.96<2<2.04 \approx \theta_*(4,1)$. 

Consider a clade with label set $[m]$, and fix a partition
$\{B_1,\ldots,B_k\}$ of $[m]$, with $2\leq k\leq b$ and $m_j=|B_j|$.
Conditional on a split, the probability that its child clades are
exactly $B_1,\ldots,B_k$ is
\begin{equation}\label{eq:ex-dirichlet-finite-m-law}
 \hat{q}_{b,\alpha}(m_1,\ldots,m_k)
 =\frac{b!}{(b-k)!}\,
 \frac{\prod_{j=1}^k(\alpha)_{m_j\uparrow}}
 {\kappa_{b,\alpha}(m)(b\alpha)_{m\uparrow}}.
\end{equation}
To verify this, given the paintbox $\mathbf S$, each label independently
chooses box $i$ with probability $S_i$.
Fix distinct boxes $i_1,\ldots,i_k$.
The probability that every label in $B_j$ chooses box $i_j$, for all $j$,
is $\prod_{j=1}^k S_{i_j}^{m_j}$.
Integrating against the Dirichlet density gives
\[
 \E\prod_{j=1}^k S_{i_j}^{m_j}
 =\frac{\Gamma(b\alpha)}{\Gamma(b\alpha+m)}
   \prod_{j=1}^k\frac{\Gamma(\alpha+m_j)}{\Gamma(\alpha)}
 =\frac{\prod_{j=1}^k(\alpha)_{m_j\uparrow}}{(b\alpha)_{m\uparrow}}.
\]
Summing over the $b!/(b-k)!$ assignments of the blocks
$B_1,\ldots,B_k$ to $k$ distinct boxes of the paintbox, and dividing the visible splitting probability   $1-\E\sum_{i=1}^b S_i^m=\kappa_{b,\alpha}(m)$  gives the preceding probability.

\section{Preliminaries}\label{sec:preliminaries}
This section collects the preliminary results on homogeneous fragmentations
used throughout the paper.  None of these results is new: they are standard
consequences of the theory of exchangeable coalescents or adaptations of
results from the branching random walk (BRW) literature.  For completeness,
we provide proofs of selected inputs in Appendix~\ref{sec:preliminary-input-proofs}, including
the details needed to justify the precise formulations used here.

\subsection{Homogeneous fragmentation.} \label{sec:fragmentation}
We first recall the partition and mass formulations; see
\cite[Chapter~3]{Bertoin} for the general theory.  Let $\mathcal P$ be the
space of partitions of $\N$.  
A homogeneous partition fragmentation is an exchangeable c\`adl\`ag Markov process
$\Pi=(\Pi(t))_{t\geq0}$ on $\mathcal P$, started from the trivial
partition $\{\N\}$, with the following \emph{fragmentation property}.
Conditionally on $\Pi(t)=(B_1,B_2,\ldots)$, the subsequent evolutions
inside the blocks $B_j$ are independent and, after relabelling, have the
law of independent copies of $\Pi$ started from $\{\N\}$. 
  
The law of $\Pi$ is determined by an erosion coefficient $c\geq0$ and a
dislocation measure $\nu$ on $\mathcal S^\downarrow$ satisfying
\eqref{eq:dislocation-integrability}.  Independently for each block,
dislocations are driven by a Poisson point process on
$\mathbb R_+\times\mathcal S^\downarrow$ with intensity
$\dif t\,\nu(\dif\mathbf s)$: at each atom $(t,\mathbf s)$, the
affected block is partitioned according to an $\mathbf s$-paintbox.
Thus, for any measurable $A\subset\mathcal S^\downarrow$ with
$\nu(A)<\infty$, events with $\mathbf s\in A$ occur at rate $\nu(A)$
per block.  Although $\nu$ may have infinite total mass,
\eqref{eq:dislocation-integrability} ensures that each finite restriction
has a finite rate of nontrivial dislocations.  Conditional on a
nontrivial dislocation of a block containing $m\geq2$ labels, the mass
partition $\mathbf s$ has distribution
\eqref{eq:finite-visible-dislocation-law}.
Independently, erosion separates each label from its current
nonsingleton block into a singleton at rate $c$.
  
For a block $B\subset\N$, define its asymptotic frequency by 
\[
 \operatorname{freq}(B)
 :=\lim_{n\to\infty}\frac{\#(B\cap[n])}{n},
\]  
whenever the limit exists.  
For every $t\geq0$, exchangeability of $\Pi(t)$ ensures that all of its
blocks have asymptotic frequencies almost surely.  Thus, if
$\Pi(t)=\{B_1(t),B_2(t),\ldots\}$ is any enumeration of its blocks, set   
\begin{equation}
 \label{eq:def-Xi}
 \bigl(X_i(t)\bigr)_{i\geq1}
 :=\bigl(\operatorname{freq}(B_j(t)):j\geq1\bigr)^\downarrow,
\end{equation} 
where $\downarrow$ denotes decreasing rearrangement, padded with zeros
when necessary.  The process
$((X_i(t))_{i\geq1},t\geq0)$ is called the mass fragmentation with
characteristics $(c,\nu)$.  At each fixed time, Kingman's paintbox
representation~\cite{Kingman} recovers the conditional law of $\Pi(t)$
from $(X_i(t))_{i\geq1}$.  The ranked frequencies alone, however, do not
retain the genealogical paths followed by individual labels.  For this
we use a direct interval-fragmentation realization.

\medskip 
\noindent \textbf{Interval fragmentation process.}
Let $\mathcal U$ be the space of open subsets of $(0,1)$.  Every
$U\in\mathcal U$ is the union of countably many pairwise disjoint open
intervals, called its \emph{interval components}.

A homogeneous \textit{interval fragmentation}
$\mathcal I=(\mathcal I(t))_{t\geq0}$ is a Markov process on
$\mathcal U$, continuous in probability, such that
$\mathcal I(0)=(0,1)$ and $\mathcal I(t+s)\subseteq\mathcal I(t)$ for
$s,t\geq0$.  Its fragmentation property says that, conditionally on
$\mathcal I(t)$, the future processes inside distinct interval
components are independent and, after affine rescaling onto $(0,1)$,
have the laws of independent copies of $\mathcal I$; see \cite[Definitions~1--2]{BertoinSelfSimilar}
with self-similarity index $\alpha=0$ and also
\cite[Definition~3.4]{Basdevant}.

An interval fragmentation provides a process-level version of
Kingman's paintbox representation. Labels are placed at i.i.d. uniform
points in $(0,1)$, independently of the entire interval fragmentation. At each time, labels that  lie in the same interval component form a block, while those
outside the open set form singletons.
The following representation is the homogeneous case ($\alpha=0$)
of Bertoin's construction in 
\cite{BertoinSelfSimilar}.

\begin{lemma}[{\cite[Lemmas~5--6]{BertoinSelfSimilar}}]
 \label{lem:genealogical-paintbox}
There exists a realization of $\Pi$ together with a homogeneous interval
fragmentation $\mathcal I=(\mathcal I(t))_{t\geq0}$ and a sequence
$(U_j)_{j\geq1}$ of i.i.d. uniform variables on $(0,1)$, independent of
the entire process $\mathcal I$, with the following properties.
Let $(I_i(t))_{i\geq1}$ be the interval components of
$\mathcal I(t)$ listed in nonincreasing order of length, with ties
broken by their left endpoints and empty sets appended when necessary.
Almost surely, simultaneously for every $t\geq0$,
\begin{enumerate}[(i)]
\item
$X_i(t)=|I_i(t)|$ for all $i \ge 1$; 
\item
distinct labels $j,k\in\N$ are in the same block
of $\Pi(t)$ if and only if $U_j$ and $U_k$ lie in the same interval
component of $\mathcal I(t)$; labels with $U_j\notin\mathcal I(t)$
form singleton blocks.
\end{enumerate} 
\end{lemma}

We henceforth work with this coupling. 
Lemma~\ref{lem:genealogical-paintbox} gives the following identity.

\begin{corollary}
For every set $A\subset\N$ of $q\geq2$ labels, 
\begin{equation}\label{eq:conditional-unseparated-probability}
 \P\bigl[
  A\text{ is contained in one block of }\Pi(t)
  \mid\mathcal I
 \bigr]
 =\sum_{i\geq1}X_i(t)^q.
\end{equation} 
\end{corollary}

\begin{proof}
For every $i\geq1$, conditional on $\mathcal I$,
\[
 \P\bigl[
  U_j\in I_i(t)\text{ for every }j\in A
  \mid\mathcal I
 \bigr]
 =|I_i(t)|^q=X_i(t)^q.
\]
These events are disjoint across interval components.  Labels outside
$\mathcal I(t)$ form singletons and do not contribute since $q\geq2$.
Summing over $i$ proves the result.
\end{proof}
 
\subsection{Martingales and extremal fragments} We begin with introducing some notation. Recall  \eqref{eq:def-Xi}.
For every $\theta>1$, set
\begin{equation}\label{eq:kappa-Sp}
 S_{\theta}(t):=\sum_{i\geq1}X_i(t)^\theta,
 \quad
 W_{\theta}(t):=e^{t\kappa(\theta)}S_{\theta}(t),
 \quad t\geq0.
\end{equation} 
The Laplace exponent formula of the tagged fragment \cite[Theorem 3.2]{Bertoin} gives
\begin{equation}\label{eq:exponent-S}
   \E[ S_{\theta}(t)]=  e^{- t \kappa(\theta)} \quad \text{ and hence } \quad  \E [W_\theta(t)] = 1 .  
\end{equation} 
Recall that $\theta_*$ is the unique root of
$\Delta(\theta):=\theta\kappa'(\theta)-\kappa(\theta)$ and that
$\kappa_*:=\kappa(\theta_*)$. Define
\begin{equation}\label{eq:boundary-position}
 V_i(t):=\theta_*\log\frac1{X_i(t)}-\kappa_* \, t \ ,
 \quad
 Z(t) :=\sum_{i:X_i(t)>0}V_i(t)e^{-V_i(t)}.
\end{equation} 
Note that
$Z(t)=-\theta_*\frac{\partial}{\partial\theta}
W_\theta(t)|_{\theta=\theta_*}$. Finally, we introduce 
the extremal process $\mathcal{E}^{V}_t$, which records all fragments near the largest one:
\begin{equation}\label{eq:extremal-process-definition}
  \mathcal{E}^{V}_t:=\sum_{i:X_i(t)>0}
  \delta_{\frac{3}{2}\log t-V_i(t)},
  \qquad t>0.
\end{equation}
For $\beta>1$, we  put $\exp_{\beta} : x \mapsto \exp( \beta x) $ the exponential function.

\begin{lemma}\label{lem:brw-inputs}
Under assumptions \eqref{eq:dislocation-integrability},
\eqref{eq:entropy-dominance}, and \eqref{eq:no-sudden-extinction}, the
following assertions hold.
\begin{enumerate}[(i)]
\item\label{prop:martingale} For every $\theta>1$,
$(W_\theta(t))_{t\geq0}$ is a nonnegative martingale.  Let $W_\theta$ be its almost sure limit. Then  for any 
$ \theta \in (1,\theta_*)$, there exists $p=p_\theta>1$ such that   
\[ W_\theta(t)\xrightarrow[t \to \infty]{L^p} W_{\theta} \ ; \ \text{ and } \   \P(W_{\theta}>0)=1 \ , \  \E[W_{\theta}]=1. \]
\item\label{input:derivative} $(Z(t))_{t\geq0}$ is a (not necessarily
nonnegative) martingale, and
\[
 Z(t) \xrightarrow[t \to \infty]{\text{a.s.}}Z \in(0,\infty)
  \, .
\]
\item\label{input:SH} At critical value $\theta=\theta_*$, we have 
\begin{equation}
 \sqrt t\,W_{\theta_*}(t) \xrightarrow[t \to \infty]{\P}
 \sqrt{\frac{2}{\pi\sigma_*^2}}\,\frac{Z}{\theta_*} \,.
 \label{eq:SH-scaling}
\end{equation} 
\item\label{input:MinimumBRW} Let $v_*:=\kappa_*/\theta_*=\kappa'(\theta_*)$. As $t\to\infty$,
\begin{equation}\label{eq:largest-fragment-input}
 \frac{\log (1/X_1(t)) - \frac{\kappa_{*}}{\theta_{*}} t}{\log t}
\xrightarrow[t \to \infty]{\P}  \frac{3}{2\theta_*} \,.
\end{equation}
\item\label{input:Madaule}
Assume \eqref{non-lattice-cond}. There exists a point process
$\mathcal E^V$ on $(-\infty,\infty]$ with the following property. 
Fix $k\in\mathbb Z_{\geq1}$. For each $1\leq j\leq k$, let
$\beta_j>1$, $g_j\in\mathrm{BUC}(\mathbb R)$\footnote{Here,
$\mathrm{BUC}(\mathbb R)$ denotes the space of bounded uniformly
continuous functions on $\mathbb R$.}, and
$(g_{j,t})_{t>0}\subset\mathrm{BUC}(\mathbb R)$ satisfy
$
 \|g_{j,t}-g_j\|_\infty\xrightarrow{t\to\infty}0$.
Set
$f_{j,t}(x):=g_{j,t}(x)\exp_{\beta_j}(x)$ and
$f_j(x):=g_j(x)\exp_{\beta_j}(x)$. Then
\begin{equation}\label{eq:weighted-extremal-convergence}
 \Bigl(
 \mathcal E_t^V,Z(t),
 \bigl(\langle f_{j,t},\mathcal E_t^V\rangle\bigr)_{j=1}^k
 \Bigr)
 \xrightarrow[t\to\infty]{\mathrm{law}}
 \Bigl(
 \mathcal E^V,Z,
 \bigl(\langle f_j,\mathcal E^V\rangle\bigr)_{j=1}^k
 \Bigr)
\end{equation}
in $\Mloc((-\infty,\infty])\times\mathbb R^{k+1}$. 
The joint law of  the limit $(\mathcal E^V,Z)$ has the following description. There
exist a constant $C_V>0$ and a point process $\mathcal D$ on
$(-\infty,0]$  whose rightmost atom is zero
almost surely, such that,  
\[
\mathcal{L}( \mathcal E^V \mid Z)  
=\DPPP\bigl(C_VZe^{-x}\,\dif x,\mathcal D\bigr).
\]
Moreover, for every $\beta>1$, both
$\langle\exp_\beta,\mathcal E^V\rangle$ and
$\langle\exp_\beta,\mathcal D\rangle$ are finite almost surely. 
\end{enumerate}
\end{lemma}

The five assertions in Lemma~\ref{lem:brw-inputs} collect standard inputs
from the fragmentation and BRW literature.  The $L^p$ criterion used in
part~(\ref{prop:martingale}) is given for finite offspring in
\cite{LiuCascades} and for possibly infinite offspring in
\cite{AlsmeyerIksanovPolotskiyRosler}; the fragmentation-martingale and
derivative-martingale results originate from \cite{BertoinAsymptotic} and \cite{BertoinRouault}.
Alternatively, the logarithmic fragment sizes form a branching L\'evy process
in the sense of \cite{BertoinMalleinBLP}, and parts~\textup{(i)}--\textup{(ii)}
can be deduced directly from \cite[Proposition~1.4]{BertoinMalleinBiggins}
and \cite[Theorem~2.3]{MalleinShiDerivative}, respectively.  The
Seneta--Heyde scaling in
part~(\ref{input:SH}) and the $3/2$ minimum estimate in
part~\textup{(iv)} were proved for homogeneous fragmentations in
\cite{KyprianouMadaule} and \cite{KyprianouLaneMorters}, respectively;
their BRW counterparts appear in \cite{AidekonShi,A_d_kon_2013}.  The
BRW extremal-process convergence used in part~\textup{(v)} is due to
\cite{MadauleTip}.  To keep the proofs of all five assertions in a common
framework, however, Appendix~\ref{sec:brw-proofs} uses a unified
time-discretization argument, following \cite{BertoinRouault}; this also
avoids relying on continuous-time results whose assumptions do not always
cover the erosion and dissipative cases allowed here.

\subsection{Point processes and Poisson approximation}
\label{eq:point-process}  
We collect results on point processes  used below: Kallenberg's convergence
criterion, the Poisson cluster representation of exp-$1$-stable point processes,
and a Chen--Stein bound for Poisson approximation.

Let   $\Mloc(E)$ denote
the space of locally finite point measures on a locally compact Polish space $E$, that is, measures of
the form
$
 \mu=\sum_{i\in I}\delta_{x_i},  x_i\in E $
such that $\mu(K)<\infty$ for every compact set $K\subset E$. We say such a measure simple if   $x_i \neq x_j$ for all $i \neq j$.  
Equipped with the vague topology, $\Mloc(E)$ is a Polish space (
\cite[Theorem~4.2]{KallenbergRandomMeasures}).  Recall that this topology is characterized by
$\mu_n\to\mu$ vaguely if and only if
$\int_E f\dif \mu_n\to\int_E f\dif \mu$
for every $f\in C_{\mathrm{c}}(E)$.  Moreover, if $\mu_n\to\mu$ vaguely and $B$ is a relatively
compact Borel set with $\mu(\partial B)=0$, then, writing
$\mu \res{B} =\sum_{j=1}^m\delta_{x_j}$ with multiplicities,
for all sufficiently large $n$ we may write
$\mu_n \res{B} =\sum_{j=1}^m\delta_{x_{n,j}}$ so that
$x_{n,j}\to x_j$ for every $1\le j\le m$. 
 
Let $\mathcal M_{\mathrm{fp}} (E)$
be the space of finite point measures on $E$, endowed with the weak
topology. That is,   
$\mu_n\to\mu$ weakly if and only if
$\int_E f\dif \mu_n\to\int_E f\dif \mu$
for every $f\in C_{\mathrm{b}}(E)$. It is also a Polish space.

A \emph{point process} on $E$ is a random element of
$\mathcal M_{\mathrm p}(E)$, equipped with the Borel
$\sigma$-field induced by the vague topology.
The following criterion of Kallenberg characterizes weak convergence
to a simple point process via the probabilities of finding
no  or at least two points in suitable sets.

\begin{lemma}[
{\cite[Theorem~4.15]{KallenbergRandomMeasures}}]
\label{lem:kallenberg-criterion}
Let $\mathcal{E}_n$ and $\mathcal{E}$ be  point processes on $E$, and assume that
$\mathcal{E}$ is simple  almost surely.   Let $\mathscr{R}$ be a dissecting ring of
relatively compact Borel sets \footnote{Here dissecting means
that every relatively compact Borel set can be covered by finitely many sets from the ring and every open set is a countable union of such sets.} such that
$\mathcal{E}(\partial R)=0$ almost surely for every $R\in\mathscr{R}$. Then
\[
 \mathcal{E}_n \xrightarrow[n \to \infty]{\mathrm{law}} \mathcal{E}
 \quad\text{in }\Mloc(E)
\]
if and only if for every $R \in \mathscr{R}$,  
\[ \lim_{n\to\infty}\P(\mathcal{E}_n(R)=0)
 = \P(\mathcal{E}(R)=0) \ , \quad 
 \limsup_{n\to\infty}\P(\mathcal{E}_n(R) \ge 2)
 \leq \P(\mathcal{E}(R) \ge 2) .\] 
\end{lemma}

For the applications below, we take $E=(-\infty,+\infty]$  and  equip it with its usual
topology, such that, $x\mapsto e^{-x}$, with
$e^{-(+\infty)}:=0$, is a homeomorphism from $(-\infty,+\infty]$ onto
$[0,\infty)$.  

Let  $\mathscr{R}_{(-\infty,\infty]}$ be the ring of finite disjoint unions of sets in
\[ 
 \{\varnothing\}\cup\bigl\{(a,b]:a\in\mathbb Q,\ 
 b\in\mathbb Q\cup\{+\infty\},\ a<b\bigr\},
\]
  Then  $\mathscr{R}_{(-\infty,\infty]}$ is a
dissecting  ring of relatively compact Borel
subsets of $(-\infty,+\infty]$. Thus, provided that the limiting point process is almost surely
simple and has no points at any rational location almost surely,
Kallenberg's criterion applies with
$\mathscr R=\mathscr{R}_{(-\infty,+\infty]}$.

The following lemma is a direct consequence of vague convergence.
Its proof, for completeness, is included in  Appendix~\ref{sec:pp-to-count-proof}.

\begin{lemma}\label{lem:pp-to-count}
If 
  $\mu_n\to\mu$ vaguely in $\Mloc((-\infty,\infty])$,  where $\mu$ is simple and has no atom
at $+\infty$, then 
\begin{equation}
    (\mu_n((x,\infty]))_{x \in \mathbb{R}} \xrightarrow{n \to \infty} (\mu((x,\infty]))_{x \in \mathbb{R}}  \label{eq:pp-to-count}
\end{equation}   in  $ D(\mathbb{R},\mathbb Z_{\geq0})$ equipped with the Skorokhod
$J_1$ topology.
\end{lemma}

  For $\mu \in \Mloc(({-}\infty,+\infty])$, and $a \in \mathbb{R}$ we write $\vartheta_a \mu$ the shifted measure,  defined as 
  \[ \vartheta_a \mu(B)=\mu(B-a) \quad \text{ for all Borel set } B. \]  

We say that a point process $\Xi$ is a \emph{decorated Poisson point
process} (also called a Poisson cluster process), and write 
$\Xi\sim\DPPP(\mu,\mathcal D)$ if there exist a Poisson point process $\sum_i\delta_{x_i}\sim\PPP(\mu)$   
and a sequence $(\mathcal D^{(i)})_{i\geq1}$ of i.i.d.\ copies of a
point process $\mathcal D$, independent of $\sum_i\delta_{x_i}$,
such that  
\begin{equation}\label{eq-def-DPPP}
   \Xi \, \overset{\mathrm{law}}{=}  \, \sum_{i \ge 1} \vartheta_{x_i} \mathcal{D}^{(i)}
\end{equation}  
When $\mathcal D=\delta_0$ almost surely, this reduces to a Poisson
point process with intensity $\mu$.

A point process $\Xi$ is \emph{exp-$1$-stable} if, for every
$\alpha,\beta\in\R$ with $e^\alpha+e^\beta=1$,
\begin{equation}\label{eq:exp-one-stability}
 \Xi\overset{\mathrm{law}}{=}
 \vartheta_\alpha\Xi^{(1)}+\vartheta_\beta\Xi^{(2)},
\end{equation}
where $\Xi^{(1)}$ and $\Xi^{(2)}$ are independent copies of $\Xi$. The general LePage decomposition for
stable point processes (see \cite[Section~2]{MaillardStable}) gives 
the Poisson cluster representation below. 
We use Maillard's formulation in \cite{MaillardStable}, where a succinct proof using  elementary methods was given. Comparing with Corollary~3.2 in \cite{MaillardStable}, we shift  
the decoration  so that its rightmost atom is at zero. 
\begin{lemma}[
{\cite[Corollary~3.2]{MaillardStable}}]
\label{lem:maillard-exp-stable}
Let $\Xi$ be an exp-$1$-stable point process in
$\Mloc((-\infty,+\infty])$ such that
$\P(\Xi\ne0)>0$ and $\Xi(\{+\infty\})=0$ almost surely.
Then there exist a constant $C\in(0,\infty)$ and a point process
$\mathcal D$ on $(-\infty,0]$, whose rightmost atom is at zero almost
surely, such that
\begin{equation}\label{eq:maillard-dppp}
 \Xi\sim\DPPP(Ce^{-x}\,\dif x,\mathcal D).
\end{equation}
In particular, for every $a\in\R$,  $\P (\Xi((a,+\infty])=0 )=\exp\{-Ce^{-a}\}$.
\end{lemma}

We finally recall a Chen--Stein bound, which we will apply conditionally
on the interval fragmentation $\mathcal I$.
For integer-valued random variables $X,Y$ and a $\sigma$-field
$\mathcal G$, define their conditional total variation distance by 
\begin{equation}\label{eq:total-variation-distance}
  d_{\mathrm{TV},  \mathcal{G}}\bigl(X,Y\bigr)
  :=   \frac12\sum_{j\in\mathbb Z}\,  \bigl|\P(X=j\mid \mathcal{G})-\P(Y=j \mid \mathcal{G})\bigr|,
\end{equation} 
We write $d_{\mathrm{TV}}(X,Y)$ for the unconditional total variation
distance. 

\begin{lemma}[{\cite[Theorem 1]{ArratiaGoldsteinGordon}}]
\label{lem:chen-stein}
Let $(\xi_\alpha)_{\alpha\in I}$ be a family of Bernoulli random variables with finite index set $I$.   Set 
\[
Y :=\sum_{\alpha\in I}\xi_\alpha \quad \text{ and } \quad 
\lambda=\E[Y ]= \sum_{\alpha\in I} \E[\xi_\alpha ]
\]
For each  $\alpha\in I$, assume a dissociating neighborhood $\mathcal{N}_\alpha\subset I\setminus\{\alpha\}$ is given. That is,  
$\xi_\alpha$ is independent of $(\xi_\beta)_{\beta\notin \bar{\mathcal{N}}_\alpha}$ with $\bar{\mathcal{N}}_{\alpha} := \mathcal N_\alpha \cup \{\alpha\}$.   
Then, we have 
\[
d_{\mathrm{TV} }\bigl(Y,\operatorname{Poi}(\lambda)\bigr)
\le \sum_{\alpha\in I}\sum_{\beta\in \bar{\mathcal{N}}_\alpha} \E[ \xi_\alpha  ]\E[\xi_\beta  ]
+ \sum_{\alpha\in I}
\sum_{ \beta\in \mathcal N_\alpha  }
\E[\xi_\alpha \xi_\beta  ].
\]  
\end{lemma}

 \section{Poissonian limits up to the freezing threshold}
\label{sec:poisson}  
This section is devoted to proving the cases $2\le q\le \theta_*$ of Theorem~\ref{thm:intro-separation-point-process}, Corollary~\ref{thm:three-regimes}, and Theorem~\ref{thm:intro-exact-q-clade-process}. Throughout, we assume that $\theta_*\ge 2$, since otherwise there is no integer $q$ in this range. Our proof strategy is to first reduce the weak convergence of the separation-time extremal process   to the convergence in distribution of suitable counting variables by Lemma \ref{lem:kallenberg-criterion}, and then apply the Chen--Stein method for Poisson approximation Lemma \ref{lem:chen-stein}. Variants of this strategy have also been employed, among others, in \cite{ChenavierStationaryTessellations,ChiariniCiprianiHazra,GhoshKirsebomRoy,KistlerSchertzerSchmidt,MaMaillardMinimumPath}.

Throughout this section we set, for $2 \le q \le \theta_*$,
\[
 \mathcal{W}_q:=W_q\ind{q<\theta_*}+a_*Z\ind{q=\theta_*}.
\] 
Recall $a_*$ is the constant defined in \eqref{eq:def-a*}.
For every $R \in \mathscr{R}_{(-\infty,\infty]}$, define 
\[
 \eta_A^{(n,q)}(R)
 :=\ind{\kappa(q)T_A-m_{n,q}\in R} \ , \ \text{ and hence }
 \ 
 \Xi_{n,q}(R)=\sum_{A\in\binom{[n]}q} \eta_A^{(n,q)}(R).
\]
\subsection{Poisson approximation for separation-time counts} Our goal in this subsection is to establish the following conditional
Poisson approximation.

\begin{lemma} 
\label{prop:ring-count-Poisson}
Fix an integer $2\leq q\leq\theta_*$, and let 
$R  \in \mathscr{R}_{(-\infty,\infty]}$.  
We have 
\begin{align}
  \Lambda_{n,q}(R)
  :=\E[\Xi_{n,q}(R)\mid\mathcal{I}] &\xrightarrow[n \to \infty]{\P}  \frac{\mathcal{W}_q}{q!} \int_R e^{-x}\,\dif x , \text{ and } \label{eq:ring-count-mean-limit} \\
  d_{\mathrm{TV}, \mathcal{I}} \bigl (
 \Xi_{n,q}(R) , 
  \Poi(\Lambda_{n,q}(R))  \bigr )
  &\xrightarrow[n \to \infty]{\P} \  0 .\label{eq:ring-count-TV}
\end{align}
\end{lemma}

Before proving Lemma~\ref{prop:ring-count-Poisson}, we record a basic convergence result that will be used both in its proof and in subsequent arguments. Recall the time window $t_{n,q}(x)$ defined  in  \eqref{eq:unified-sampling-window}.  

\begin{lemma} 
\label{lem:point-process-inputs}
For every fixed integer $2\leq q\leq\theta_*$ and $x\in\R$, we have 
\begin{equation}
 n^qS_q(t_{n,q}(x))
 \xrightarrow[n \to \infty]{\P} e^{-x}\mathcal{W}_q  
  \quad \text{ and } \quad 
 nX_1(t_{n,q}(x))
  \xrightarrow[n \to \infty]{\P} 0  .
 \label{eq:point-process-inputs}
\end{equation} 
\end{lemma}

\begin{proof}[Proof of Lemma \ref{prop:ring-count-Poisson}]
By \eqref{eq:conditional-unseparated-probability}, and using the convention $S_q( \infty):=0$, we have $\P(T_A \in (a,b]\mid\mathcal{I})=S_q(a)- S_q(b)$ for every $A \in \binom{[n]}{q}$. For  $R=\bigcup_{j=1}^m(a_j,b_j]$ with pairwise disjoint $(a_j,b_j]$, the linearity gives 
\[
 \Lambda_{n,q}(R) = \binom nq\sum_{j=1}^m
 \bigl[S_q(t_{n,q}(a_j))-S_q(t_{n,q}(b_j))\bigr]
\]
   The finite collection of endpoint limits in
Lemma~\ref{lem:point-process-inputs}, together with
$\binom nq/n^q\to1/q!$, proves
\eqref{eq:ring-count-mean-limit}.

We next condition on the  interval fragmentation process $\mathcal{I}$ and apply
Lemma~\ref{lem:chen-stein} to $ \Xi_{n,q}(R)=\sum_{A\in\binom{[n]}q} \eta_A^{(n,q)}(R)$. Since indicators indexed by disjoint label sets are conditionally independent by Lemma \ref{lem:genealogical-paintbox}, we can use
\[
 \mathcal N_A
 := \Bigl\{  A'\in\binom{[n]}q:
 A'\ne A,\ A'\cap A\ne\varnothing \Bigr\} 
\]
as the dependency neighborhood.     
Put
$\underline{a}:= \inf R \in (-\infty,+\infty]$.  Then 
\[
 \E \bigl[ \eta_A^{(n,q)}(R)\mid\mathcal{I} \bigr]  \le  \P(  T_A>t_{n,q}(\underline a) \mid\mathcal{I}) 
 \leq S_q(t_{n,q}(\underline{a})).
\]  
If $|A\cap A'|=q-\ell$ with $\ell\in\{1,\ldots,q-1\}$ and both marked
events occur, then all $q+\ell$ labels in $A\cup A'$ lie in one block at
$t_{n,q}(\underline{a})$.  Consequently, 
\begin{align}
   \E \bigl[  \eta_A^{(n,q)}(R)\eta_{A'}^{(n,q)}(R)\mid\mathcal{I}\bigr] & \le \P \bigl( T_{A \cup A'} > t_{n,q}(\underline{a}) \mid \mathcal{I} \bigr) \\
   &
 \leq S_{q +\ell}(t_{n,q}(\underline{a}))  \ 
 \leq S_q(t_{n,q}(\underline{a}))\, [X_1(t_{n,q}(\underline{a}))]^{\ell}.
\end{align} 
Applying Lemma~\ref{lem:chen-stein}  we obtain
\begin{align*}
&d_{\mathrm{TV},\mathcal{I}}
\bigl(\Xi_{n,q}(R),\Poi(\Lambda_{n,q}(R))\bigr) \\
&\quad  \le
\!\!\!\sum_{A\in\binom{[n]}q, A'\in\overline{\mathcal N}_A} \!\!\!
\E\bigl[\eta_A^{(n,q)}(R)\mid\mathcal{I}\bigr]
\E\bigl[\eta_{A'}^{(n,q)}(R)\mid\mathcal{I}\bigr] \,+\!\!\! 
\sum_{A\in\binom{[n]}q, A'\in \mathcal{N}_A} \!\!\!
\E\bigl[
    \eta_A^{(n,q)}(R)\eta_{A'}^{(n,q)}(R)
    \mid\mathcal{I}
\bigr] \\
& \quad\lesssim_{q}
 \Bigl\{ 
    \frac{
        \bigl(n^qS_q(t_{n,q}(\underline a))\bigr)^2
    }{n}
    +
    n^qS_q(t_{n,q}(\underline a))
    \sum_{\ell=1}^{q-1}
    \bigl(nX_1(t_{n,q}(\underline a))\bigr)^{\ell} \Bigr\}  . 
\end{align*} 
Here we used that for each $\ell\in\{1,\ldots,q-1\}$, the number of ordered
pairs $(A,A')$ satisfying $|A\cap A'|=q-\ell$ is
$\binom{n}{q}\binom{q}{\ell}\binom{n-q}{\ell}
\lesssim_q n^{q+\ell}$.
Combining Lemma~\ref{lem:point-process-inputs} with the last inequality,   the required convergence \eqref{eq:ring-count-TV}   follows.
\end{proof}

\begin{proof}[Proof of Lemma \ref{lem:point-process-inputs}]
Suppose first that $q<\theta_*$.  By
Lemma~\ref{lem:brw-inputs}\ref{prop:martingale},
$W_q(t)\to W_q$ almost surely, and
\[
 n^qS_q(t_{n,q}(x))
 =e^{-x}W_q(t_{n,q}(x))
 \xrightarrow[n \to \infty]{\mathrm{a.s.}}  e^{-x}W_q.
\]
Since $q < \theta_*$, we have  $ \Delta(q)= q \kappa'(q)- \kappa(q)>0$.  
 Hence, by differentiating at  $a=1$ and using continuity, there exists  $a>1$ such that
$\kappa(a q )>a\kappa(q)$.  
Using 
$X_1(t)^{aq}\leq S_{aq}(t)$ and
\eqref{eq:exponent-S}, we get  
\begin{align*}
 \E\bigl( [nX_1(t_{n,q}(x)) ]^{aq} \bigr)
 &\leq n^{qa}\E[ S_{qa}(t_{n,q}(x)) ] \\
 &  =  e^{-x\kappa(qa)/\kappa(q)}
 \, n^{q\{a-\kappa(qa)/\kappa(q)\}}  \xrightarrow{n \to \infty}0.
\end{align*} 

It remains to consider $q=\theta_*$.  In this case,
\[
 n^{\theta_*}S_{\theta_*}(t_{n,\theta_*}(x))
 =e^{-x}\sqrt{\log n}\,
 W_{\theta_*}(t_{n,\theta_*}(x)).
\]
Since $t_{n,\theta_*}(x)/\log n\to\theta_*/\kappa_*$,
the Seneta--Heyde scaling \eqref{eq:SH-scaling} gives
\[
 \sqrt{\log n}\,W_{\theta_*}(t_{n,\theta_*}(x))
 \xrightarrow[n \to \infty]{\P}
 \sqrt{\frac{\kappa_*}{\theta_*}}
 \sqrt{\frac{2}{\pi\sigma_*^2}}\frac{Z}{\theta_*}
 =a_*Z.
\]  
For the largest fragment, put $v_*=\kappa_*/\theta_*$.  By
\eqref{eq:largest-fragment-input},
$
 \log X_1(t)=- \frac{\kappa_*}{\theta_*} t-[\frac{3}{2\theta_*}+o_{\P}(1)]\log t
  $.
For $t=t_{n,\theta_*}(x)$, the definition of the sampling window gives
$
 \frac{\kappa_*}{\theta_*} t=\log n-\frac{1}{2\theta_*}\log\log n+\frac{x}{\theta_*}$ and $
  \log t=\log\log n+O(1)$.  
Substitution therefore yields
\[
 \log\{nX_1(t_{n,\theta_*}(x))\}
 =-\Bigl[   \frac{1}{\theta_*}+o_{\P}(1)\Bigr] \log\log n
\]
In particular, $nX_1(t_{n,\theta_*}(x))\to0$ in probability, which
concludes the proof.
\end{proof}

\subsection{From count convergence to process limits}
We shall prove a slightly stronger result.  We first introduce some
notation.  Let $E$ be a Polish space, and let $\mathcal P(E)$ denote the
space of Borel probability measures on $E$, equipped with a metric that induces the weak convergence topology. Recall that 
  for an $E$-valued random element $\xi$ and a
$\sigma$-field $\mathcal G$, we denote by $\mathcal{L}(\xi \mid \mathcal{G})$ or $\mathcal{L}_{ \xi \mid \mathcal{G}}$ or  the regular conditional distribution of
$\xi$ given $\mathcal G$, i.e., 
\[ \mathcal{L}_{ \xi \mid \mathcal{G}} (\omega, B) =\P( \xi \in B \mid \mathcal{G}) (\omega) . \]
We view $\mathcal{L}_{ \xi \mid \mathcal{G}}$  as a random element of $\mathcal P(E)$, and  write
$ \mathcal{L}_{ \xi_n \mid \mathcal{G}} 
 \xrightarrow[n\to\infty]{\P}
\mathcal{L}_{ \xi \mid \mathcal{G}}
$
if these $\mathcal{P}(E)$-valued random elements converge in probability. This is  equivalent to  
\[
 \E[f(\xi_n)\mid\mathcal G]
 \xrightarrow[n\to\infty]{\P}
 \E[f(\xi)\mid\mathcal G] \quad \text{ for every } f \in C_b(E).
\]

Parts~(i) and~(ii) of Theorem~\ref{thm:intro-separation-point-process} thus are immediate consequences of the following proposition. 
\begin{proposition} 
\label{prop:point-process}
Under assumptions
\eqref{eq:dislocation-integrability},
\eqref{eq:entropy-dominance}, and
\eqref{eq:no-sudden-extinction}, fix an integer
$2 \le q  \le \theta_*$.
There exists a locally finite point process $\Xi_q$ on
$(-\infty,+\infty]$, coupled with $\mathcal W_q$, such that
$
  \mathcal L(\Xi_q\mid\mathcal W_q)
  =
  \PPP(
    \frac{\mathcal W_q}{q!}e^{-x}\,\dif x)$
and
\begin{equation}\label{eq:Cox-limit}
  \mathcal{L}_{\Xi_{n,q} \mid \mathcal I } 
 \xrightarrow[n\to\infty]{\P}
 \mathcal{L}_{\Xi_q\mid\mathcal W_q}  \quad  \text{ in } \ \mathcal{P}(\Mloc((-\infty,+\infty])).
\end{equation} 
In particular, we have 
$\, \Xi_{n,q}\xrightarrow[n\to\infty]{\mathrm{law}}\Xi_q$ in
$\Mloc((-\infty,+\infty])$.
\end{proposition}

\begin{proof} 
 By subsequence convergence criterion, it suffices to show for an arbitrary subsequence $\mathbb{N}' \subset \mathbb{Z}_{\ge 0}$ there is a   further
subsequence  $\mathbb{N}'' \subset \mathbb{N}'$ such that for almost every $\omega$. 
\begin{equation}
  \label{eq-conv-1}
    \mathcal{L}_{\Xi_{n,q} \mid \mathcal I } (\omega) \, \xrightarrow[n \in \mathbb{N}'']{  n \to \infty} \,    \mathcal{L}_{\Xi_{q} \mid \mathcal W_q }  (\omega) \quad  \text{ in }\ \mathcal{P}(\Mloc((-\infty,+\infty])).
\end{equation}  

For each fixed $\omega$, since  
$\mathcal{L}_{\Xi_{q} \mid \mathcal W_q }(\omega)$ is the law of a PPP with Lebesgue continuous intensity measure,    it is 
simple  and assigns no mass to $\partial R$ for every
$R\in\mathscr R_{(-\infty,\infty]}$. 
By Kallenberg's criterion  Lemma~\ref{lem:kallenberg-criterion},     \eqref{eq-conv-1} follows directly once we show that for almost every $\omega$,  
\begin{equation}
  \label{eq-conv-2}
   \mathcal{L}_{\Xi_{n,q} \mid \mathcal I } (\omega, \{\mu:\mu(R)=\ell\})
   \, \xrightarrow[n \in \mathbb{N}'']{  n \to \infty} \,    \mathcal{L}_{\Xi_{q} \mid \mathcal W_q } (\omega,\{\mu:\mu(R)=\ell\}) 
\end{equation} 
 for every integer  
 $ \ell \ge 0$ and $ R\in\mathscr R_{(-\infty,\infty]}$.  
By definition, we have 
\begin{align}
   \mathcal{L}_{\Xi_{n,q} \mid \mathcal I } (\omega, \{\mu:\mu(R)=\ell\})  &=\P(\Xi_{n,q}(R) = \ell   \mid \mathcal{I} ) (\omega) \\
    \mathcal{L}_{\Xi_{q} \mid   \mathcal{W}_q} (\omega, \{\mu:\mu(R)=\ell\})  &=\P( \Poi( \tfrac{1}{q!}\mathcal{W}_q(\omega)  {\textstyle \int}_R e^{-x} \dif x    ) = \ell  \big).
\end{align}  
Lemma~\ref{prop:ring-count-Poisson} and
continuity of the Poisson  distribution in its parameter give 
\begin{align*}
& \frac{1}{2} \sum_{\ell \ge 0} | \P(\Xi_{n,q}(R) = \ell \mid \mathcal{I} ) - \P(\Xi_{q}(R) = \ell   \mid \mathcal{W}_q ) | \\ 
 &\ \leq
 d_{\mathrm{TV},\mathcal I}\big(
   \Xi_{n,q}(R), \Poi(\Lambda_{n,q}(R))
  \big)
 + d_{\mathrm{TV},\mathcal{I}} \big( \Poi(\Lambda_{n,q}(R)), \Poi( \tfrac{1}{q!}\mathcal{W}_q  {\textstyle \int}_R e^{-x} \dif x    )  \big)
 \xrightarrow[n\to\infty]{\P}0.
\end{align*} 
Since
$\mathscr R_{(-\infty,\infty]}$ is countable, the preceding convergence
and a diagonal argument yield a  deterministic subsequence $\mathbb{N}''$ of $\mathbb{N}'$ such that,  
almost surely 
\[  \sum_{\ell \ge 0} | \P(\Xi_{n,q}(R) = \ell \mid \mathcal{I} ) - \P(\Xi_{q}(R) = \ell   \mid \mathcal{W}_q ) |  \, \xrightarrow[n \in \mathbb{N}'']{  n \to \infty} \,     0   \]
for every $R\in\mathscr R_{(-\infty,\infty]}$.  Thus \eqref{eq-conv-2} follows. This completes the proof.
\end{proof}

  \smallskip
\begin{proof}[Proof of Corollary~\ref{thm:three-regimes} for
$q\leq\theta_*$]  
Let $q=r+1$.  The atoms of $\Xi_{n,q}$, listed in decreasing order with
multiplicities, are
$\kappa(q)D_{n,r}^{(j)}-m_{n,q}$,
$1\leq j\leq\binom{n}{q}$.
Since $\Xi_q$ is simple, has no atom at $+\infty$, and has infinitely many
atoms, Proposition~\ref{prop:point-process} and the continuous-mapping
theorem for the first $k+1$ ordered atoms give
\[
 \big(
  \kappa(q)D_{n,r}^*-m_{n,q},
  (\kappa(q)(D_{n,r}^{(j)}-D_{n,r}^{(j+1)}))_{j=1}^k
 \big)
 \xrightarrow[n\to\infty]{\mathrm{law}}
 \big(
  \mathsf A_q^{(1)},(\mathsf\Delta_{q,j})_{j=1}^k
 \big).
\] 
This is the asserted joint convergence since $\gamma_q=\kappa(q)$ for
$q\leq\theta_*$.

It remains to identify the limit. Let $0<\Gamma_1<\Gamma_2<\cdots$ be the arrival times of a rate-one
Poisson process, independent of $\mathcal W_q$.  Then we have   
\begin{equation}\label{eq:poissonian-ordered-atoms}
 (\mathsf A_q^{(j)})_{j\geq1}
 \overset{\mathrm{law}}{=}
 \bigl(  \log(\mathcal W_q/q! ) -\log\Gamma_j \bigr)_{j\geq1}.
\end{equation} 
Since $-\log\Gamma_1$ is a standard Gumbel random variable independent of
$\mathcal W_q$, \eqref{eq:poissonian-ordered-atoms} gives  
\[
 \mathsf A_q^{(1)}
 \overset{\mathrm{law}}{=}
 G+\log(\mathcal W_q/q!).
\] 
This proves the first two cases
for the largest atom.

For the   gaps, set $\rho_j:=\Gamma_j/\Gamma_{j+1}$.  For every $k\geq1$,
conditional on $\Gamma_{k+1}$, the variables
$\Gamma_1/\Gamma_{k+1},\ldots,\Gamma_k/\Gamma_{k+1}$ are the order
statistics of $k$ independent uniform random variables on $(0,1)$.
By making change of variables, we obtain that  $(\rho_1,\ldots,\rho_k)$ has the joint density function 
\[
 \prod_{j=1}^k j\rho_j^{j-1},
 \qquad (\rho_1,\ldots,\rho_k)\in(0,1)^k.
\]
Thus the $\rho_j$ are independent with
$\rho_j\sim\operatorname{Beta}(j,1)$.  By
\eqref{eq:poissonian-ordered-atoms}, we get 
\[
 (\mathsf\Delta_{q,j})_{j=1}^{k}
 \overset{\mathrm{law}}{=}
\big(\log(\Gamma_{j+1}/\Gamma_j)
 =-\log\rho_j\big)_{j=1}^{k}.
\]
Hence the gaps are mutually independent and
$\mathsf\Delta_{q,j}\sim\operatorname{Exp}(j)$, as claimed.
\end{proof}

   \smallskip

\begin{proof}[Proof of Theorem~\ref{thm:intro-exact-q-clade-process}
for $q\leq\theta_*$]  
By \eqref{eq:pp-to-count}, the map
$\mu\mapsto(\mu((x,+\infty]))_{x\in\R}$ is continuous at every simple
point measure with no atom at $+\infty$.  Since $\Xi_q$ almost surely has
these properties, Proposition~\ref{prop:point-process} and the continuous
mapping theorem give 
\begin{equation}\label{eq-pp-to-count-2}
   \bigl(\Xi_{n,q}((x,+\infty])\bigr)_{x\in\R}
   \xrightarrow[n\to\infty]{\mathrm{law}}
   \bigl(\Xi_q((x,+\infty])\bigr)_{x\in\R} \quad \text{ in } \, D(\R,\mathbb Z_{\geq0}) .
\end{equation}

It remains to replace the number of unseparated $q$-sets by the number of
clades of size exactly $q$. 
Fix $a\geq1$. 
Note that the two counts $\Xi_{n,q}((x,+\infty])$ and $N_{n,q}(x)$ agree
for every $|x|\leq a$, unless $\Pi(t_{n,q}(-a))|_{[n]}$ contains a clade
of size at least $q+1$. 
Conditional on the interval fragmentation process  $\mathcal I$, the probability of this exceptional event
is at most
\[
 \min \bigl\{  n^{q+1} S_{q+1}(t_{n,q}(-a)) ,1
  \bigr\}\leq   \min \bigl\{   n^qS_q(t_{n,q}(-a)) n X_1(t_{n,q}(-a)) , 1 \bigr\} 
 \xrightarrow[n\to\infty]{\P}0,
\]
where we used $S_{q+1}(t)\leq X_1(t)S_q(t)$ and
Lemma~\ref{lem:point-process-inputs}. Thus, by dominated convergence,
\[
 \P(N_{n,q}(x)=\Xi_{n,q}((x,+\infty])
 \text{ for every }|x|\leq a)
 \xrightarrow[n\to\infty]{}1.
\]
Combining with \eqref{eq-pp-to-count-2},  this proves the asserted
weak convergence for $(N_{n,q}(x))_{x\in\R}$.

Finally, conditional on $\mathcal W_q$,   $\Xi_q$ is a
Poisson point process with intensity $(\mathcal W_q/q!)e^{-x}\,\dif x$.
Given that there are $m$ atoms in $(x,+\infty]$, each lies in
$(x+t,+\infty]$ independently with probability $e^{-t}$ for $t>0$.
This gives the claimed conditional pure-death description.
\end{proof}

\section{Clustered limits above the freezing threshold}
\label{sec:frozen}

Throughout this section, we fix an integer $q>\theta_*$ and assume the
nonlattice condition \eqref{non-lattice-cond}, in addition to the standing
assumptions.  We prove the frozen cases of
Theorem~\ref{thm:intro-separation-point-process},
Corollary~\ref{thm:three-regimes}, and
Theorem~\ref{thm:intro-exact-q-clade-process}, in this order.

At the time $t_{n,q}(x)$, the largest fragments have masses of order
$n^{-1}$.  Their sampled occupancies therefore remain nondegenerate, and
one fragment may produce several late separation times.  This is why the
Poisson approximation for individual $q$-sets used in
Section~\ref{sec:poisson} no longer applies.  We instead use the classical
Poissonization trick: replacing the fixed sample size by an independent
Poisson variable makes the occupancies of
distinct fragments independent.  See Joseph~\cite[Section~3.2]{Joseph}
and the references therein.

\subsection{Extremal fragments at the frozen scale}

Recall the sampling window $t_{n,q}(x)$ in
\eqref{eq:unified-sampling-window}. Since it is indeed \emph{independent} of $q$ when $q > \theta_*$, we write $t_n(x)$ in this section.
For each $x\in\R$, define the
\emph{fragment extremal process} at time $t_{n}(x)$ by
\[
 \mathcal{E}^{\mathrm{fr}}_{n,x}:=\sum_{i\geq1}\delta_{nX_i(t_{n}(x))}.
\] 
We view it as a point process in  $\Mloc((0,\infty])$. 
The object in this section is to show the weak convergence of this {fragment extremal process}, which indeed is implied by  Lemma~\ref{lem:brw-inputs} (\ref{input:Madaule}).

We first describe its limit. Recall 
$v_*:=\kappa_*/\theta_*=\kappa'(\theta_*)$.   Let $\mathcal E^V$ be the limiting point process introduced in
Lemma~\ref{lem:brw-inputs}\textup{(\ref{input:Madaule})}, jointly defined with $Z$.   For
$x\in\mathbb R$, put
\[
 f_x(y):=
 v_*^{3/(2\theta_*)}
 \exp \bigl(  {[y-x]}/{\theta_*} \bigr) \ \quad  y\in\mathbb R,
\]
and let $ \mathcal E_x^{\mathrm{fr}}$ to be the pushforward measure of $\mathcal{E}^V$ under $f_x$, i.e.,
\begin{equation}\label{eq:mass-process-definition}
 \mathcal E_x^{\mathrm{fr}}
 :=\mathcal E^V\circ f_x^{-1}
 =\sum_{i\geq1}\delta_{\lambda_i(x)} 
\end{equation}
where $ 
 \infty>\lambda_1(x)\geq\lambda_2(x)\geq\cdots$
are the atoms of $\mathcal E_x^{\mathrm{fr}}$, listed with
multiplicity.  Here $(\mathcal{E}^{\mathrm{fr}}_x)_{x \in \mathbb{R}}$  are coupled together and note that   $
 \lambda_i(x+h)
 =e^{-h/\theta_*}\lambda_i(x),
$ for $ i\geq1$, $ x,h\in\mathbb R. $

For each $x$ the law of $\mathcal{E}^{\mathrm{fr}}_{x}$ can be   equivalently described as follows: 
Conditional on $Z$, let $\sum_{j \ge 1}\delta_{y_j}$ be a Poisson point process on
$(0,\infty]$ with intensity
$C_V\theta_*v_*^{3/2}e^{-x}
Z  y^{-\theta_*-1}\,\dif y$. Let $(\mathcal{D}_{j}=\sum_{k \ge 1} \delta_{d_{j,k}})_{j \ge 1}$ be i.i.d. copies of $\mathcal{D}=\sum_{k \ge 1} \delta_{d_k}$, and independent of $\sum_{j \ge 1}\delta_{y_j}$. Then, we have 
\begin{equation}\label{eq:cluster-mass-representation}
 \mathcal E^{\mathrm{fr}}_x
 \overset{\mathrm{law}}=
 \sum_{j,k}
 \delta_{y_j \exp(d_{j,k}/\theta_*)}.
\end{equation}  

\begin{lemma} \label{lem:extremal-masses}
Fix $x\in\R$ . Let  $g$ be a  measurable function on $(0,\infty)$
satisfying $ \sup_{y \in (0,1]} y^{-\beta} |g(y)| <\infty $ for some $\beta>\theta_{*}$.
Then,  almost surely
$ \langle |g|,  \mathcal E^{\mathrm{fr}}_x \rangle <\infty $.
If in addition $g$ is continuous, then
\begin{equation}\label{eq:mass-process-limit}
 \bigl( 
  \mathcal{E}^{\mathrm{fr}}_{n,x}, Z(t_{n}(x)), \langle g,  \mathcal E^{\mathrm{fr}}_{n,x} \rangle
  \bigr)
 \xrightarrow[n\to\infty]{\mathrm{law}}
\bigl(  
  \mathcal{E}^{\mathrm{fr}}_x,Z,  \langle g,  \mathcal E^{\mathrm{fr}}_{x} \rangle
 \bigr) \quad \text{ in } \Mloc((0,\infty]) \times \R^2 .
\end{equation}   
\end{lemma}

\begin{proof} 
 For each $x \in \mathbb{R}$,  define the map 
\[
 f_{n,x}(y):=   
 [ \tfrac{\log n}{t_n(x)} ]^{3/(2\theta_*)}  \exp({[y-x]/\theta_*}) \ , \quad  y\in\mathbb R.
\]
Then we have 
$  nX_i(t_{n}(x)) = f_{n,x} (  \frac32\log t_{n}(x)-V_i  ( 
  t_{n}(x) )  ) $. 
Together with  
\eqref{eq:mass-process-definition},  we get  
\begin{equation}\label{eq:pushforward-id-1}
   \mathcal E^{\mathrm{fr}}_{n,x}
   =\mathcal E^V_{t_n(x)}\circ f_{n,x}^{-1}
   \quad \text{and} \quad
   \mathcal E^{\mathrm{fr}}_x=\mathcal E^V\circ f_x^{-1}.
\end{equation} 

Since $\beta/\theta_*>1$, Lemma~\ref{lem:brw-inputs}
\textup{(\ref{input:Madaule})} gives
 \[ 
 \int_{(0,1]}|g(y)|\,
 \mathcal E_x^{\mathrm{fr}}(\dif y)
 \lesssim \int_{(0,1]} y^{\beta}  \mathcal E^{\mathrm{fr}}_{ x}(\dif y) \lesssim   e^{-\beta x/\theta_*} 
 \langle\exp_{\beta/\theta_*},\mathcal E^V\rangle < \infty \ \text{ a.s..} \] 
  On the other hand, by \eqref{eq:cluster-mass-representation},  $\mathcal E_x^{\mathrm{fr}}$ has only finitely many atoms in $[1,\infty)$ and $\mathcal E_x^{\mathrm{fr}}(\{\infty\})=0$,   the
contribution of these atoms to
$\langle |g|,\mathcal E_x^{\mathrm{fr}}\rangle$ is finite. Thus we have $\P( 
 \langle |g|,  \mathcal E^{\mathrm{fr}}_x \rangle <\infty  )=1$.

Suppose now that $g$ is continuous.  We first prove
\eqref{eq:mass-process-limit} under the additional assumption that
$g$ is bounded. To use  
Lemma~\ref{lem:brw-inputs}~\textup{(\ref{input:Madaule})}, we  choose $1<\eta<\beta/\theta_*$ and set
\[
 h_{n,x}(y):=e^{-\eta y}g(f_{n,x}(y))
 \quad \text{and} \quad
 h_x(y):=e^{-\eta y}g(f_x(y)).
\]
The boundedness of $g$ and the bound $|g(y)|\leq K y^\beta$ imply that
$h_x(y)\to0$ as $y\to\pm\infty$. Thus
$h_x\in\mathrm{BUC}(\mathbb R)$.  Put
$
 b_n(x):=\frac32\log (
 \frac{\log n}{v_*t_n(x)}
 )$
Then $b_n(x)\to0$ and 
 $ h_{n,x}(y)
 =e^{\eta b_n(x)}h_x(y+b_n(x))$.
Consequently, $h_{n,x}\in\mathrm{BUC}(\mathbb R)$ and
$\|h_{n,x}-h_x\|_\infty\to0$. Applying
Lemma~\ref{lem:brw-inputs}~\textup{(\ref{input:Madaule})} to the function $g(f_{n,x}(y)) =  h_{n,x}(y) \exp_{\eta}(y)$   therefore yields
\[
 \Bigl(\mathcal E^V_{t_n(x)},Z(t_n(x)),
 \langle g\circ f_{n,x},\mathcal E^V_{t_n(x)}\rangle\Bigr)
 \xrightarrow[n\to\infty]{\mathrm{law}}
 \Bigl(\mathcal E^V,Z,
 \langle g\circ f_x,\mathcal E^V\rangle\Bigr).
\]
in $\Mloc((-\infty,\infty]) \times \R^2$ .
Combining this with the
pushforward identities, the fact $f_{n,x} \to f_x$ locally uniformly, \eqref{eq:pushforward-id-1} and the
continuous-mapping theorem proves \eqref{eq:mass-process-limit} when
$g$ is bounded.

For a general continuous $g$, choose a continuous function
$\chi_M:(0,\infty)\to[0,1]$ such that $\chi_M=1$ on $(0,M]$ and
$\chi_M=0$ on $[M+1,\infty)$, and set $g_M:=g\chi_M$.
Then $g_M$ is bounded and continuous and the preceding argument applies to $g_M$. 
Since 
$\mathcal E_{n,x}^{\mathrm{fr}}\Rightarrow
\mathcal E_x^{\mathrm{fr}}$ and 
  $\mathcal E_x^{\mathrm{fr}}(\{\infty\})=0$,  we have 
$ \P (
  \mathcal E_{n,x}^{\mathrm{fr}}((M,\infty])>0
 ) \to 0$ as $n \to \infty$ then $M \to \infty$. 
On the complementary event, it holds 
$\langle g_M,\mathcal E_{n,x}^{\mathrm{fr}}\rangle
 =\langle g,\mathcal E_{n,x}^{\mathrm{fr}}\rangle$.
The same argument applies to $\mathcal E_x^{\mathrm{fr}}$.
Letting $M\to\infty$ therefore proves
\eqref{eq:mass-process-limit} for every continuous $g$ satisfying the
assumption of the lemma. 
\end{proof}
 
We next record a lemma that will be used later. For $\lambda\geq0$, let
\begin{equation}\label{eq:occupancy-piq}
 \pi_q(\lambda):=\P(\Poi(\lambda)<q)
 =e^{-\lambda}\sum_{k=0}^{q-1}\frac{\lambda^k}{k!}.
\end{equation}
Recall $\mathcal{D}= \sum_{k \ge 1} \delta_{d_{k}}$. We further define
\begin{equation}\label{eq:frozen-constant}
 C_q^\star
 :=C_V\theta_*v_*^{3/2}
 \int_0^\infty
 \Bigl[  
  1-\E\prod_{k \ge 1}
  \pi_q \bigl( ye^{d_k/\theta_*} \bigr)  
  \Bigr]y^{-\theta_*-1}\,\dif y.
\end{equation}

\begin{lemma}\label{lem:Cq-finite}
The constant $C_q^\star$ belongs to $(0,\infty)$.
\end{lemma}

\begin{proof}
Write $\mathcal E^{\mathrm{fr}}_x=\sum_{k\geq1}\delta_{\lambda_k(x)}$ and,
conditional on $\mathcal E^{\mathrm{fr}}_x$, attach independent variables
$\mathsf N_k\sim\Poi(\lambda_k(x))$. Since $q>\theta_*$,
Lemma~\ref{lem:extremal-masses} gives, almost surely,
\[
 \sum_{k\geq1}
 \P(\mathsf N_k\geq q\mid\mathcal E^{\mathrm{fr}}_x)
 \leq\frac1{q!}\int_0^\infty \lambda^q
 \mathcal E^{\mathrm{fr}}_x(\dif \lambda)  <\infty.
\]
The   Borel--Cantelli lemma yields that only finitely many atoms
have $\mathsf N_k\geq q$ almost surely. 

Set
$
 p_q(y):=1-\E\prod_{d\in\mathcal D}\pi_q(ye^{d/\theta_*}).$
By \eqref{eq:cluster-mass-representation} and the independent marking
theorem, conditional on $Z$, the leaders $y_j$ whose clusters contain an atom
with mark at least $q$ form a Poisson point process with intensity
\[
 C_V\theta_*v_*^{3/2}e^{-x}Zp_q(y)y^{-\theta_*-1}\,\dif y.
\]
Its total intensity is $C_q^\star Ze^{-x}$. Its number of points is bounded
by the number of marked atoms in $\mathcal E^{\mathrm{fr}}_x$   and is therefore finite almost surely.
Since $Z\in(0,\infty)$  we conclude that $C_q^\star<\infty$.

Finally, since $\mathcal D$ contains an atom at zero almost surely and 
$0\leq\pi_q\leq1$, we have 
\[
 \prod_{d\in\mathcal D}\pi_q(ye^{d/\theta_*})\leq\pi_q(y),
 \qquad
 p_q(y)\geq1-\pi_q(y)>0.
\]
It follows from \eqref{eq:frozen-constant} that $C_q^\star>0$.
\end{proof}

\subsection{Separation-time extremal process} This subsection is devoted to the proof of the following Proposition, which implies  
Theorem~\ref{thm:intro-separation-point-process} for the case $q>\theta_*$.

\begin{proposition} 
\label{thm:frozen-separation-process}
There exists a point process $\Xi_q$ on $(-\infty,\infty]$ with $\Xi_q(\{+\infty\})=0$ almost surely, jointly defined with $Z$,  such that  
\begin{equation}\label{eq:frozen-point-process-limit}
 (\Xi_{n,q},Z)
 \xrightarrow[n\to\infty]{\mathrm{law}}
 (\Xi_q,Z)
 \quad\text{in }\Mloc((-\infty,+\infty])\times\R.
\end{equation} 
The law of the limit $(\Xi_q,Z)$ is described as follows. There is a point process $\mathcal D_q^{\star}$ on
$(-\infty,0]$,  whose rightmost atom is at zero almost
surely, such that 
\begin{equation}\label{eq:frozen-SDPPP}
 \mathcal{L}_{ \Xi_q \mid Z }
= \DPPP(
  C_q^\star Z e^{-x}\,\dif x,  \mathcal D_q^{\star}
 ).
\end{equation}   
\end{proposition}

Before giving the proof, let us introduce some notation. 
Given a   measure $\mu$ on $(-\infty,\infty]$ and a measurable subset $B$, we write $\mu\res{B}$ the measure obtained by restricting $\mu$ on $B$.  Henceforth, let $\mathsf{N}$ be a rate-one PPP on $[0,\infty)$ independent of the fragmentation processes $(\Pi, \mathcal{I})$. Set 
\begin{equation}\label{eq:residual-separation-process}
 \mathsf{N}_{\lambda}:= \mathsf{N}([0,\lambda]) \quad \text{ and }\quad \mathcal R_{q,\lambda}
 := \sum_{A\in\binom{[\mathsf{N}_{\lambda}]}q}\delta_{\kappa_*T_A}, \quad \text{ for } \lambda \ge 0 ,
\end{equation}
where the sum is understood to be zero measure when $\mathsf N_\lambda<q$. 
Moreover, for every $a \in \mathbb{R}$ and  non-negative measurable function $f$ on $(-\infty,\infty]$, define 
\begin{equation}\label{eq:frozen-mark-kernel}
 \Phi_{q,f}^{(a)}(\lambda)
 :=   \E \Bigl[\exp \bigl\{  
  -\langle f, \vartheta_a  \mathcal R_{q,\lambda} \rangle 
  \bigr\}   \Bigr] \in (0,1 ] \ , \quad  \lambda \ge 0. 
\end{equation} 
Finally let
$(\mathsf N^{(i)},\Pi^{(i)}, \mathcal{I}^{(i)})_{i\geq1}$ be i.i.d. copies of
$(\mathsf N,\Pi,\mathcal{I} )$. For each $i\geq1$,
define $(\mathcal R^{(i)}_{q,\lambda})_{\lambda\geq0}$ from
$(\mathsf N^{(i)},\Pi^{(i)})$ according to
\eqref{eq:residual-separation-process}.

\begin{proof}[Proof of Proposition \ref{thm:frozen-separation-process}]
The proof is divided into five steps. In Steps 1,2  we prove the weak convergence of the following Poissonized
extremal process
\[
 \Xi_{n,q}^{\mathrm{Poi}}
 :=\sum_{A\in\binom{[\mathsf{N}_{n}]}q}
 \delta_{\kappa_*T_A-m_{n,q}} \,,
\]
 and identify its weak limit $\Xi_q$.  
Step~3 removes the Poissonization. In Steps~4 and~5, we identify the conditional law of $\Xi_q$
given $Z$ and prove \eqref{eq:frozen-SDPPP}.

\medskip
\noindent \underline{{Step 1}.} 
We first claim that, for every fixed
$a\in\R$, conditionally on $\mathcal I(t_n(a))  $, it holds 
\begin{equation} 
 \mathcal{L} \Bigl\{ \Xi_{n,q}^{\mathrm{Poi}}\res{(a,+\infty]} \,\Bigm|\, \mathcal I(t_n(a)) 
  \Bigr\} 
\,\overset{\mathrm{law}}{=} \,     
 \mathcal{L} \Bigl\{  \,   \sum_{i\geq1}\vartheta_a
  \mathcal R^{(i)}_{q,nX_i(t_n(a))} \,\Bigm|\, \mathcal I(t_n(a)) 
  \Bigr\} 
 \label{eq:cond-distri}
\end{equation}
Indeed, given $\mathcal I(t_n(a))$, the Poisson
thinning property shows that the numbers of sampled labels in the interval
components $(I_i(t_n(a)))_{i\geq1}$ are independent Poisson  r.v.'s with respective means
$n|I_i(t_n(a))|=nX_i(t_n(a))$. 
Labels in the complement of $\mathcal I(t_n(a))$ are
singletons and hence do not contribute.  
By the 
fragmentation property, the subsequent evolutions inside distinct
interval components are conditionally independent and, after affine rescaling and
relabelling, have the laws of independent copies of the original
fragmentation.  Finally, for every residual separation time $T_A^{(i)}$,
$ \kappa_*\bigl(t_n(a)+T_A^{(i)}\bigr)-m_{n,q}
 =a+\kappa_*T_A^{(i)}$
which accounts for the shift $\vartheta_a$.

From \eqref{eq:cond-distri} it follows that for every $f\in C_{\mathrm c}^+((a,+\infty])$, writing
$\phi_{q,f}^{(a)}:=-\log\Phi_{q,f}^{(a)}$,
\begin{equation}\label{eq:poissonized-mark-product}
 \E \bigl[  \exp \bigl\{   -\langle f,\Xi_{n,q}^{\mathrm{Poi}}\rangle\bigr\}
  \,\bigm|\, \mathcal{I}  \big( t_n(a)\big)  \bigr]
 =\prod_{i\geq1}
   \Phi_{q,f}^{(a)}\bigl(nX_i(t_n(a))\bigr)
 =\exp\bigl\{-\langle\phi_{q,f}^{(a)},
   \mathcal E^{\mathrm{fr}}_{n,a}\rangle\bigr\}.
\end{equation}
  Applying Lemma \ref{lem:extremal-masses} and using $Z(t)\xrightarrow{a.s.} Z$, we obtain for every $h\in C_{\mathrm{b}}(\mathbb{R})$ and  $f\in C_{\mathrm c}^+((a,+\infty])$,
\begin{equation}
 \E \bigl[ h(Z) \exp \bigl\{   -\langle f,\Xi_{n,q}^{\mathrm{Poi}}\rangle\bigr\}
   \bigr]  \xrightarrow{n \to \infty} \E \bigl[ h(Z)  \exp\bigl\{-\langle\phi_{q,f}^{(a)},
   \mathcal E^{\mathrm{fr}}_{a}\rangle\bigr\}  
   \bigr] .\label{eq:lim-poi-ex}
\end{equation} 

We now verify the assumptions in Lemma \ref{lem:extremal-masses}. 
  First, the  dominated convergence theorem and continuity  in probability of Poisson process imply the continuity of 
$\Phi_{q,f}^{(a)}$. 
On
$\{\mathsf N_\lambda<q\}$, $\mathcal R_{q,\lambda}$ is the zero measure.  Hence
$1\geq\Phi_{q,f}^{(a)}(\lambda)\geq\pi_q(\lambda)>0$, which proves the continuity of 
$\phi_{q,f}^{(a)}$.  
For $0\leq\lambda\leq1$, we have
$1-\pi_q(\lambda)= \P( \mathsf{N}_{\lambda} \ge q )  \lesssim \lambda^q/q!$.
Since $\pi_q(\lambda)\geq\pi_q(1)$ and
$-\log x\leq(1-x)/x$, this yields
$-\log\pi_q(\lambda)\leq\lambda^q/(q!\pi_q(1))$. and hence  $\sup_{\lambda \in (0,1]}\lambda^{-q} \phi_{q,f}^{(a)}(\lambda) < \infty$.

\medskip
\noindent \underline{{Step 2}.}
We now construct the limiting point process from \eqref{eq:lim-poi-ex}. For $a \in \mathbb{R}$, we define a point process $ \Xi_{q,a}$ as follows:
Let 
$\mathcal E_a^{\mathrm{fr}}  =\sum_{i\geq1}\delta_{\lambda_i(a)} $, which is independent of  $\{(\mathcal R_{q,\lambda}^{(i)})_{\lambda\ge0}: i \ge 1\}$. Set  
\begin{equation}\label{eq:def-Xi-q-a}
   \Xi_{q,a}
   :=\sum_{i\geq1}\vartheta_a
   \mathcal R_{q,\lambda_i(a)}^{(i)}.
\end{equation} 
Then $ \Xi_{q,a}$ is almost surely a finite measure, since the identity
$\E[\mathcal R_{q,\lambda}([0,\infty])]=\E[ \binom{\mathsf{N}_{\lambda}}{q} ]  =\lambda^q/q!$ and
Lemma~\ref{lem:brw-inputs}~\textup{(\ref{input:Madaule})} give
$
 \E [
   \Xi_{q,a}((a,\infty])
  \mid \mathcal E_a^{\mathrm{fr}} ]
 =\frac1{q!}\int_0^\infty \lambda^q
 \mathcal E_a^{\mathrm{fr}}(\dif\lambda)<\infty$.  
Moreover, for every $f\in C_{\mathrm c}^+((a,\infty])$, conditional
independence gives
\[
 \E  \bigl[  
 \exp\bigl\{-\langle f,
  \Xi_{q,a}\rangle\bigr\}  
   \,\bigm|\, 
  \mathcal E_a^{\mathrm{fr}}
  \bigr] 
 =
 \prod_{i\geq1}\Phi_{q,f}^{(a)}(\lambda_i(a))
 =
 \exp\bigl\{
  -\langle\phi_{q,f}^{(a)},
   \mathcal E_a^{\mathrm{fr}}\rangle
 \bigr\}.
\]
 The standard criterion of point-process convergence through
Laplace functionals and \eqref{eq:lim-poi-ex} give  
\begin{equation}
   \bigl(\Xi_{n,q}^{\mathrm{Poi}}\!\restriction_{(a,+\infty]},Z\bigr)
 \xrightarrow[n\to\infty]{\mathrm{law}}
 \bigl( \Xi_{q,a} ,Z\bigr) \quad \text{ in } \Mloc((a,\infty]) \times \R. \label{eq:conv-re-a}
\end{equation}  
Furthermore, we claim that \eqref{eq:conv-re-a} implies  
$(\Xi_{q,a}\res{(b,\infty]},Z) \overset{\mathrm{law}}{=} (\Xi_{q,b},Z)$.\footnote{This is  consistency in law only: the point processes
$(\Xi_{q,a})_{a\in\mathbb R}$ defined by \eqref{eq:def-Xi-q-a} on their
original common probability space need not be pathwise consistent. }
Indeed, by \eqref{eq:def-Xi-q-a} and the fact that
$\mathcal R_{q,\lambda}(\{s\})=0$ almost surely for each fixed
$s\geq0$, we have $\Xi_{q,a}(\{b\})=0$ almost surely whenever $a<b$.
The assertion then follows from \eqref{eq:conv-re-a} and the continuous
mapping theorem. 

Thanks to   Kolmogorov's extension theorem, the consistency allows us to
realize the pairs
$(\Xi_{q,-k},Z)_{k\geq1}$ on a common probability space in such a way 
that  
$
 \Xi_{q,-(k+1)}\res{(-k,+\infty]}
 =\Xi_{q,-k} $ almost surely
for every $k\geq1$.
We may therefore define a point measure $\Xi_q$ on
$(-\infty,+\infty]$ by requiring
\[
 \Xi_q\res{(-k,+\infty]} :=\Xi_{q,-k},
 \qquad k\geq1.
\]
 In particular, $\Xi_q$ is locally finite point process on $(-\infty,\infty]$ with $\Xi_q(\{\infty\})=0$ and  satisfies 
 \[ \bigl(\Xi_{n,q}^{\mathrm{Poi}} ,Z\bigr)
 \xrightarrow[n\to\infty]{\mathrm{law}}
 \bigl( \Xi_{q} ,Z\bigr) \quad \text{ in } \Mloc((-\infty,\infty]) \times \R.\]

\medskip
\noindent \underline{{Step 3}.}
We next remove the Poissonization.  On the event $\{|\mathsf N_n-n|\leq n^{2/3}\}$,
the symmetric difference between  the collections of $q$-subsets of $[\mathsf{N}_n]$ and $[n]$ 
 has cardinality $|\binom{\mathsf{N}_{n}}{q} - \binom{n}{q} |\lesssim_q n^{q-1}|\mathsf N_n-n|$.  
Each $q$-subset $A$ in the symmetric difference of
$\binom{[\mathsf N_n]}q$ and $\binom{[n]}q$ contributes an atom to
exactly one of the two point processes. This atom is located at $\kappa_*T_A-m_{n,q}$ and belongs to
$(a,+\infty]$ if and only if $T_A>t_n(a)$.  The union bound and \eqref{eq:conditional-unseparated-probability} yield, on the event $\{|\mathsf N_n-n|\leq n^{2/3}\}$,
\[
 \P\Bigl(
  \Xi_{n,q}^{\mathrm{Poi}}\res{(a,+\infty]}
  \neq\Xi_{n,q}\res{(a,+\infty]}
  \Bigm|\mathcal I,\mathsf{N}_{n}\Bigr)
 \lesssim_q n^{-1/3}n^qS_q(t_n(a))
 =  n^{-1/3}\int_0^\infty \lambda^q
 \mathcal E_{n,a}^{\mathrm{fr}}(\dif\lambda).
\]

Lemma~\ref{lem:extremal-masses}, applied with $g(\lambda)=\lambda^q$,
shows that the last integral is tight.  Moreover, Chebyshev's
inequality gives
$\P(|\mathsf N_n-n|>n^{2/3})\to0$.  Thus we obtain $ \P (
  \Xi_{n,q}^{\mathrm{Poi}}\res{(a,+\infty]}
  \neq\Xi_{n,q}\res{(a,+\infty]}
) \to 0$ for every $a \in \mathbb{R}$. 
Combining this  with Step~2 yields, for every
fixed $a\in\R$,
\[
 \bigl(\Xi_{n,q}\res{(a,+\infty]},Z\bigr)
 \xrightarrow[n\to\infty]{\mathrm{law}}
 \bigl(\Xi_q\res{(a,+\infty]},Z\bigr).
\]
Taking $a=-k$, $k\geq1$, and using the definition of the local vague
topology proves \eqref{eq:frozen-point-process-limit}.

\medskip
\noindent \underline{{Step 4}.}
It remains to identify the   law of $\Xi_q$.  Fix $a\in\R$. 
Plugging the construction \eqref{eq:cluster-mass-representation} of $\mathcal{E}^{\mathrm{fr}}_a$   into 
\eqref{eq:def-Xi-q-a}, we  rewrite the law $\Xi_q\res{(a,+\infty]}$ as $ \sum_{j,k}\vartheta_a
   \mathcal R^{(j,k)}_{q, y_j \exp(d_{j,k}/\theta_*) }.$ Here $(\mathcal R_{q, \lambda}^{(j,k)})_{j,k \ge 1}$ are i.i.d. copies of $\mathcal R_{q, \lambda} $ independent of $\sum_{j} \delta_{y_j}$ and $\mathcal{D}_j $. Write  $\mathcal D=\sum_{k\geq1}\delta_{d_k}$.
 The independence 
gives, for every $y>0$ and every  $f \in C_{\mathrm{c}}^{+}((a,\infty])$
\[
 \E \Bigl[ 
  \exp \Bigl\{   -\Big\langle f,
   \sum_{k \ge 1}\vartheta_a
   \mathcal R_{q,y \exp(d_k/\theta_*)}^{(k)}
  \Big\rangle\Bigr\} \Bigr] 
 =\E  \Bigl[\,  \prod_{k \ge 1}
  \Phi_{q,f}^{(a)}\bigl( ye^{d_k/\theta_*}   \bigr) \Bigr].
\]
 The Campbell formula for Poisson point processes therefore yields,   for every   $f \in C^{+}_{\mathrm{c}}((a,\infty])$ 
\begin{align}
  \E \bigl[  
  e^{ - \langle f,
   \Xi_q 
   \rangle} \mid  Z
  \bigr] &=  \E \bigl[  
  e^{ - \langle f,
   \Xi_q \res{(a,\infty]} 
   \rangle} \mid  Z
  \bigr] \\
  & =\exp  \Bigl\{ 
  -C_VZ\theta_*v_*^{3/2} \,e^{-a}
  \int_0^\infty \Bigl[ 
   1-\E  \prod_{k \ge 1}
  \Phi_{q,f}^{(a)}\bigl( ye^{d_k/\theta_*}  \bigr) 
   \Bigr]y^{-\theta_*-1}\,\dif y
 \Bigr\} . 
 \label{eq:frozen-cluster-Laplace}
\end{align}

This Laplace exponent suggests
defining the following measure on $ \mathcal{M}_{\mathrm{fp}}((a,\infty]) \setminus \{\mathbf{0}\}$, the space of non-zero finite point
measures on $(a,+\infty]$:  for every Borel subset $\mathcal{B}$ of $\mathcal{M}_{\mathrm{fp}}((a,\infty]) \setminus \{\mathbf{0}\}$,
\begin{equation}\label{eq:frozen-cluster-intensity}
 \Lambda_q^{(a)}(\mathcal{B})
 :=C_V\theta_*v_*^{3/2}e^{-a}
 \int_0^\infty
 \P \Bigl( \,
  \sum_{k\geq1}\vartheta_a
  \mathcal R_{q,ye^{d_k/\theta_*}}^{(k)}
  \in\mathcal{B}  \Bigr)y^{-\theta_*-1}\,\dif y.
\end{equation}
Here we verify that  $ \sum_{k\geq1} 
 \mathcal R_{q,ye^{d_k/\theta_*}}^{(k)}$ is almost surely a finite measure. As in the argument below \eqref{eq:def-Xi-q-a}   $\E[ \sum_k\mathcal R_{q,ye^{d_k/\theta_*}}^{(k)}([0,\infty]) \mid \mathcal{D}] \lesssim_{q} y^{q} \sum_{k} e^{q d_k/\theta_*} = y^{q} \langle \exp_{q/\theta_{*}} ,\mathcal{D} \rangle< \infty$ by Lemma~\ref{lem:brw-inputs}~\textup{(\ref{input:Madaule})}. 
We further show that $\Lambda_q^{(a)}  $ is a finite measure. Since each  
$\mathcal R_{q,ye^{d_k/\theta_*}}^{(k)}$ is the zero measure if and only if the corresponding Poisson count 
$\mathsf N^{(k)}_{ye^{d_k/\theta_*}}<q$, the total mass  of  $\Lambda_q^{(a)}  $ is   
\begin{equation}
   \Lambda_q^{(a)}( \mathcal{M}_{\mathrm{fp}}((a,\infty]) \setminus \{\mathbf{0}\})
   =C_V\theta_*v_*^{3/2}e^{-a}
   \int_0^\infty \Bigl[ 
    1-\E\prod_{d\in\mathcal D}
    \pi_q(ye^{d/\theta_*})
   \Bigr]   y^{-\theta_*-1}\,\dif y
   =C_q^\star\,e^{-a}.\label{eq:total-mass-Lambda}
\end{equation}
 
 Combining \eqref{eq:frozen-cluster-intensity} and
\eqref{eq:frozen-mark-kernel} with  
\eqref{eq:frozen-cluster-Laplace}, we obtain  for every  $f \in C^{+}_{\mathrm{c}}((a,\infty])$
\begin{equation}\label{eq:frozen-canonical-Laplace}
 \E \bigl[  
  e^{ - \langle f,
   \Xi_q 
   \rangle} \mid  Z
  \bigr]
 =\exp \Bigl\{  -Z\int
  \bigl(1-e^{-\langle f,\mu\rangle}\bigr)
  \Lambda_q^{(a)}(\dif\mu) \Bigr\} .
\end{equation}

\medskip
\noindent \underline{{Step 5}.}
It remains to identify the law of $\Xi_q$.  Since
$\Lambda_q^{(a)}$ is finite with total mass $C_q^\star e^{-a}$,
\eqref{eq:frozen-canonical-Laplace} gives
\begin{equation}\label{eq:frozen-void-probability}
 \P\bigl(\Xi_q((a,+\infty])=0\mid Z\bigr)
 =\exp\{-C_q^\star Ze^{-a}\}.
\end{equation}  
For $f\in C_{\mathrm c}^+(\R)$, choose $a\in\R$ such that
$\supp(f)\subset(a,+\infty)$, and denote by  \[ K_q(f) := \int
  \bigl(1-e^{-\langle f,\mu\rangle}\bigr)
  \Lambda_q^{(a)}(\dif\mu) =   C_V \theta_*v_*^{3/2} \,e^{-a}
  \int_0^\infty \Bigl[ 
   1-\E  \prod_{k \ge 1}
  \Phi_{q,f}^{(a)}\bigl( ye^{d_k/\theta_*}  \bigr) 
   \Bigr]y^{-\theta_*-1}\,\dif y  \] 
This definition does not depend on
the choice of $a$: for every admissible $a$, that equation gives the
same conditional Laplace functional by \eqref{eq:frozen-canonical-Laplace} , and $Z>0$ almost surely.

For $s\in\R$, set $f_s(x):=f(x+s)$.  If
$\supp(f)\subset(a,+\infty)$, then
$\Phi_{q,f_s}^{(a-s)}=\Phi_{q,f}^{(a)}$ and $\supp(f_s)\subset(a-s,+\infty)$. Thus the second expression of $K_q(\cdot)$ gives 
\begin{equation}\label{eq:frozen-exponential-scaling}
 K_q(f_s)=e^sK_q(f).
\end{equation} 
We may set
$
 \widehat\Xi_q
 :=\vartheta_{-\log(C_q^\star Z)}\Xi_q$, since $Z>0$ almost surely.
Then \eqref{eq:frozen-exponential-scaling} gives
\[
 \E\bigl[e^{-\langle f,\widehat\Xi_q\rangle}\mid Z=z\bigr]
 =\exp \bigl\{ 
  -zK_q\bigl(f_{-\log(C_q^\star z)}\bigr)
  \bigr\}
 =\exp \Bigl\{ -\frac{1}{C_q^\star}K_q(f)  \Bigr\}.
\]
In particular, $\widehat\Xi_q$ is independent of $Z$. 
Moreover, 
$\widehat\Xi_q$ is exp-$1$-stable  in the sense of
\eqref{eq:exp-one-stability} :  if
$e^\alpha+e^\beta=1$, then for independent copies
$\widehat\Xi_q^{(1)}$ and $\widehat\Xi_q^{(2)}$  of $\widehat\Xi_q$, we have 
$
 \vartheta_\alpha\widehat\Xi_q^{(1)}
 +\vartheta_\beta\widehat\Xi_q^{(2)}
 \overset{\mathrm{law}}{=}\widehat\Xi_q$. Because 
\eqref{eq:frozen-exponential-scaling} gives
\[
 \E\Bigl[
   \exp\bigl\{ -\langle f,\vartheta_\alpha\widehat\Xi_q^{(1)}\rangle
     -\langle f,\vartheta_\beta\widehat\Xi_q^{(2)}\rangle
   \bigr\} \Bigr]
 =
 \exp\Bigl\{ -\frac{K_q(f_\alpha)+K_q(f_\beta)}{C_q^\star} \Bigr\}
 =
 \exp\Bigl\{-\frac{K_q(f)}{C_q^\star}\Bigr\}.
\]
Since $\Xi_q=\vartheta_{\log(C_q^\star Z)}\widehat\Xi_q$,
\eqref{eq:frozen-void-probability} also yields
$
 \P\bigl(\widehat\Xi_q((a,+\infty])=0\bigr)
 =\exp\{-e^{-a}\}$ for $a \in \mathbb{R}$.
In particular, $\widehat\Xi_q$ is nonzero almost surely.
It is locally finite on $(-\infty,+\infty]$ and has no atom at
$+\infty$, so Lemma~\ref{lem:maillard-exp-stable} yields
\[
 \widehat\Xi_q\sim\DPPP(e^{-x}\,\dif x,\mathcal D_q^\star),
\]
where $\mathcal D_q^\star$ has rightmost atom zero almost surely.  
Shifting a  $\PPP( e^{-x}\,\dif x)$ by $c$
changes its intensity to $e^{c} e^{-x}\,\dif x$.  Since
$\Xi_q=\vartheta_{\log(C_q^\star Z)}\widehat\Xi_q$ and $Z$ is independent of $\widehat\Xi_q$, we conclude  \eqref{eq:frozen-SDPPP}.
\end{proof}
 
Now we are ready to prove Corollary~\ref{thm:three-regimes}.

\begin{proof}[Proof of Corollary~\ref{thm:three-regimes} for
$q>\theta_*$]
 Let  $r=q-1$. We divide the proof into several steps.

\smallskip\noindent
\underline{\textit{Step 1.}} 
Listed in decreasing order and with multiplicities, the atoms of
$\Xi_{n,q}$ are
$\kappa_*D_{n,r}^{(j)}-m_{n,q}$,
$1\leq j\leq\binom{n}{q}$.  By
\eqref{eq:frozen-SDPPP}, the limiting process has no atom at $+\infty$,
has locally finite upper tails, and has infinitely many atoms.  Thus
Proposition~\ref{thm:frozen-separation-process} and the continuous-mapping
theorem, applied to the first $k+1$ ordered atoms, yield
\[
 \Bigl(  
  \kappa_*D_{n,r}^*-m_{n,q},
  (\kappa_*(D_{n,r}^{(j)}-D_{n,r}^{(j+1)}))_{j=1}^k
 \Bigr)
 \xrightarrow[n\to\infty]{\mathrm{law}}
  \Bigl(  \mathsf A_q^{(1)},(\mathsf\Delta_{q,j})_{j=1}^k\Bigr) . 
\] 
Since $\gamma_q=\kappa_*$,   this is precisely the
first convergence asserted in the corollary.
 Moreover, by \eqref{eq:frozen-void-probability},   we have   
$
 \P(\mathsf A_q^{(1)}\leq x\mid Z)
 =\exp\{-C_q^\star Ze^{-x}\}$ for every $x \in \mathbb{R}$
and therefore
\[
 \mathsf A_q^{(1)}
 \overset{\mathrm{law}}=
 G+\log(C_q^\star Z),
\] 
where $G$ is a standard Gumbel random variable independent of $Z$. 
It remains to prove that
$\P(\mathsf\Delta_{q,j}=0)>0$ for every fixed $j\geq1$.
 
\smallskip\noindent
\underline{\textit{Step 2.}} We first describe another point of view of the law of $\Xi_q$.  
Fix $a\in\R$. By
\eqref{eq:frozen-canonical-Laplace} and Campbell's formula, conditionally on $Z$,   we may realize $\Xi_q$ as follows
\[
 \Xi_q\res{(a,\infty]}
=  \sum_{ i} \mu_i,
 \qquad
 \sum_i\delta_{\mu_i}
 \sim\PPP\bigl(Z\Lambda_q^{(a)}\bigr),
\]  
Define  for every  nonempty bounded open interval
$J\subset(a,\infty)$, 
\[  \mathsf{Pur}_m (J):=\{m \delta_x:x\in J\} \quad \text{ for } \quad m \in\{1,q+1\}. \]
Choose bounded open intervals
$J_{\mathrm{L}},J_{\mathrm{R}}\subset(a,\infty)$ with $\sup J_{\mathrm{L}}<\inf J_{\mathrm{R}}$ and consider the event 
\[  
G^{\Xi}_{j}=\Bigl\{  \sum_{i} \ind{\mu_i\in\mathsf{Pur}_{1}(J_{\mathrm{R}}) } = j-1, \sum_{i} \ind{\mu_i\in\mathsf{Pur}_{q+1}(J_{\mathrm{L}}) } = 1 ,\sum_{i} \ind{\mu_i \notin \mathsf{Pur}_{q+1}(J_{\mathrm{L}})\cup \mathsf{Pur}_1(J_{\mathrm{R}})}=0 
\Bigr\} 
\]
On this event $G^{\Xi}_{j}$, the
$j-1$ singleton atoms in $J_{\mathrm{R}}$ are
the only atoms above the $(q+1)$-fold atom in $J_{\mathrm{L}}$; all remaining
atoms lie at or below $a$.  Hence,  
$\mathsf A_q^{(j)}=\mathsf A_q^{(j+1)}$, or equivalently
$\mathsf\Delta_{q,j}=0$. 

It is sufficient to show $\P(G^{\Xi}_{j})>0$. Write 
\[  \lambda_{\mathrm{L},q+1}:=\Lambda_q^{(a)}(\mathsf{Pur}_{q+1}(J_{\mathrm{L}})) \ \text{ and } \
 \lambda_{\mathrm{R},1}:=\Lambda_q^{(a)}(  \mathsf{Pur}_{1}(J_{\mathrm{R}}) ). \]  
Then  the standard Poisson point process calculation and  \eqref{eq:total-mass-Lambda} give
\[
 \P(G^{\Xi}_{j} \mid Z)=e^{-ZC_q^\star e^{-a}} \ Z\lambda_{\mathrm{L},q+1}\ 
 \frac{(Z\lambda_{\mathrm{R},1})^{j-1}}{(j-1)!} . 
\]
 The desired result therefore follows from the following assertion:   for every  nonempty bounded open interval
$J\subset(a,\infty)$, 
\begin{equation}\label{eq:pure-truncated-frozen-clusters}
 \Lambda_q^{(a)}(\mathsf{Pur}_1(J))>0 \quad \text{ and } \quad 
 \Lambda_q^{(a)}(\mathsf{Pur}_{q+1}(J))>0, 
\end{equation} 

\smallskip\noindent
\underline{\textit{Step 3.}}  We now prove \eqref{eq:pure-truncated-frozen-clusters} from the definition \eqref{eq:frozen-cluster-intensity} of $ \Lambda_q^{(a)}$. Recall $(\mathsf{N}^{(k)}_{\lambda})$ are the independent Poisson point process used to define $\mathcal{R}^{(k)}_{q,\lambda}$, and 
$\mathcal D=\sum_{k\geq1}\delta_{d_k}$  
with $d_1=0$ as one  of its
rightmost atom.  
Consider the event, for
$m\in\{q,q+1\}$,
\begin{equation}
  F_m(y) :=  \bigl\{
     \mathsf N_y^{(1)}=m,\
   \mathsf N_{ye^{d_k/\theta_*}}^{(k)}<q \
   \text{ for every }k\geq2   \bigr\}.
\end{equation} 
 Note that on $F_m(y)$, we have 
\begin{equation}\label{eq:r-red-1}
   \sum_{k \ge 1} \vartheta_a \mathcal R_{q,ye^{d_k/\theta_*}}^{(k)}  = \sum_{A \in \binom{[m]}{q}} \delta_{a+\kappa_* T_A } 
\end{equation}
Moreover, 
$
 \P ( F_m(y)
  \mid \mathcal D)
 =
 e^{-y}\frac{y^m}{m!}
 \prod_{k\geq2}\pi_q(ye^{d_k/\theta_*})
 >0.$  because  $q/\theta_*>1$ and   
Lemma~\ref{lem:brw-inputs}  (\ref{input:Madaule}) yields that 
$
 \sum_{k\geq2}
 [1-\pi_q(ye^{d_k/\theta_*}) ]
 \leq
 \frac{y^q}{q!}
 \langle\exp_{q/\theta_*},\mathcal D\rangle
 <\infty$.

Take first $m=q$.  Applying   \eqref{eq:r-red-1} gives
 \[  F_m(y) \cap \{a+\kappa_*T_{[q]}\in J\} \subset   \bigl\{   \sum_{k } \vartheta_a \mathcal R_{q,ye^{d_k/\theta_*}}^{(k)}  \in \mathsf{Pur}_1(J) \bigr\}.  \] 
 Since  $T_{[q]}$ is exponential with rate $\kappa(q)$ and is
independent of $\mathcal{D}$ and $(  \mathsf{N}^{(k)}  )_{k \ge 1}$,   taking $\mathcal B=\mathsf{Pur}_1(J)$ in
\eqref{eq:frozen-cluster-intensity} yields 
\[  \Lambda_q^{(a)}(\mathsf{Pur}_1(J))
 \gtrsim  e^{-a}
 \int_0^\infty
 \P (F_m(y))  \P(a+\kappa_*T_{[q]}\in J)y^{-\theta_*-1}\,\dif y >0
. \]

Let $m=q+1$. 
Denote by $G^{\Pi}_q$ the event that every sub-clade created when $[q+1]$
splits at time $T_{[q+1]}$ has size at most $q-1$. 
Then on $G^{\Pi}_q$, every $A\in\binom{[q+1]}q$ has separation
time $T_{[q+1]}$.  Combining this with   \eqref{eq:r-red-1}, we get 
\[
  F_m(y) \cap G^{\Pi}_q \cap \{a+\kappa_*T_{[q+1]}\in J\} \subset  \Bigl\{  \sum_{k\geq1}\vartheta_a
 \mathcal R_{q,ye^{d_k/\theta_*}}^{(k)}
 =(q+1)\delta_{a+\kappa_*T_{[q+1]}}
 \in\mathsf{Pur}_{q+1}(J) \Bigr\} .
\]

 Let
$\rho_q$ be the rate of the event that $[q+1] $ is separated and  $G^{\Pi}_q$ occurs.
Since $T_{[q+1]}$ has density
$\kappa(q+1)e^{-\kappa(q+1)t}\,\dif t$, and the conditional
probability of $G^{\Pi}_q$ given the separation event is
$\rho_q/\kappa(q+1)$, we have  
\[
 \P\bigl(T_{[q+1]}\in\dif t,\,G^{\Pi}_q\bigr)
 =\kappa(q+1)e^{-\kappa(q+1)t}\,\dif t
   \frac{\rho_q}{\kappa(q+1)}
 =\rho_qe^{-\kappa(q+1)t}\,\dif t.
\]
The fragmentation is independent of $\mathcal{D}$ and $(  \mathsf{N}^{(k)}  )_{k \ge 1}$. Taking
$\mathcal B=\mathsf{Pur}_{q+1}(J)$ in
\eqref{eq:frozen-cluster-intensity} gives
\[
 \Lambda_q^{(a)}(\mathsf{Pur}_{q+1}(J))
 \gtrsim  e^{-a}
 \int_0^\infty \P(F_{q+1}(y))
y^{-\theta_*-1}\,\dif y    
  \int_{\{t>0:\,a+\kappa_*t\in J\}}
  \rho_q \, e^{-\kappa(q+1)t}\,\dif t 
\]

We only have to show $\rho_q>0$. 
During a splitting event of $\Pi$ with
a $\mathbf s$-paintbox, if  
$[q+1] $ is not separated or $G^{\Pi}_q$ does not occur, 
then either all $q+1$ labels enter one box, or $q$ labels
enter one box and the remaining label does not. 
Therefore, we get 
\begin{align}
 \rho_q
 =\int_{\mathcal S^\downarrow}
    \Bigl[1-\sum_i s_i^{q+1}
   -(q+1)\sum_i s_i^q(1-s_i)   \Bigr]  \nu(\dif\mathbf s)
    =(q+1)\kappa(q)-q\kappa(q+1) . 
 \label{eq:pure-q-plus-one-split-rate}
\end{align} 
Since   
$
 \frac{\dif}{\dif\theta}\frac{\kappa(\theta)}\theta
 =\frac{\Delta(\theta)}{\theta^2}<0$ for all $\theta> \theta_*$, we obtain   $\rho_q>0$. This completes the proof.
\end{proof}

\subsection{\texorpdfstring{The size-$q$ clade-count process
}{The size-q block-count process}}

The process $\mathcal{R}_{q,\lambda}$ \eqref{eq:residual-separation-process} records the
separation times generated by a Poissonian number of labels. For the same $(\mathsf{N},\Pi,\mathcal{I})$ define the  counting process on the frozen time scale  
\begin{equation}\label{eq:finite-exact-q-count}
 K_{q,\lambda}(s)
 :=\#\bigl\{\mathsf C\in\Pi(s/\kappa_*)|_{[\mathsf N_\lambda]}:
 |\mathsf C|=q\bigr\},
 \qquad \lambda\geq0,\ s\geq0.
\end{equation} 
In particular  $K_{q,\lambda}(0)=\ind{\mathsf N_\lambda=q}$ and when $\mathsf N_\lambda<q$, $K_{q,\lambda} \equiv 0$.

Let $(\mathsf N^{(i)},\Pi^{(i)},\mathcal I^{(i)})_{i\geq1}$ be
i.i.d.\ copies of $(\mathsf N,\Pi,\mathcal I)$, independent  of $(Z,(\mathcal E_a^{\mathrm{fr}})_{a\in\R})$. 
For each $i\geq1$ define
$(K^{(i)}_{q,\lambda})_{\lambda\geq0}$ from $(\mathsf N^{(i)},\Pi^{(i)})$ according to
\eqref{eq:finite-exact-q-count}.   
For every  point measure
$\mu=\sum_{i\geq1}\delta_{\lambda_i}$ on $(0,\infty]$ with
$\int_0^\infty z^q\,\mu(\dif z)<\infty$, set
\begin{equation}\label{eq:def-K-q}
 \mathcal{K}_q(\mu) : [0,\infty) \to [0,\infty) \ ; \ s \mapsto \sum_{i\geq1}K^{(i)}_{q,\lambda_i}(s)
\end{equation}
The sum in \eqref{eq:def-K-q} is almost surely finite,
$\mathcal{K}_q(\mu)$ is a well-defined random element of
$D([0,\infty),\mathbb Z_{\geq0})$. Indeed,  since $\P(\mathsf N_\lambda\geq q) \lesssim_q \lambda^q $, we have
$
 \sum_{i\geq1}\P\bigl(\mathsf N^{(i)}_{\lambda_i}\geq q\bigr)
 \lesssim_q \int_0^\infty z^q\,\mu(\dif z)<\infty$. Thus by the Borel--Cantelli lemma, almost surely $K^{(i)}_{q,\lambda_i}\equiv0$ for all
but finitely many $i$.   

For   $a\in\mathbb R$, and a function $f$ is defined on $[a,\infty)$, set 
$ \tau_a f:[0,\infty)\to \mathbb{R}, s \mapsto f(a+s)$. 

The case $q>\theta_*$ of
Theorem~\ref{thm:intro-exact-q-clade-process} follows from the
following proposition.
  
\begin{proposition} 
\label{thm:frozen-exact-q-blocks}
There exists a
$D(\mathbb R,\mathbb Z_{\geq0})$-valued process $N_q$, jointly defined
with $Z$, such that
\begin{equation}\label{eq:frozen-exact-q-process-limit}
 \bigl((N_{n,q}(x))_{x\in\R},  
  Z\bigr)
 \xrightarrow[n\to\infty]{\mathrm{law}}
 \bigl((N_q(x))_{x\in\R}, 
  Z\bigr) \ \text{ in }  D(\R,\mathbb Z_{\geq0}) \times\R  
\end{equation}
and the following assertions hold:
\begin{enumerate}[(i)]
  \item For every fixed $a\in\mathbb R$, there exists a joint realization
(possibly depending on $a$) of
$N_q$, $\mathcal E_a^{\mathrm{fr}}$, and $Z$ whose
$(N_q,Z)$-marginal coincides with the limiting law in
\eqref{eq:frozen-exact-q-process-limit} and whose
$(\mathcal E_a^{\mathrm{fr}},Z)$-marginal is the law specified in
\eqref{eq:mass-process-definition}, such that
\begin{equation}\label{eq:frozen-exact-q-limit-construction}
 \tau_aN_q
 =
 \mathcal K_q(\mathcal E_a^{\mathrm{fr}})
 \qquad\text{almost surely}.
\end{equation} 
\item For every $a<b$, with positive probability $N_q$ has
both an upward and a downward jump in $(a,b)$.
\end{enumerate} 
\end{proposition}

\begin{proof}[Proof of Proposition~\ref{thm:frozen-exact-q-blocks}]
 
In the proof we write $D_{[0,\infty)}$ short for $D ([0,\infty),\mathbb Z_{\geq0})$.

\smallskip\noindent
\underline{\textit{Step 1.}}  Fix $a \in \mathbb{R}$ and  set  
\[
 N^{\mathrm{Poi}}_{n,q}(x)
 :=\#\bigl\{\mathsf{C}\in\Pi(t_n(x))|_{[\mathsf N_n]}:
 |\mathsf{C}|=q\bigr\},\qquad x\geq a.
\] 
We claim the following analog to \eqref{eq:frozen-exact-q-limit-construction}:
\begin{equation}
   \label{eq:poi-N-cov-1}
   \Bigl(
    \tau_a N^{\mathrm{Poi}}_{n,q} , \, 
    \mathcal E^{\mathrm{fr}}_{n,a},Z(t_n(a))
   \Bigr)
   \xrightarrow[n\to\infty]{\mathrm{law}}
   \Bigl(
      \mathcal{K}_q(\mathcal E^{\mathrm{fr}}_a )  ,\,
    \mathcal E^{\mathrm{fr}}_a,Z
   \Bigr)
\end{equation}
in
$D_{[0,\infty)}\times
\Mloc((0,\infty])\times\mathbb R$. 
Indeed, by Poisson thinning and the fragmentation property, the argument for proving \eqref{eq:cond-distri} yields
\begin{align}
  \mathcal{L} \Bigl\{(N^{\mathrm{Poi}}_{n,q}(a+s))_{s \ge 0} \,\Bigm|\, \mathcal{I}(t_n(a))  \Bigr\}
  &= \mathcal{L}\Bigl\{
    \Bigl(
    \sum_{i\geq1}
    K^{(i)}_{q,nX_i(t_n(a))}
    \bigl(s\bigr)\Bigr)_{s \ge 0} \,\Bigm|\, \mathcal{I}(t_n(a)) \Bigr\} \\
  &= \mathcal{L}\bigl\{  \mathcal{K}_q(\mathcal E^{\mathrm{fr}}_{n,a} ) 
     \,\bigm|\, \mathcal{I}(t_n(a)) \bigr\}.
\end{align}
Since both $\mathcal E^{\mathrm{fr}}_{n,a}$ and $Z(t_n(a))$ are
$\mathcal I(t_n(a))$-measurable, the preceding conditional identity in
law reduces \eqref{eq:poi-N-cov-1} to showing
\begin{equation}
 \Bigl(
  \mathcal{K}_q(\mathcal E^{\mathrm{fr}}_{n,a} ) ,
  \mathcal{E}_{n,a}^{\mathrm{fr}},Z(t_n(a))
 \Bigr)
 \xrightarrow[n\to\infty]{\mathrm{law}}
 \Bigl(
  \mathcal{K}_q(\mathcal E^{\mathrm{fr}}_a ) ,
  \mathcal E_{a}^{\mathrm{fr}},Z
 \Bigr).
 \label{eq:K-cov-1}
\end{equation}

We next reduce \eqref{eq:K-cov-1} to the corresponding convergence for
the point measures truncated below at $\epsilon$. Let
$\bar{\mathcal E}_{n,a}^{\mathrm{fr},[\epsilon]}:=
\sum_i\ind{nX_i(t_n(a))\geq\epsilon}\delta_{nX_i(t_n(a))}$ and
$\bar{\mathcal E}_{a}^{\mathrm{fr},[\epsilon]}:=
\sum_i\ind{\lambda_i(a)\geq\epsilon}\delta_{\lambda_i(a)}$. Then
\begin{equation}
 \Bigl(
  \mathcal{K}_q(\bar{\mathcal E}_{n,a}^{\mathrm{fr},[\epsilon]}) ,
  \mathcal E_{n,a}^{\mathrm{fr}},Z(t_n(a))
 \Bigr)
 \xrightarrow[n\to\infty]{\mathrm{law}}
 \Bigl(
  \mathcal{K}_q(\bar{\mathcal E}_{a}^{\mathrm{fr},[\epsilon]}),
  \mathcal E_{a}^{\mathrm{fr}},Z
 \Bigr).
 \label{eq:K-cov-2}
\end{equation}

To justify this reduction, let $B_{n,\epsilon}$ be the event that
$\mathsf N^{(i)}_{nX_i(t_n(a))}\geq q$ for some $i$ satisfying
$nX_i(t_n(a))<\epsilon$. On $B_{n,\epsilon}^{\mathrm c}$, we have
$K^{(i)}_{q,nX_i(t_n(a))}\equiv0$ for every such $i$, which implies that
$\mathcal{K}_q(\mathcal E^{\mathrm{fr}}_{n,a})$ and
$\mathcal{K}_q(\bar{\mathcal E}^{\mathrm{fr},[\epsilon]}_{n,a})$ coincide.
A union bound and the Poisson tail estimate give
\[
 \P[B_{n,\epsilon}\mid\mathcal{I}(t_n(a))]
 \leq
 1\wedge\frac1{q!}
 \int_{(0,\epsilon)}z^q\,
 \mathcal E^{\mathrm{fr}}_{n,a}(\dif z).
\]
Since $q>\theta_*$, Lemma~\ref{lem:extremal-masses}, applied with
$g(y)=y^q$, together with dominated convergence shows that
$\P(B_{n,\epsilon})$ tends to zero when $n\to\infty$ followed by
$\epsilon\downarrow0$. The same argument shows that
$\P(\mathcal{K}_q(\mathcal E^{\mathrm{fr}}_a)
\equiv\mathcal{K}_q(\bar{\mathcal E}^{\mathrm{fr},[\epsilon]}_a))\to1$
as $\epsilon\downarrow0$. Thus by Slutsky's lemma
\eqref{eq:K-cov-2} is enough for \eqref{eq:K-cov-1}.

\medskip\noindent
\underline{\textit{Step 2.}}
By Lemma~\ref{lem:extremal-masses} and the Skorokhod representation
theorem, we may work on an auxiliary probability space on which
$(\mathcal E_{n,a}^{\mathrm{fr}},Z(t_n(a)))\to
(\mathcal E_a^{\mathrm{fr}},Z)$ almost surely in
$\Mloc((0,\infty])\times\R$. Set
$M_\epsilon:=\mathcal E_a^{\mathrm{fr}}([\epsilon,\infty])$.
Since $\mathcal E_a^{\mathrm{fr}}$ is locally finite and has no atom at
$+\infty$, we have $M_\epsilon<\infty$ almost surely. Moreover,
$\P(\mathcal E_a^{\mathrm{fr}}(\{\epsilon\})=0)=1$, since
\eqref{eq:cluster-mass-representation}, Campbell's formula   yield
\[
 \E[\mathcal E_a^{\mathrm{fr}}(\{\epsilon\})\mid Z]
 \lesssim e^{-a}Z
 \E\Bigl[
  \sum_{d\in\mathcal D}\int_0^\infty
  \ind{ye^{d/\theta_*}=\epsilon}
  y^{-\theta_*-1}\,\dif y
 \Bigr]
 =0.
\]
The vague convergence thus implies that, almost surely for all
sufficiently large $n$,
$\mathcal E_{n,a}^{\mathrm{fr}}([\epsilon,\infty])=M_\epsilon$ and
$nX_i(t_n(a))\to\lambda_i(a)$ for $1\leq i\leq M_\epsilon$. Here we used that the
atoms $nX_i(t_n(a)), \lambda_i(a)$ both are listed in decreasing order with multiplicity. Hence
\begin{align}
 &\P \Bigl[  
  \mathcal K_q(\bar{\mathcal E}^{\mathrm{fr},[\epsilon]}_{n,a})
  \ne
  \mathcal K_q(\bar{\mathcal E}^{\mathrm{fr},[\epsilon]}_a)
 \,\Bigm|\, \mathcal E^{\mathrm{fr}}_{n,a},
                    \mathcal E^{\mathrm{fr}}_a
 \Bigr] \le \sum_{i=1}^{M_{\epsilon}} \P\Bigl[
  K^{(i)}_{q,nX_i(t_n(a))}\ne K^{(i)}_{q,\lambda_i(a)} 
  \,\Bigm|\,\mathcal E_{n,a}^{\mathrm{fr}},\mathcal E_a^{\mathrm{fr}}
 \Bigr] \\
 &\qquad  \leq \sum_{i=1}^{M_{\epsilon}}
 \P\Bigl[
  \mathsf N^{(i)}_{nX_i(t_n(a))}\ne\mathsf N^{(i)}_{\lambda_i(a)} 
  \,\Bigm|\,\mathcal E_{n,a}^{\mathrm{fr}},\mathcal E_a^{\mathrm{fr}}
 \Bigr]  \leq
 \sum_{i=1}^{M_\epsilon}
 \bigl[1-e^{-|nX_i(t_n(a))-\lambda_i(a)|}\bigr]
 \xrightarrow[n\to\infty]{\mathrm{a.s.}}0.
\end{align} 
Together with the almost sure convergence of the other two coordinates,
this proves \eqref{eq:K-cov-2}.  By  reductions in Step 1, \eqref{eq:poi-N-cov-1} follows.

\medskip\noindent
\underline{\textit{Step 3.}}
We next remove the Poissonization, by adapting the argument in Step~3 of
the proof of Proposition~\ref{thm:frozen-separation-process}. Observe
that if the sample paths of $N^{\mathrm{Poi}}_{n,q}$ and $N_{n,q}$
differ on $[a,\infty)$, some $q$-subset in the symmetric difference
$\binom{[\mathsf N_n]}q\mathbin{\triangle}\binom{[n]}q$ must still be
unseparated at time $t_n(a)$. Hence, by
\eqref{eq:conditional-unseparated-probability}, the union bound and
$|\binom{\mathsf N_n}{q}-\binom nq|
\lesssim_q n^{q-1}|\mathsf N_n-n|$, on
$\{|\mathsf N_n-n|\leq n^{2/3}\}$,
\[
 \P\left(
  \tau_a  N^{\mathrm{Poi}}_{n,q} 
  \ne   \tau_a  N_{n,q} 
  \,\bigm|\,\mathcal I(t_n(a)),\mathsf N_n
 \right)
 \lesssim_q n^{-1/3}
 \int_0^\infty z^q\mathcal E^{\mathrm{fr}}_{n,a}(\dif z).
\]
By Lemma~\ref{lem:extremal-masses} and the fact 
$\P(|\mathsf N_n-n|>n^{2/3})\to0$, we obtain 
$
 \P (
  N^{\mathrm{Poi}}_{n,q}=N_{n,q}\text{ on }[a,\infty)
 )\longrightarrow1$ as $n \to \infty$. 
Hence we can replace $N^{\mathrm{Poi}}_{n,q}$ in
\eqref{eq:poi-N-cov-1} by $N_{n,q}$ and obtain 
\begin{equation}
 \label{eq:poi-N-cov-2}
  \Bigl( 
   \tau_a  N_{n,q} , \, 
  \mathcal E^{\mathrm{fr}}_{n,a},Z(t_n(a))
  \Bigr)
 \xrightarrow[n\to\infty]{\mathrm{law}}
 \Bigl( 
  \mathcal K_q(\mathcal E^{\mathrm{fr}}_a), \,
  \mathcal E^{\mathrm{fr}}_a,Z
  \Bigr) \quad \text{ in }D_{[0,\infty)}\times
\Mloc((0,\infty])\times\mathbb R.
\end{equation}

\medskip\noindent 
\underline{\textit{Step 4.}} We now identify the limiting process $N_q$.  
We first show that  for every $a\in \mathbb{R}$ and $ h>0$,
\begin{equation}\label{eq:consistency-2}
  \Bigl( 
    \tau_h  \mathcal K_q(\mathcal E_a^{\mathrm{fr}}) ,\, Z
    \Bigr) 
   \overset{\mathrm{law}}{=}
    \Bigl(  
     \mathcal K_q(\mathcal E_{a+h}^{\mathrm{fr}}) ,\, Z
   \Bigr) . \footnote{We emphasize that \eqref{eq:consistency-2} is only an identity in law.
Although $(\mathcal E_a^{\mathrm{fr}})_{a\in\mathbb R}$ and the copies
defining $\mathcal K_q$ are realized on a common probability space, this
realization does not in general satisfy
$
 \mathcal K_q(\mathcal E_a^{\mathrm{fr}})(h+\cdot)
 =
 \mathcal K_q(\mathcal E_{a+h}^{\mathrm{fr}})$ a.s.}
\end{equation}
 Indeed $\tau_h :  D_{[0,\infty)} \to D_{[0,\infty)}$ is continuous at every function $f$ who is
continuous at $h$.   
 By the argument following \eqref{eq:def-K-q}, only finitely many
summands of $\mathcal K_q(\mathcal E_a^{\mathrm{fr}})$ are nonzero.
Conditional on $\mathcal E_a^{\mathrm{fr}}$, these summands are
independent and have atomless jump times. Hence, almost surely, no two
nonzero summands have a common jump time. In particular,
$\mathcal K_q(\mathcal E_a^{\mathrm{fr}})$ is almost surely continuous
at each fixed $h>0$.
Applying the
continuous mapping theorem to \eqref{eq:poi-N-cov-2} at $a$, and
comparing the resulting convergence with \eqref{eq:poi-N-cov-2} at
$a+h$, gives  \eqref{eq:consistency-2}.

It follows from \eqref{eq:consistency-2} that, for integers $k, \ell\geq1$,
$ (  \tau_{k}\mathcal K_q(\mathcal E_{-\ell-k}^{\mathrm{fr}})  , Z  )
 \overset{\mathrm{law}}{=}
 (  \mathcal K_q(\mathcal E_{-\ell}^{\mathrm{fr}}) , Z) .$
Thus, these joint laws are consistent under translations and determine
a $D(\mathbb R,\mathbb Z_{\geq0})$-valued process $N_q$, jointly defined
with $Z$, such that 
\begin{equation}\label{eq-consistency-2}
   \bigl(   \tau_{-\ell} N_q ,Z \bigr)
   \overset{\mathrm{law}}{=}
   \bigl( 
     \mathcal K_q(\mathcal E_{-\ell}^{\mathrm{fr}}) ,Z
    \bigr)  \ ,
   \quad \forall\, \ell\geq1.
\end{equation} 
and   $N_q$ has no fixed jump times. By the continuous mapping theorem again, \eqref{eq:poi-N-cov-2} implies
\eqref{eq:frozen-exact-q-process-limit}. 

To prove  \eqref{eq:frozen-exact-q-limit-construction}, fix $a\in\mathbb R$. Since
$
 \left(\tau_aN_q,Z\right)
 \overset{\mathrm{law}}{=}
 \left(\mathcal K_q(\mathcal E_a^{\mathrm{fr}}),Z\right)$ by \eqref{eq-consistency-2},
we may therefore choose a joint realization of
$N_q$, $\mathcal E_a^{\mathrm{fr}}$, $Z$, and the copies defining
$\mathcal K_q$, preserving their prescribed laws, such that
$
 \tau_aN_q
 =
 \mathcal K_q(\mathcal E_a^{\mathrm{fr}})$ almost surely. 
This proves \eqref{eq:frozen-exact-q-limit-construction}.

\smallskip\noindent
\underline{\textit{Step 5.}}
It remains to prove the existence of both upward and downward jumps. Use
the coupling from Step~4. By \eqref{eq:cluster-mass-representation},
the largest atom $\lambda_1(a)$ of $\mathcal E_a^{\mathrm{fr}}$ belongs
to $(0,\infty)$ almost surely.  
Let $F_a = \{  \mathsf N^{(1)}_{\lambda_1(a)}=q+1 \}$.  
 Then $\P(F_a \mid \mathcal{E}^{\mathrm{fr}}_a ) = e^{-\lambda_1(a)}  \frac{\lambda_1(a)^{q+1}}{(q+1)!} >0$. 
Moreover, on $F_a$, $K^{(1)}_{q,\lambda_1(a)}(0)=0$, as
$\Pi^{(1)}(0)|_{[q+1]}$ consists of a single block of size $q+1$.  
The rate at which $\Pi^{(1)}|_{[q+1]}$ splits into blocks of sizes $q$
and $1$ is  
\begin{equation}
   (q+1)c+(q+1)
   \int_{\mathcal S^\downarrow}\sum_{j\geq1}s_j^q(1-s_j)\,
   \nu(\dif\mathbf s)
   =(q+1)\bigl(\kappa(q+1)-\kappa(q)\bigr)>0.
\end{equation} 
Here the positivity follows from \eqref{eq:entropy-dominance}:  As noted after \eqref{eq:def-a*}, \eqref{eq:entropy-dominance} implies
$
 \nu\bigl(\{\mathbf s:s_2>0\}\bigr)>0$.
which implies  $ \int_{\mathcal S^\downarrow} s_2^q(1-s_2) \nu(\dif\mathbf s)>0$.

Since the total jump rate of $\Pi^{(1)}|_{[q+1]}$ is $\kappa(q+1)$,
conditional on $F_a$, the probability that
$T^{(1)}_{[q+1]}\in\dif s$ and
$\Pi^{(1)}(T^{(1)}_{[q+1]})|_{[q+1]}$ consists of two blocks
of sizes $q$ and $1$ is $(q+1) [\kappa(q+1)-\kappa(q)] e^{-\kappa(q+1) s} \dif s$. 
On this event,
$K^{(1)}_{q,\lambda_1(a)}$ jumps from $0$ to $1$ at  
$\kappa_*T^{(1)}_{[q+1]}$. 
We further require $T^{(1)}_{[q+1]}<L$, and   the newly created size-$q$ block to split before
time $L$. By the fragmentation property, its lifetime
after birth is independent of the past and exponential with rate
$\kappa(q)>0$. On this stronger event,
$K^{(1)}_{q,\lambda_1(a)}$ also jumps from $1$ to $0$ before time $\kappa_* L$.
By the no-common-jump property established in Step~4 below
\eqref{eq:consistency-2}, no other summand of
$\mathcal K_q(\mathcal E_a^{\mathrm{fr}})$ jumps at either time.
Consequently, the probability that $N_q$ has both an upward and a
downward jump in $(a,a+\kappa_* L)$ is bounded below by 
\begin{equation}
   \E\left[e^{-\lambda_1(a)}
            \frac{\lambda_1(a)^{q+1}}{(q+1)!}\right]
   \int_0^{L}
   (q+1)\bigl(\kappa(q+1)-\kappa(q)\bigr)
   e^{-\kappa(q+1)s}
   \bigl[   1- e^{-\kappa(q)(L-s)}\bigr]  \dif s
   >0.
\end{equation} 
This proves the asserted nonmonotonicity.
\end{proof}

\addtocontents{toc}{\protect\setcounter{tocdepth}{1}}
\appendix
\section{Proofs of the preliminary inputs}
\label{sec:preliminary-input-proofs}

We prove the preliminary results from Section~\ref{sec:preliminaries}.  We  derive the martingale
and extremal estimates from branching random walk results by sampling
the fragmentation at discrete times and passing to continuous time. Finally, we prove
Lemma~\ref{lem:pp-to-count} on convergence of tail-count processes.
 
\subsection{Branching random walk estimates}
\label{sec:brw-proofs} 
The corresponding continuous-time statements in the literature do not
quite cover the present level of generality.  In particular, 
\cite[Theorem~3.1]{KyprianouMadaule} imposes additional restrictions on
the dislocation measure, and \cite[Theorem~2.2]{KyprianouLaneMorters}
works with conservative interval fragmentations.  Adapting these results
to the present setting would require additional arguments.  We instead use
the time-discretization approach in  
\cite{BertoinRouault}. 

\smallskip
\noindent\underline{The associated BRW.}
Fix $h>0$.  The interval fragmentation sampled at times $nh$, $n\geq0$, defines a rooted tree
$\mathbb T_h$: the components of $\mathcal I(nh)$ form generation $n$,
and a component of $\mathcal I((n+1)h)$ is a child of the unique
component of $\mathcal I(nh)$ that contains it.  If $u$ in generation
$n$ of $\mathbb{T}_h$, write $|u|_{h}=n$ and set 
\[
 \widetilde{X}_h(u):=|I_u| \ , \quad
 \widetilde{V}_h(u):= \log\frac1{\widetilde X_h(u)} .
\]
The fragmentation property shows that $(\widetilde{V}_h(u): u \in \mathbb T_h)$ is a discrete-time
BRW with offspring displacement point process
$\sum_{i:X_i(h)>0}\delta_{ \log \frac{1}{X_i(h)} }$ 
and cumulant  
\begin{equation}\label{eq:brw:generating-function}
   \Phi_{h}(\theta): = -\log \E  \, \sum_{ |u|_{h}=1}  e^{-\theta \widetilde{V}_h(u)}     =-\log \E \, \sum_{ i \ge 1}   X_{i}(h)  ^{\theta}  =   h \, \kappa(\theta) . 
\end{equation}  
The additive martingale with parameter $\theta$ of the BRW $(\widetilde{V}_{h},\mathbb{T}_h)$ is then 
\begin{equation}
  \sum_{ |u|_{h}=n}  e^{-\theta \widetilde{V}_h(u)+ n \Phi_{h}(\theta)} = e^{h n\kappa(\theta) } \sum_{i:X_{i}(nh)>0} X_i(nh)^{\theta} = W_\theta(nh).
\end{equation}  

\smallskip
\noindent\underline{Nonextinction and supercriticality of the BRW skeleton.}
To apply the results of BRW, we need a preliminary property that  $\mathbb{T}_h$ is super-critical. We claim that $\mathbb{T}_h$ has no leaf  
\begin{equation}\label{eq:positive-genealogy}
 \P( X_1(h)>0 )=1 \, , \text { for every } h > 0. 
\end{equation} 
Moreover, $\mathbb{T}_h$ is not a line:
\begin{equation}\label{eq:supercritical}
 \P(X_2(h)>0)>0, \text { for every } h > 0. 
\end{equation}

Indeed, let
$\tau:=\inf\{t\geq0:X_1(t)\leq1/2\}$.
By \cite[Proposition~4.1]{BerestyckiRanked}, there exists a
subordinator $\xi$ with drift $c$ and Lévy measure
$\nu\circ\bigl(\mathbf s\mapsto\log(1/s_1)\bigr)^{-1}$
such that
$X_1(t)=e^{-\xi_t},  0\leq t< \tau.$
The assumptions \eqref{eq:dislocation-integrability} and
\eqref{eq:no-sudden-extinction} imply that $\xi$ has no killing and
is finite at every finite time. Consequently, $\tau>0$ almost surely
and, on $\{\tau<\infty\}$, $X_1(\tau)>0$. 
At time $\tau$, select the largest fragment and repeat the construction
inside that fragment. By the strong Markov property and
homogeneity, the successive durations $\tau_1,\tau_2,\ldots$ are
independent copies of $\tau$. Since $\tau>0$ almost surely,
$ \sum_{j\geq1}\tau_j=\infty$ almost surely.
Thus only finitely many selections are made before any fixed time $h$, and
the selected lineage has positive mass throughout.

Now we verify \eqref{eq:supercritical}.  By
\eqref{eq:entropy-dominance}, we have  $
 \nu(\{\mathbf s:s_2>0\})>0$.
Choose $\varepsilon>0$ such that
$
 0<\lambda_\varepsilon
 :=\nu(\{\mathbf s:s_2\geq\varepsilon\})$.  
Let
$
 \tau' :=\inf\{t\geq0: X_2(t)>0 \}$. 
 By the Poisson construction, we have  for every $h>0$,
$
 \P(\tau' \leq h/2)
 \geq1-e^{-\lambda_\varepsilon h/2}$.
On $\{\tau' \leq h/2\}$, two positive-mass fragments are present at some
time before $h$.  By \eqref{eq:positive-genealogy} and the strong
Markov property, each of them has a positive-mass descendant at
time $h$.  Hence
$\{\tau' \leq h/2\}\subset \{X_2(h)>0\}$, and \eqref{eq:supercritical} follows.  

\smallskip
\noindent\underline{Boundary case BRW.} For each $u \in \mathbb{T}_h$ define 
\[ V^{\mathrm{bd}}_h(u) = \theta_* \widetilde{V}_h (u)-nh\kappa_* \   \text{ if } \quad  |u|_{h}=n .\] 
Differentiating \eqref{eq:exponent-S} twice at $\theta_*$ and using
$\theta_*\kappa'(\theta_*)=\kappa_*$ gives
\begin{equation}\label{eq:boundary-skeleton-data}
 \E\sum_{|u|_{h}=1} e^{-V^{\mathrm{bd}}_h(u)}=1,
 \ 
 \E\sum_{|u|_{h}=1} V^{\mathrm{bd}}_h(u)e^{-V^{\mathrm{bd}}_h(u)}=0,
 \ 
 \E\sum_{|u|_{h}=1} V^{\mathrm{bd}}_h(u)^2e^{-V^{\mathrm{bd}}_h(u)}
 =h\theta_*^2\sigma_*^2.
\end{equation}
Furthermore, using $\sum_iX_i(h)\leq1$ and
$x^{\theta_*-1}\log(1/x)\leq1/[e(\theta_*-1)]$, we have 
\begin{equation}\label{eq:boundary-skeleton-bounds}
 \sum_{|u|_{h}=1} e^{-V^{\mathrm{bd}}_h(u) }\leq e^{h\kappa_*} \ , \
\sum_{|u|_{h}=1} V^{\mathrm{bd}}_h(u) ^+ e^{-V^{\mathrm{bd}}_h(u) } 
 \leq\frac{\theta_*e^{h\kappa_*}}{e(\theta_*-1)}.
\end{equation}
Thus the $h$-skeleton is a supercritical boundary BRW, survives almost
surely, and satisfies conditions \emph{(1.4)}--\emph{(1.5)} of
\cite{AidekonShi} and condition \emph{(1.5)} of \cite{MadauleTip}.

 \begin{proof}[Proof of Lemma~\ref{lem:brw-inputs}
  \textup{(\ref{input:derivative})}]
The martingale property follows by differentiating the conditional
martingale identity for $W_\theta(t)$ at $\theta=\theta_*$.

Although \cite[Proposition~7]{BertoinRouault} is stated for conservative
fragmentations, its argument extends to the present setting, since dissipated mass can be
viewed as particles sent to $+\infty$. Hence $Z(t)$ converges almost surely to a nontrivial nonnegative limit $Z$;
\eqref{eq:positive-genealogy} gives $Z>0$ almost surely. 
Alternatively, one may apply directly
\cite[Theorem~2.3]{MalleinShiDerivative}, whose condition $\mathrm{(H^*)}$
follows from
\[
  \sum_i s_i^{\theta_*}\le 1,
  \qquad
  \theta_*\sum_i s_i^{\theta_*}\log(1/s_i)
  \le \frac{\theta_*}{e(\theta_*-1)},
\]
and $\theta_*\kappa_\nu'(\theta_*)<\infty$. 
\end{proof}

\begin{proof}[Proof of Lemma~\ref{lem:brw-inputs} (\ref{input:SH})]   We  apply
\cite[Theorem~1.1]{AidekonShi} to the $h$-skeleton $(V_h^{\mathrm{bd}},\mathbb T_h)$, which   is a supercritical boundary BRW, survives almost
surely, and satisfies conditions  (1.4) and (1.5)  of
\cite{AidekonShi} by \eqref{eq:boundary-skeleton-data} \eqref{eq:boundary-skeleton-bounds}. 
With the variance in
\eqref{eq:boundary-skeleton-data}, this gives
\begin{equation}\label{eq:critical-skeleton-SH}
 \sqrt{n}\,W_{\theta_*}(n h)
 \xrightarrow[n\to\infty]{\P}
 \left(\frac{2}{\pi h\theta_*^2\sigma_*^2}\right)^{1/2}Z.
\end{equation}
For $s\leq t\leq s+h$, since  $S_{\theta_*}(\cdot)$ is decreasing, we have $ 
 e^{-h\kappa_*}W_{\theta_*}(s+h)
 \leq W_{\theta_*}(t)
 \leq e^{h\kappa_*}W_{\theta_*}(s)$.  
Apply \eqref{eq:critical-skeleton-SH} with
$s=h\lfloor t/h\rfloor$, first let $t\to\infty$, and then let
$h\downarrow0$.  This yields \eqref{eq:SH-scaling}. 
\end{proof}

\begin{proof}[Proof of Lemma~\ref{lem:brw-inputs} (\ref{input:MinimumBRW})] Since A\"id\'ekon's  result
\cite{A_d_kon_2013} assumes nonlattice displacements in order to show $O(1)$-fluctuation and identify the limiting distribution of the BRW minimum (see also A\"id\'ekon--Shi~\cite[Remark after Theorem~6.1]{AidekonShi}),
we instead use the
lattice-insensitive estimate of Madaule~\cite[Theorem~2.3]{MadauleTip}  
 For
$\beta>1$, put
\[
 P_\beta(n):= \sum_{|u|_{1}=n} e^{- \beta V^{\mathrm{bd}}_1(u) }  = \sum_{i \ge 1} e^{-\beta V_i(n)}
\]
Theorem~2.3 of \cite{MadauleTip} gives a random variable $P_\beta \in (0,\infty)$ with   $1/\beta$-stable distribution ($\E[e^{-\lambda P_\beta}]
=\exp\{-c_\beta\lambda^{1/\beta}\}$), such that  
$
 n^{3\beta/2}Z^{-\beta}P_\beta(n)
 \xrightarrow[n\to\infty]{\mathrm{law}}P_\beta$.
Together with the fact $Z>0$, this implies
\[
 \log P_\beta(n)=-\frac{3\beta}{2}\log n+O_\P(1),
\]
  Since
$
 e^{-\beta V_1(n)}\leq P_\beta(n)
 \leq e^{-(\beta-1)V_1(n)}W_{\theta_*}(n)$,  and 
 $ \log W_{\theta_*}(n)=-\frac12\log n+O_\P(1)$ by Lemma~\ref{lem:brw-inputs} (\ref{input:SH}), 
it follows that
\[
 \frac32+o_\P(1)
 \leq\frac{V_1(n)}{\log n}
 \leq\frac{3\beta-1}{2(\beta-1)}+o_\P(1).
\]
Letting $\beta\to\infty$ yields
$V_1(n)/\log n\to3/2$ in probability. 
For $n\leq t\leq n+1$, monotonicity of $X_1(\cdot)$ gives
$
 V_1(n)-\kappa_*\leq V_1(t) \leq V_1(n+1)+\kappa_*$.
Thus $V_1(t)/\log t\to3/2$ in probability.  Rearranging proves
\eqref{eq:largest-fragment-input}.
\end{proof}

\begin{proof}[Proof of Lemma~\ref{lem:brw-inputs} (\ref{input:Madaule})]
Assume \eqref{non-lattice-cond}.  

\smallskip
\noindent\underline{\textit{Step 1.}}
We  show that, for any $h>0$,
the offspring point process
$  L_h:=\sum_{|u|_h=1}\delta_{V_h^{\mathrm{bd}}(u)}$
is nonlattice. 

Let $Y_h$ have its $e^{-x}$-size-biased distribution,
that is,
$\E[f(Y_h)]:=\E  [\int_{\mathbb R}e^{-x}f(x)\,L_h(\dif x) ].$
The first identity in \eqref{eq:boundary-skeleton-data} ensures that this
is a probability distribution.  If $(\xi_t)$ denotes the logarithmic size
of a tagged fragment  \cite[Section~3.2.2]{Bertoin}, then  
\begin{equation}\label{eq:size-biased-tagged-fragment}
 \E[f(Y_h)]
 =e^{h\kappa_*}\E \Bigl[  
 e^{-(\theta_*-1)\xi_h}
 f(\theta_*\xi_h-h\kappa_*) \ind{\xi_h<\infty}
  \Bigr].
\end{equation} 
By \cite[Theorem~3.2]{Bertoin}, $\xi$ is a (possibly killed) subordinator
with L\'evy measure given by 
$
 \Pi_{\xi}(\dif y)
 =\int_{\mathcal S^\downarrow}
   \sum_{i:s_i>0}s_i\,
   \delta_{\log(1/s_i)}(\dif y)\,\nu(\dif\mathbf s)$. 
Condition~\eqref{non-lattice-cond} says that $\Pi_\xi$ is not
supported on any discrete additive subgroup $d\mathbb Z$. Thus 
by the Lévy--Khintchine formula,  for any real number $\lambda \neq 0$, 
\[
 \bigl| \E[ e^{\mathrm i\lambda\xi_h} \mid \xi_{h} < \infty ] \bigr|
 =
 \exp \Bigl\{  
 h\int_{(0,\infty)}
 \bigl(\cos(\lambda y)-1\bigr)\Pi_\xi(\dif y)
 \Bigr\}  <1.  
\]
Hence, conditional on $\{\xi_h<\infty\}$,  
$ \xi_h $ has a nonlattice distribution. From  \eqref{eq:size-biased-tagged-fragment} it follows that $Y_h$ is
nonlattice too. This implies $L_h$ must be nonlattice: if $L_h$ was supported on $a+d\mathbb Z$ for some
$a\in\mathbb R$ and $d>0$ almost surely, then
$
 \P(Y_h\notin a+d\mathbb Z)
 =\E [\int_{\mathbb R\setminus(a+d\mathbb Z)}
 e^{-x}L_h(\dif x) ]=0$, 
a contradiction.

\smallskip
\noindent\underline{\textit{Step 2.}}
For a point measure $\mu$, let $\vartheta_a\mu(B):=\mu(B-a)$. Since $Z>0$ almost surely, we may set
\[
 \widehat{\mathcal E}_{t}
 :=\vartheta_{-\log Z} \mathcal E_t^V 
 =\sum_{i:X_i(t)>0}
   \delta_{\frac32\log t-V_i(t)-\log Z}.
\] 
Fix $h=1/m$ for $m \in \mathbb{Z}_{\ge 1}$. Then $ \widehat{\mathcal E}_{nh}
 =   \sum_{|u|_{h}=n}
   \delta_{\frac32\log (nh)-V_h^{\mathrm{bd}}(u)-\log Z} $ is, 
after reflection about the origin and translation by
$\frac32\log h$,  
the extremal process of the BRW $(V_h^{\mathrm{bd}}, \mathbb{T}_h)$ considered in \cite{MadauleTip}.  
By \eqref{eq:positive-genealogy}--\eqref{eq:boundary-skeleton-bounds}, and 
the required nonlattice assumption was verified in Step 1,    
\cite[Theorem~1.1]{MadauleTip} yields  a point process  $\widehat{\mathcal E} \sim \DPPP(C_V e^{-x} \dif x, \mathcal{D} )$,  independent of $Z$, whose  decoration
$\mathcal D$ is supported on $(-\infty,0]$ and has maximum zero almost
surely,    
such that 
\begin{equation} 
  \label{eq:conv-1}
\bigl(  \widehat{\mathcal E}_{n h} ,Z(nh) 
    \bigr)
  \xrightarrow[n \to \infty]{\mathrm{law}} 
 \bigl( \widehat{\mathcal E}  ,Z  \bigr) \quad  \text{ in } \ \Mloc(\R) \times \R . 
\end{equation} 
We emphasize that the law of $\widehat{\mathcal E}$  do not depend on $h=1/m$.  Indeed, taking $n=m\ell$ in the
$1/m$-skeleton gives exactly the same process at time $\ell$ as the
unit-time skeleton.  

Let $\chi_{K}\in C_c(\mathbb R)$ take values in $[0,1]$, equal to one on
$[-K,K]$, and vanish outside $[-K-1,K+1]$. Then for each $\beta \in \mathbb{R}$, and $g \in  \mathrm{BUC}(\mathbb{R})$, we have  $\chi_{K} g\exp_{\beta}\in C_c(\mathbb{R})$. The continuity mapping theorem gives 
\[ \bigl(  \widehat{\mathcal E}_{n h} ,Z(n h) ,  \langle \chi_K  g \exp_{\beta} ,  \widehat{\mathcal E}_{n h} \rangle 
    \bigr)
  \xrightarrow[n \to \infty]{\mathrm{law}} 
 \bigl( \widehat{\mathcal E} ,Z ,  \langle\chi_K g \exp_{\beta}  \,,\widehat{\mathcal E} \rangle \bigr) \quad \text{ in } \ \Mloc(\R) \times \R^2 .  \] 
For any $\beta>1$ and $h$, by \cite[Theorem~2.3]{MadauleTip},  the sequence  
 $( \langle     \exp_{\beta} ,  \widehat{\mathcal E}_{n h} \rangle  )_{n \ge 1}$ is tight in $\R$. 
This prevents atoms of $\widehat{\mathcal E}_{n h}$  from escaping to
$+\infty$, so the convergence in the first coordinate also holds in
$\Mloc((-\infty,\infty])$. 
 Moreover,  since   $ \langle  (1- \chi_K) \exp_{\beta}, \mu\rangle 
 \leq e^{- \frac{\beta-1}{2} K }  \langle \exp_{\frac{\beta+1}{2}}  , \mu\rangle + e^{- \beta K }  \langle \exp_{2 \beta}  , \mu\rangle  $ for every measure $\mu$,  we get $\langle  (1-\chi_K )  g \exp_{\beta} ,  \widehat{\mathcal E}_{n h} \rangle$ converges to zero in probability as $n \to \infty$ first, then $K \to \infty$. Combining this with \eqref{eq:conv-1}, 
 we thus obtain  
\begin{equation}\label{eq:skeleton-extremal-convergence}
 \Bigl( \widehat{\mathcal E}_{n h} ,Z(n h) ,  (\langle   f_j ,  \widehat{\mathcal E}_{n h} \rangle )_{j=1}^{k}  \Bigr)  
 \xrightarrow[n\to\infty]{\mathrm{law}}
 \Bigl(\widehat{\mathcal E},Z,
   (\langle   f_j ,  \widehat{\mathcal E}  \rangle )_{j=1}^{k}   \Bigr)  \quad \text{ in } \ \Mloc((-\infty,\infty]) \times \R^{k+1} . 
\end{equation} 
  Consequently,  set $
 \mathcal E^V:=\vartheta_{\log Z}\widehat{\mathcal E} $, and since $ \mathcal E_t^V = \vartheta_{\log Z}\widehat{\mathcal E}_{t}$,   we get 
 \begin{equation} 
\qquad \Bigl(  \mathcal{E}^{V}_{\frac{n}{m}} ,Z(\tfrac{n}{m}),
 ( \langle f_j, \mathcal{E}^{V}_{\frac{n}{m}} \rangle )_{j=1}^k  \Bigr)
  \xrightarrow[n \to \infty]{\mathrm{law}} 
 \Bigl(  \mathcal{E}^{V},Z,
( \langle f_j \, ,   \mathcal{E}^{V}  \rangle )_{j=1}^k  \Bigr)  \text{ in } \Mloc((-\infty,\infty])\times\mathbb R^{k+1},
\end{equation} 
with $\langle \exp_{\beta}, \mathcal{E}^{V} \rangle<\infty$ almost surely for $\beta> 1$.  
This implies $\langle\exp_\beta,\mathcal D\rangle<\infty$ almost surely. Otherwise the decorated Poisson representation would then contain, with positive
probability, a cluster having infinite $\exp_\beta$-mass, contradicting
$\langle\exp_\beta,\mathcal E^V\rangle<\infty$ almost surely.

 \smallskip
\noindent\underline{\textit{Step 3.}}  Set
$
 \ell=\ell_m(t):=m^{-1}\lfloor mt\rfloor$ and $s:=t-\ell\in[0,1/m)$.
Fix $\beta>1$ and $g\in\mathrm{BUC}(\mathbb R)$, and put
$f:=g\exp_\beta$.  
We claim that \eqref{eq:weighted-extremal-convergence} follows once we
prove that, for every $\eta>0$,
\begin{equation}\label{eq:dis-2-con}
 \lim_{m\to\infty}\limsup_{t\to\infty}
 \P \bigl(  
 \bigl|  \langle f,\mathcal E_t^V\rangle
 -\langle f,\mathcal E_{\ell_m(t)}^V\rangle\bigr| >\eta
 \bigr) =0.
\end{equation} 
Indeed, let $d_{\mathrm{vag}}$ be a metric for vague convergence on
$\Mloc((-\infty,\infty])$ generated by a countable family $(\varphi_i)_{i\geq1}$
in $C_c((-\infty,\infty])$. Applying \eqref{eq:dis-2-con} $\varphi_i$, since  
$\exp_{-\beta}\cdot\varphi_i|_{\mathbb R}\in\mathrm{BUC}(\mathbb R)$,
it follows that
$d_{\mathrm{vag}}(\mathcal E_t^V,\mathcal E_{\ell_m(t)}^V)$
converges to zero in probability as first $t\to\infty$ and then
$m\to\infty$. 
 Applying \eqref{eq:dis-2-con} to
$f_j=g_j\exp_{\beta_j}$, $1\leq j\leq k$, similarly gives  
$\sum_{j=1}^k|\langle f_j,\mathcal E_t^V\rangle
-\langle f_j,\mathcal E_{\ell_m(t)}^V\rangle| \to 0$ in probability in the same iterated limit. 
Moreover,  we have 
$\sum_{j=1}^k|\langle f_{j,t}-f_j,\mathcal E_t^V\rangle|
\leq\sum_{j=1}^k\|g_{j,t}-g_j\|_\infty
\langle\exp_{\beta_j},\mathcal E_t^V\rangle
\xrightarrow[t\to\infty]{\P}0$, since  
\eqref{eq:dis-2-con}  and  the tightness of
$\langle\exp_{\beta_j},\mathcal E_{\ell_m(t)}^V\rangle$ implies the   tightness of$(\langle\exp_{\beta_j},\mathcal E_t^V\rangle)_{t>0}$.
Consequently,
the desired result  follows from the convergence along the skeleton
and Slutsky's lemma. 

To prove \eqref{eq:dis-2-con}, we write 
 $\mathcal E_\ell^V=\sum_{k \ge 1}\delta_{y_k}$.  The fragmentation
property gives,  
\begin{equation}\label{eq:short-time-decomposition}
 \mathcal E_t^V
 =  \sum_{k \ge 1} \sum_{j \ge 1} \delta_{y_k-V_j^{(k)}(s) + \epsilon_{\ell,s}} \ , \quad 
 \epsilon_{\ell,s}:=\frac32\log\frac{\ell+s}{\ell},
\end{equation}
where the families $(V_j^{(k)}(s))_{j\ge 1}$ are independent copies of
$(V_j(s))_j$ and are independent of $\mathcal F^{\mathcal{I}}_\ell= \sigma (\mathcal{I}(s): 0 \le s \le t)$. Then from \eqref{eq:short-time-decomposition}, conditionally on $\mathcal F^{\mathcal{I}}_{\ell}$, we get 
\begin{equation}
   \E\bigl[|\langle f,\mathcal E_t^V\rangle
                -\langle f,\mathcal E_\ell^V\rangle|
            \mid\mathcal F^{\mathcal{I}}_{\ell}\bigr]
   \leq
   \bigl[e^{\beta\epsilon_{\ell,s}}\rho_{f,\beta}(1/m)
  +\omega_{f,\beta}(\epsilon_{\ell,s})\bigr]
   \langle\exp_\beta,\mathcal E_{\ell_m(t)}^V\rangle. \label{eq:cond-bnd-1}
\end{equation}
 
where $
 \omega_{f,\beta}(\epsilon)
 :=\sup_{y\in\mathbb R,\,|b|\leq\epsilon}
 e^{-\beta y}|f(y+b)-f(y)|$ and 
\begin{equation}
 \rho_{f,\beta}(\epsilon)
 :=\sup_{ y\in\mathbb R, 0\leq s\leq \epsilon }
 e^{-\beta y}
 \E \Bigl[   \Bigl|f(y) - \sum_{j \ge 1} f(y-V_j(s)) \Bigr| \Bigr]  \label{eq:one-parent-modulus} 
\end{equation}  

For fixed $m$, the quantity $\epsilon_{\ell,s}$ tends to zero uniformly in
$s\in [0,1/m)$ as $t\to\infty$. Note that 
\[
 \omega_{f,\beta}(\epsilon)
 \leq
 e^{\beta\epsilon}
 \sup_{y\in\mathbb R,\,|b|\leq\epsilon}|g(y+b)-g(y)|
 +(e^{\beta\epsilon}-1)\|g\|_\infty
 \xrightarrow{\epsilon\downarrow0}0.
\]  
By the convergence of $\mathcal{E}^{V}_{t}$ along each skeleton $(\frac{n}{m})_{n \ge 1}$, whose limit does not depend on
$m$, we get
\[
 \sup_{m\geq1}\limsup_{t\to\infty}
 \P \bigl(   
 \langle\exp_\beta,\mathcal E_{\ell_m(t)}^V\rangle>K
  \bigr) \le   
 \P \bigl(    
 \langle\exp_\beta,\mathcal E^V\rangle \ge K
  \bigr)  \xrightarrow{K \to \infty} 0
\]
Together with the preceding estimate \eqref{eq:cond-bnd-1} and Markov's
inequality, this uniform tightness shows that
\eqref{eq:dis-2-con} follows once we prove
\begin{equation}\label{eq:small-time-test-function}
 \rho_{f,\beta}(\epsilon)\xrightarrow{\epsilon\downarrow0}0.
\end{equation} 
\smallskip
\noindent\underline{\textit{Step 4.}}
We first record two consequences of stochastic continuity at time zero.
For every $\beta>1$, we have 
$
 \E[\sum_j e^{-\beta V_j(s)}]
 =\exp\{s(\beta\kappa_*-\kappa(\beta\theta_*))\}$, 
which converges to one uniformly for $0\leq s\leq \epsilon$ as $\epsilon\downarrow0$.
Moreover, monotonicity of $X_1$ gives
 \[ 
 \sup_{0\leq s\leq \epsilon}|V_1(s)|
 \leq \epsilon|\kappa_*|+\theta_*\log\frac1{X_1(\epsilon)}
 \xrightarrow[\epsilon\downarrow0]{\P}0 .\]
Since $e^{-\beta V_1(s)}\leq e^{\beta \epsilon\kappa_*^+}$, we also have
$e^{-\beta V_1(s)}\to1$ in $L^1$, uniformly in $s\leq \epsilon$.
Subtracting the largest-child
term from the preceding expectation yields
\[
 \sup_{0\leq s\leq \epsilon}
 \E[\sum_{j\geq2}e^{-\beta V_j(s)}]  \xrightarrow{\epsilon\downarrow0} 0 .
\]

Now we prove \eqref{eq:small-time-test-function}. For every $\tilde{\epsilon}>0$, $\sup_{ y\in\mathbb R, 0\leq s\leq \epsilon } e^{-\beta y}
 \E  [   |f(y) -   f(y-V_1(s))  |  ]$ is at  
most
\[
 \omega_{f,\beta}(\tilde{\epsilon})
 + \|g\|_{\infty} \sup_{0\leq s\leq \epsilon}
 \E \Bigl[  (1+e^{-\beta V_1(s)})
          \ind{|V_1(s)|>\tilde{\epsilon}} \Bigr] \lesssim  \omega_{f,\beta}(\tilde{\epsilon})+   \sup_{0\leq s\leq \epsilon} \P( |V_1(s)|>\tilde{\epsilon})
\]
 The remaining term 
contribute at most
$\|g\|_{\infty}\sup_{0\leq s\leq \epsilon}
 \E[\sum_{j\geq2}e^{-\beta V_j(s)}]$.  Letting $\epsilon \downarrow0$ first then $\tilde{\epsilon} \downarrow 0$ proves
\eqref{eq:small-time-test-function}, and completes the proof of Lemma~\ref{lem:brw-inputs}.
\end{proof}

\subsection{From point measures to tail-count processes}
\label{sec:pp-to-count-proof}

\begin{proof}[Proof of Lemma \ref{lem:pp-to-count}] 
Indeed,
fix $a<b$ such that $\mu(\{a,b\})=0$, and let
$a<z_1<\cdots<z_k<b$ be the atoms of $\mu$ in $(a,b)$.  Vague convergence
implies that, for all sufficiently large $n$, the atoms of $\mu_n$ in
$(a,b)$ can be written as $a<z_{n,1}<\cdots<z_{n,k}<b$, where
$z_{n,j}\to z_j$, and that
$\mu_n((b,+\infty])=\mu((b,+\infty])$. 
 Let $\lambda_n$ be the increasing
piecewise linear bijection of $[a,b]$ that fixes $a$ and $b$ and satisfies
$\lambda_n(z_j)=z_{n,j}$.  Then all sufficiently large $n$, we have 
\[
 \sup_{x\in[a,b]}|\lambda_n(x)-x|\longrightarrow0 \ , \quad 
 \mu_n((\lambda_n(x),+\infty])=\mu((x,+\infty]) , x \in [a,b]
\]
   Thus the tail-count
functions converge in $D([a,b],\mathbb Z_{\geq0})$ under the Skorokhod
$J_1$ topology.
 \end{proof}

\section*{Acknowledgement \& Statement on AI use.}
H.M. thanks Oren Louidor for his interest in this work and for several enlightening discussions.
H.M. is supported in part by a Lady Davis Fellowship at the Technion. 

\smallskip
\noindent\textbf{Statement on AI use.} The author formulated the research questions, and
outlined an initial proof framework, indicating the intended
approach for its main steps. 
OpenAI's   GPT-5.6 Sol  were used to discuss, test and refine
these ideas, supply the missing computations and technical arguments
needed to complete the proofs, identify potential gaps. Anthropic's Claude Opus 5 was used for
language editing and to generate TikZ for   figures.  
All
model-assisted material included in the manuscript was critically
reviewed, and all mathematical claims and citations were independently
verified by the author. The author   takes
full responsibility for the manuscript.

\bibliographystyle{amsplain}
\bibliography{references}

\end{document}